\documentclass{amsart}
\usepackage{epsfig}
\usepackage{amsmath}
\usepackage{amssymb}
\usepackage{amscd}
\usepackage{graphicx}
\usepackage{xcolor}
\usepackage{float}
\usepackage{verbatim}
\usepackage{bbm}
\usepackage{cancel}
\usepackage{enumerate}
\usepackage{hyperref}

\newtheorem{thm}{Theorem}[section]
\newtheorem{lem}[thm]{Lemma}
\newtheorem{prop}[thm]{Proposition}

\newtheorem{ques}[thm]{Question}
\newtheorem{cor}[thm]{Corollary}
\newtheorem{defn}[thm]{Definition}

\theoremstyle{remark}
\newtheorem{rem}[thm]{Remark}

\newtheorem{exam}[thm]{Example}

\def \N {\mathbb N}

\def \B {\mathcal B}

\def \D {\mathcal D}
\def \Z {\mathbb Z}
\def \R {\mathbb R}

\def \F {\mathcal F}

\def \Q {\mathbb Q}
\def \U {\mathcal U}
\def \P {\mathcal P}

\def \M {\mathcal M}

\def \htop {h_{\mathsf{top}}}
\def \xt {$(X,T)$}

\def \eps {\varepsilon}

\def \sq {sequence}

\def \ys {$(Y,S)$}
\def \tl {topological}
\def \im {invariant measure}
\def \inv {invariant}

\def \ds {dynamical system}

\def \htop{h_{\mathsf{top}}}

\def \pr1 {\hspace{1pt}+\hspace{-5pt}^\rightarrow\hspace{3pt}}
\def \pl1 {\hspace{3pt}^\leftarrow\hspace{-5pt}+\hspace{1pt}}
\def \pl{\looparrowleft}
\def \pr{\looparrowright}

\numberwithin{equation}{section}

\begin{document}

\title[Normality and Determinism]{On preservation of normality and determinism under arithmetic operations}

\author{Vitaly Bergelson and Tomasz Downarowicz}

\address{\vskip 2pt \hskip -12pt Tomasz Downarowicz}

\address{\hskip -12pt Faculty of Pure and Applied Mathematics, Wroc\l aw University of Technology, Wroc\l aw, Poland}

\email{downar@pwr.edu.pl}

\medskip
\address{\vskip 2pt \hskip -12pt Vitaly Bergelson}

\address{\hskip -12pt Department of Mathematics, Ohio State University, Columbus, OH 43210, USA}

\email{vitaly@math.ohio-state.edu}

\thanks{
This research is supported by National Science Center, Poland (Grant HARMONIA No. 2018/30/M/ST1/00061) and by the Wroc\l aw University of Science and Technology.
}

\subjclass[2010]{Primary 37B05, 37B20; Secondary 37A25}
\keywords{}

\begin{abstract}
In this paper we develop a general ergodic approach which reveals the underpinnings of the effect of arithmetic operations involving normal and deterministic numbers. This allows us to recast in new light and amplify the result of Rauzy, which states that a number $y$ is deterministic if and only if $x+y$ is normal for every normal number $x$. Our approach is based on the notions of lower and upper entropy of a point in a \tl\ \ds. The ergodic approach to Rauzy theorem naturally leads to the study of various aspects of normality and determinism in the general framework of dynamics of endomorphisms of compact metric groups. In particular, we generalize Rauzy theorem to ergodic toral endomorphisms. 
 Also, we show that the phenomena described by Rauzy do not occur when one replaces the base $2$ normality associated with the $(\frac12,\frac12)$-Bernoulli measure by the variant of normality associated with a $(p,1-p)$-Bernoulli measure, where $p\neq\frac12$.
 Finally, we present some rather nontrivial examples which show that Rauzy-type results are not valid when addition is replaced by multiplication. 
\end{abstract}

\maketitle

\tableofcontents

\section{Introduction}\label{S1}

Fix a natural number $r\ge2$. For any number $x\in\R$ consider its \emph{base $r$ expansion}
$$
x=\sum_{n=-m}^\infty \frac{x_n}{r^n},
$$ 
where $m\ge 0$ and, for each $n\ge -m$, $x_n\in\{0,1,\dots,r-1\}$.\footnote{Some rational numbers have two base $r$ expansions, in this case we choose the one that terminates with zeros.}

A number $x$ is \emph{normal in base $r$} if the \sq\ of digits in its expansion $\omega=(x_n)_{n\ge-m}\in\{0,1,\dots,r-1\}^\N$ is normal, meaning that for any $k\in\N$, every finite block of digits $w=w_1w_2\dots w_k$ appears in $\omega$ with the ``correct'' limiting frequency~$r^{-k}$. 

A property dual to normal is that of \emph{deterministic}. Precise definition of this property is quite intricate and will be given in Section~\ref{S3}. For now, let us just say that a number $x$ is deterministic in base $r$ if the appropriately defined \emph{epsilon-complexity} of $\omega$ grows subexponentially (see, e.g.,~\cite[Lemma 8.9]{W3} and \cite[Definition~1]{BV}). 

Let $\mathcal N(r)$ and $\D(r)$ denote the sets of real numbers normal and deterministic in base~$r$, respectively. A remarkable result of G. Rauzy \cite{Ra} states that if $x\in\mathcal N(r)$ and $y\in\D(r)$ then $x+y\in\mathcal N(r)$. Rauzy also proved the converse: if $y$ has the property that $x+y\in\mathcal N(r)$ for any $x\in\mathcal N(r)$ then $y\in\D(r)$. To summarise, a number $y$ is deterministic (in base $r$) if and only if the operation $x\mapsto x+y$ \emph{preserves normality} in base $r$. Also, one can derive from the results obtained in \cite{Ra} that if $x\in\D(r)$ and $y\in\D(r)$ then $x+y\in\D(r)$. As a matter of fact, the converse holds as well (see Corollary~\ref{rrr}(3) below): if $y$ has the property that $x+y\in\D(r)$ for any $x\in\D(r)$ then $y\in\D(r)$.

In this paper we develop a general ergodic approach to the study of the effect of arithmetic operations on normality and determinism. This allows us to recast in new light and amplify the work of Rauzy (for instance, our methods allow for an almost immediate generalization of Rauzy theorem to $\R^k$). Our approach is based on the notions of lower and upper entropy of a point in a \tl\ \ds. To recover Rauzy's results we work with the \ds\ $(\mathbb T,R)$, where $\mathbb T=\R/\Z$ is the 1-dimensional torus (circle) and $R$ is the map given by $t\mapsto rt$, $t\in\mathbb T$. The ergodic approach to Rauzy theorem naturally leads to the study of various aspects of normality and determinism in the general framework of dynamics of endomorphisms of compact metric groups. In particular, we generalize Rauzy theorem to ergodic toral endomorphisms. 
A more detailed discussion of the diverse applications of our ergodic approach is given in the description of the structure of the paper provided below.
\smallskip

Our paper also contains some elaborate constructions which indicate the limits to  possible extensions of the results obtained in this paper:
\begin{itemize}
 \item We show that the phenomena described by Rauzy do not occur when one replaces the base $2$ normality associated with the $(\frac12,\frac12)$-Bernoulli measure by the variant of normality associated with a $(p,1-p)$-Bernoulli measure, where $p\neq\frac12$.
 \item We present some rather nontrivial examples which show that Rauzy-type results are not valid when addition is replaced by multiplication. 
\end{itemize}

The structure of the paper is as follows.
In Section~\ref{S2} we introduce the basic notions of \tl\ dynamics such as \im s, factors, joinings, generic and quasi-generic points for an \im, and we interpret the notion of a normal number in dynamical terms. We also introduce the definition of a $p$-normal number.

In Section~\ref{S3} we introduce the notions of lower and upper entropies of a point in a \tl\ \ds. Also, in this section, we define deterministic numbers and discuss an equivalent definition given by Rauzy. Finally, we provide an interpretation of normality and determinism in terms of lower and upper entropies. 

In Section~\ref{Sn}, we prove our first main result, Proposition~\ref{x+y}, which deals with the behavior of lower and upper entropy under addition, and, as a corollary, we derive in terms of pure ergodic theory one direction of Rauzy's seminal characterization of deterministic numbers, namely that if $x\in\mathcal N(r)$ and $y\in\D(r)$ then $x+y\in\mathcal N(r)$. We show by examples that the bounds given in Proposition~\ref{x+y} are sharp. Next, in Proposition~\ref{xq} we show (again, by purely ergodic means) that for any number $x$ and any nonzero rational number $q$, $qx$ has the same lower and upper entropy as~$x$. This result is a refinement of an old result by D.D.\ Wall, which states that if $x$ is normal, so is any nonzero rational multiple of $x$. We conclude the section with a streamlined proof of the other direction of Rauzy theorem. 

In Section~\ref{S5} we utilize the results obtained in Section~\ref{Sn} to obtain a multidimensional version of Rauzy theorem.

In Section~\ref{S4} we deal with generalizations of Rauzy theorem in two directions. First, in Subsections~\ref{5.1} and ~\ref{fnor} we extend the framework to the more general context which involves averaging along an arbitrary F\o lner \sq\ $\F$ in $\N$. Next, in Subsection~\ref{sred} we define the notions of $\F$-normality and $\F$-determinism for actions of endomorphisms on compact metric groups, and, in this generality, we prove a version of Rauzy theorem for endomorphisms of some Abelian groups including ergodic toral endomorphisms. 

In Section~\ref{S6} we deal with $p$-normal numbers which were defined in Section~\ref{S2}, and we show that if $p\neq\frac12$ then for any $p$-normal number its sum with any deterministic number, as well as its product by any rational number, is never $p$-normal (nor $p'$-normal for any $p'$).

Finally, in Section~\ref{S7} we give a rather elaborate example of a normal (in base~2) number $x$ and two deterministic numbers $y$ and $z$ (the frequency of the digit 1 in the binary expansion of $y$ is zero while in $z$ it is positive) such that neither $xy$ nor $xz$ are normal. In fact, both these products are deterministic. The example allows us also to show that the products and squares of deterministic numbers need not be deterministic. We conclude the section with a series of open problems and some pertinent observations and remarks. 

Finally, in the Appendix we provide a proof of an important result by B.\ Weiss, which characterizes deterministic \sq s in terms of complexity (this result was stated without a proof in \cite[Lemma~7.9]{W3}).


\section{Background material}\label{S2}
Let $X$ be  a compact metrizable space and let $T:X\to X$ be a continuous transformation. The pair $(X,T)$ is called a \emph{\tl\ \ds} (or just a \emph{\ds}). Let $\M(X)$ denote the space of all Borel probability measures\footnote{By abuse of language, we will often say that $\mu$ is a ``measure on $X$'', meaning that $\mu\in\M(X)$.} on $X$, endowed with the (compact) topology of the weak* convergence. A measure $\mu\in\M(X)$ is called \emph{$T$-invariant} (or just \emph{invariant}), if $\mu(T^{-1}(A))=\mu(A)$ for any Borel set $A\subset X$. The collection $\M(X,T)\subset\M(X)$ of all $T$-\im s is convex and compact (see, e.g.,~\cite{W} for more details). If $\mu\in\M(X,T)$ then the triple $(X,\mu,T)$ will be called a \emph{measure-preserving system}.

Let \xt\ and \ys\ be \ds s and let a map $\phi:X\to Y$ be continuous, surjective and \emph{equivariant}, i.e., such that $\phi\circ T=S\circ\phi$. In this case we say that $\phi$ is a \emph{factor map} from the system \xt\ to the system \ys. For brevity, we will write $\phi:(X,T)\to(Y,S)$.
The system \ys\ is called a \emph{factor} of \xt\ and \xt\ is called an \emph{extension} of \ys. Note that $\phi$ induces a natural map $\phi^*$ from $\M(X,T)$ onto $\M(Y,S)$
given by
\begin{equation}\label{star}
\phi^*(\mu)(A)=\mu(\phi^{-1}(A)),
\end{equation}
where $A$ is a Borel subset of $Y$. The measure-preserving system $(Y,\phi^*(\mu),S)$ is referred to as a \emph{continuous factor} of the measure-preserving system $(X,\mu,T)$ (via~$\phi$). 

Measure-preserving systems $(X,\mu,T)$ and $(Y,\nu,S)$ are \emph{isomorphic} if there exists an equivariant Borel-measurable (not necessarily continuous) map $\phi:X\to Y$ defined and invertible $\mu$-almost everywhere and such that $\phi^*(\mu)=\nu$.

If a factor map $\phi$ from $(X,T)$ to $(Y,S)$ is invertible, then it is a homeomorphism and it is called a \emph{\tl\ conjugacy}. 

\begin{rem}\label{conj}
Note that if $(X,T)$ and $(Y,S)$ are \tl ly conjugate then the map $\phi^*$ is a homeomorphism between $\M(X,T)$ and $\M(Y,S)$ and for each \im\ $\mu\in\M(X,T)$, $\phi$ is an isomorphism between $(X,\mu,T)$ and $(Y,\phi^*(\mu),S)$.
\end{rem}

A \ds\ which plays an important role in the study of normality is the \emph{symbolic system on $r$ symbols}, $(\{0,1,\dots,r-1\}^\N,\sigma)$, where the shift map $\sigma$ is given by 
$$
\sigma((a_n)_{n\ge1})=(a_{n+1})_{n\ge1}, \ (a_n)_{n\ge1}\in\{0,1,\dots,r-1\}^\N.
$$

We now introduce some terminology associated with symbolic systems. By a \emph{block} we will understand any finite \sq\ $B=(b_1,b_2,\dots b_k)$, $k\in\N$, of elements of the \emph{alphabet} $\{0,1,\dots,r-1\}$. The number of elements of $B$ is called the \emph{length} of $B$ and is denoted by $|B|$. 
We will find it convenient to denote the set of consecutive integers of the form $\{n,n+1,\dots,m\}$ as $[n,m]$.  
Given an $\omega=(a_n)_{n\ge1}\in\{0,1,\dots,r-1\}^\N$ and a set $\mathbb S\subset\N$, by $\omega|_{\mathbb S}$ we will denote the \emph{restriction} of $\omega$ to $\mathbb S$. For instance, if $\mathbb S=\{s_1,s_2,\dots\}$ is infinite, where $s_1<s_2<\dots$, then $\omega|_{\mathbb S}=(a_{s_1},a_{s_2},\dots)\in\{0,1,\dots,r-1\}^\N$.
If $\mathbb S=[n,n+k-1]$ then $\omega|_{\mathbb S}$ is the block $(a_n,a_{n+1},\dots,a_{n+k-1})$. 

We say that a block $B=(b_1,b_2,\dots,b_k)$ \emph{occurs} in $\omega$ at a coordinate $n\ge1$ if $\omega|_{[n,n+k-1]}=B$.
\medskip

Recall that the notion of normality of a real number $x$ in base $r$ was informally outlined in the Introduction in terms of statistics of appearance of blocks in the \sq\ of digits of the base $r$ expansion of $x$. The goal of the following definitions is to establish a formal setup for dealing with the notion of normality.

\begin{defn}\label{density}
The \emph{lower density} of a set $\mathbb S\subset\mathbb N$ is defined as
$$
\underline d(\mathbb S)=\liminf_{n\to\infty}\frac{|\mathbb S\cap[1,n]|}n.
$$
\emph{Upper density} $\overline d(\mathbb S)$ is defined analogously with $\limsup$. If $\underline d(\mathbb S)=\overline d(\mathbb S)$ then the common value is called \emph{the density of $\mathbb S$} and denoted by $d(\mathbb S)$. In this case we say that \emph{the density $d(\mathbb S)$ exists}.
\end{defn}

\begin{defn}\label{frequency} Given $\omega\in\{0,1,\dots,r-1\}^\N$, $k\in\N$, and a block $B\in\{0,1,\dots,r-1\}^k$, denote $A_\omega(B)=\{n\in\N: \omega|_{[n,n+k-1]}=B\}$. 
\begin{enumerate}[(a)]
\item The \emph{lower} and \emph{upper frequency} of $B$ in $\omega$ are defined, respectively, as 
$\underline d (A_\omega(B))$ and $\overline d(A_\omega(B))$. If $d(A_\omega(B))$ exists, we call it \emph{the frequency of $B$ in $\omega$} and denote by $\mathsf{Fr}(B,\omega)$.
\item If $\B$ is a finite family of blocks then 
$\underline d(\bigcup_{B\in\B}A_\omega(B))$ is called the \emph{lower joint frequency of the blocks from $\B$ in $\omega$} (the same convention applies to upper joint frequency and joint frequency).
\end{enumerate}
\end{defn}

\begin{defn}\label{ns}
A \sq\ $\omega\in\{0,1,\dots,r-1\}^\N$ is \emph{normal} if for any $k\in\N$, any finite block $B=(b_1,b_2,\dots,b_k)\in\{0,1,\dots,r-1\}^k$ appears in $\omega$ with frequency~$r^{-k}$. 
\end{defn}

\smallskip

A distinctive class of \im s on the system $(\{0,1,\dots,r-1\}^\N,\sigma)$ is that of Bernoulli measures. Let $\bar p=(p_0,p_1,\dots,p_{r-1})$ be a probability vector and let $P$ be the probability measure on $\{0,1,\dots,r-1\}$ given by $P(\{i\})=p_i$. The $\bar p$-Bernoulli measure $\mu_{\bar p}$ is the product measure $P^\N$ on $\{0,1,\dots,r-1\}^\N$. If $p_i=\frac1r$ for each $i$ then $\mu_{\bar p}$ is referred to as the \emph{uniform Bernoulli measure}.
\medskip

We say that a point $x\in X$ in a \ds\ \xt\ \emph{generates} (or \emph{is generic for}) a measure $\mu\in\M(X)$ if, in the weak* topology, we have
\begin{equation}\label{ws}
\lim_{n\to\infty}\frac1n\sum_{i=0}^{n-1}\delta_{T^ix}=\mu,
\end{equation}
where $\delta_{T^ix}$ denotes the point-mass concentrated at $T^ix$. Note that in view of the correspondence between Borel probability measures on $X$ and nonnegative normalized functionals on the space $C(X)$ of continuous real functions on $X$, 
the formula~\eqref{ws} is equivalent to the uniform distribution of the orbit $(T^nx)_{n\ge1}$, i.e:  
\begin{equation}\label{wsk}
\lim_{n\to\infty}\frac1n\sum_{i=0}^{n-1}f(T^ix) = \int f \,d\mu, \text{ \ for any $f\in C(X)$}.
\end{equation}

We can now characterize normal \sq s (and hence normal numbers) in terms of dynamics.

\begin{prop}\label{nortop}
A \sq\ $\omega\in\{0,1,\dots,r-1\}^\N$ is normal if and only if it is generic under the shift $\sigma$ for the uniform Bernoulli measure on $\{0,1,\dots,r-1\}^\N$.
\end{prop}

\begin{proof}
By Definition~\ref{ns}, normality of $\omega$ is equivalent to the condition that, for any $k\in\N$ and any block $B$ of length $k$, one has
$$
\lim_{n\to\infty}\frac1n\sum_{i=0}^{n-1}\mathbbm 1_{[B]}(\sigma^i\omega) = r^{-k},
$$
where
\begin{equation}\label{cyl}
[B] = \{\omega\in \{0,1,\dots,r-1\}^\N: \omega|_{[1,k]}=B\}
\end{equation}
is the \emph{cylinder} associated with the block $B$.
Note that for $\bar p=(\frac1r,\frac1r,\dots,\frac1r)$ one has $r^{-k}=\mu_{\bar p}([B])$, where $\mu_{\bar p}$ is the uniform Bernoulli measure. In other words, normality of $\omega$ is equivalent to \eqref{wsk} with $X=\{0,1,\dots,r-1\}^\N$, $T=\sigma$, $\mu=\mu_{\bar p}$, and functions $f$ of the form $\mathbbm1_{[B]}$ where $B$ is any finite block 
(note that such functions belong to $C(X)$). This shows that if $\omega$ is generic for $\mu_{\bar p}$ then it is normal. The opposite implication follows by a standard 
approximation argument from the fact that functions of the form $\mathbbm1_{[B]}$ are linearly dense in $C(X)$.
\end{proof}

Given a general \ds\ $(X,T)$ we say that $x\in X$ \emph{quasi-generates} (or \emph{is quasi-generic for}) a measure $\mu\in\M(X)$ if, for some increasing \sq\ $\mathcal J=(n_k)_{k\ge1}$, we have
\begin{equation}\label{wsl}
\lim_{k\to\infty}\frac1{n_k}\sum_{i=0}^{n_k-1}\delta_{T^ix}=\mu.
\end{equation}
Alternatively, we will say that $x$ \emph{generates $\mu$ along $\mathcal J$}.

It is not hard to see that any measure defined by a limit of the form \eqref{wsl} is necessarily invariant. By compactness of $\M(X)$, every point $x\in X$ quasi-generates at least one \im. We will denote the (nonempty and compact) set of measures quasi-generated by $x$ by $\M_x$. Clearly, $x$ is generic for some measure if and only if $\M_x$ is a singleton.

\begin{rem}\label{mxmy}
Whenever $\phi:X\to Y$ is a factor map from a \ds\ \xt\ to a \ds\ \ys, and $x\in X$ generates (or generates along a \sq\ $\mathcal J$) an \im\ $\mu\in\M(X,T)$ then the point $\phi(x)$ generates (respectively, generates along $\mathcal J$) the \im\ $\phi^*(\mu)\in\M(Y,S)$. Conversely, if $\phi(x)$ generates an \im\ $\nu\in\M(Y,S)$ along a \sq\ $\mathcal J$ then, along some sub\sq\ of $\mathcal J$, $x$ generates some measure $\mu$ and then $\phi^*(\mu)=\nu$. It follows that $\phi^*$ maps $\M_x$ onto $\M_{\phi(x)}$.
\end{rem}
 
Given \ds s \xt\ and \ys\ and \im s $\mu\in\M(T,X)$ and $\nu\in\M(S,Y)$, a \emph{joining} of $\mu$ and $\nu$ is any measure $\xi\in\M(X\times Y)$, invariant under $T\times S$ (defined by $(T\times S)(x,y)=(Tx,Sy)$), with marginals\footnote{Given a measure $\xi$ on a product space $X\times Y$, the marginal of $\xi$ on $X$ is the measure $\xi_X$ satisfying $\xi_X(A)=\xi(A\times Y)$ (where $A$ is a Borel subset of $X$). The marginal $\xi_Y$ on $Y$ is defined analogously.} $\mu$ and $\nu$ on $X$ and $Y$, respectively. We then write $\xi=\mu\vee\nu$ (although there may exist many different joinings of $\mu$ and $\nu$). The product measure $\mu\times\nu$ is a joining. When $\mu\times\nu$ is the unique joining, we will say that the measures $\mu$ and $\nu$ are \emph{disjoint} (in the sense of Furstenberg, see \cite{F}). If $\xi$ is a joining of $\mu$ and $\nu$ then both measure-preserving systems $(X,\mu,T)$ and $(Y,\nu,S)$ are continuous factors of $(X\times Y, \xi, T\times S)$ via the projections on the respective coordinates. If measures $\mu\in\M(X,T)$ and $\nu\in\M(Y,S)$ are generated by some points $x\in X$ and $y\in Y$ along a common \sq\ $\mathcal J$ then a joining of $\mu$ and $\nu$ can be constructed in the following natural way: any measure $\xi$ on $X\times Y$ generated in the product system $(X\times Y, T\times S)$ by the pair $(x,y)$ along a sub\sq\ of $\mathcal J$ (note that such a sub\sq\ always exists by compactness) is a joining of $\mu$ and $\nu$.

\medskip
When dealing with the symbolic system $(\{0,1,\dots,r-1\}^\N,\sigma)$ we will use the following terminology. For each pair of blocks $B$ and $C$ with $|C|\le|B|$ we define the \emph{density of $C$ in $B$} by the formula
\begin{equation}\label{empirical}
\mu_B(C)=\frac1{|B|-|C|+1}|\{n\in[1,|B|-|C|+1]:B|_{[n,n+|C|-1]}=C\}|.
\end{equation}
\begin{defn}\label{wss}
We will say that a \sq\ of blocks $(B_k)_{k\ge1}$, whose lengths $|B_k|$ increase, \emph{generates} an \im\ $\mu$ on $\{0,1,\dots,r-1\}^\N$ if, for every block $C$ over $\{0,1,\dots,r-1\}$, we have
\begin{equation}\label{ws1s}
\lim_{k\to\infty} \mu_{B_k}(C)=\mu([C]).
\end{equation}
\end{defn}

It is a standard fact in symbolic dynamics that any \sq\ of blocks with increasing lengths contains a sub\sq\ which generates an \im. 

Note that a \sq\ $\omega\in\{0,1,\dots,r-1\}^\N$ generates, in the sense of \eqref{wsl}, a measure $\mu$ along a \sq\ $\mathcal J=(n_k)_{k\ge1}$, if and only if the \sq\ of blocks $(B_{k})_{k\ge1}$ generates $\mu$ in the sense of Definition~\ref{wss}, where $B_{k}=\omega|_{[1,n_k]}$.

\begin{rem}\label{uo}
If a \sq\ of blocks $(B_k)_{k\ge1}$ generates an \im\ $\mu$ and, for each $k\ge1$, $B_k$ is a concatenation of $B^{(1)}_k$ and $B^{(2)}_k$ where $\lim_{k\to\infty}\frac{|B^{(1)}_k|}{|B_k|}=\alpha\in[0,1]$, and the \sq s $(B^{(1)}_k)_{k\ge1}$ and $(B^{(2)}_k)_{k\ge1}$ generate some measures $\theta_1$ and $\theta_2$, respectively, then $\mu=\alpha\theta_1+(1-\alpha)\theta_2$.
\end{rem}

Given a number $x\in\R$, consider its base $r$
expansion  
\begin{equation}\label{expansion}
x=\sum_{n=-m}^\infty \frac{x_n}{r^n}.
\end{equation}
The formula~\eqref{expansion} gives rise to a representation of $x$ in the form of a \sq\ of digits $(x_n)_{n\ge-m}$ with a dot between the coordinates 0 and 1, separating the integer part from the fractional part. Clearly, the statistical properties of this \sq\ (which are the main subject of our interest) do not depend on any finite collection of digits, so it is natural to omit the portion representing the integer part as well as the separating dot. The resulting \sq, $\omega_r(x)=(x_n)_{n\ge1}$ is an element of the symbolic space $\{0,1,\dots,r-1\}^\N$. We will call it the \emph{symbolic alias of $x$ in base $r$}, or just \emph{alias}, when there is no ambiguity about the base $r$. When $r=2$, we will often use the term \emph{binary alias}. We can now formalize the definition of the key concept of this paper, outlined at the beginning of the Introduction:

\begin{defn}\label{nn}
Fix an integer $r\ge 2$. A number $x\in\R$ is normal in base $r$ if its alias $\omega_r(x)$ is a normal \sq\ in $\{0,1,\dots,r-1\}^\N$.
\end{defn}

\begin{rem}
It is well known (see \cite[Theorem~1]{Wa} or \cite[Chapter~1, Theorem~8.1]{KN}) that a real number $x$ is normal in base $r$ if and only if the 
\sq\ $(r^nx)_{n\ge1}$ is uniformly distributed\!\!$\mod 1$, i.e:
\begin{equation}\label{wsK2}
\lim_{n\to\infty}\frac1n\sum_{i=0}^{n-1}f(r^ix\!\!\mod 1) = \int f \,dx, \text{ \ for any $f\in C([0,1])$}.
\end{equation}
\end{rem}
Formula \eqref{wsK2} can be viewed as a special case of \eqref{wsk}. The definition of normality via formula~\ref{wsK2} will enable us to prove results dealing with real numbers with the help of the compact dynamical system $(\mathbb T,R)$, where
$\mathbb T$ is the circle $\R/\Z$ and $R$ is the transformation $t\mapsto rt\mod1$, $t\in\mathbb T$ (see more details in Section~\ref{Sn}, in particular Definition~\ref{torus}(2)). 
\smallskip

We conclude this section by introducing a definition which will be instrumental in most of our considerations.

\begin{defn}\label{normpre}
We say that a number $y\in\R$ \emph{preserves normality in base $r$} if $x+y\in\mathcal N(r)$ for every $x\in\mathcal N(r)$. The set of numbers that preserve normality in base $r$ will be denoted by $\mathcal N^\perp(r)$.
\end{defn}

\section{Entropy and determinism}\label{S3}
We start by summarizing some basic facts from the theory of entropy, keeping in mind that throughout this paper we deal only with measure-preserving systems arising from \tl\ systems equipped with an \im.
Recall that the entropy of an \im\ $\mu$ in a \ds\ \xt\ is defined in three steps (see, e.g.,~\cite{W}):
\begin{enumerate}
\item Given a finite measurable partition $\P$ of $X$ one defines the \emph{Shannon entropy of $\P$ with respect to $\mu$} as 
$$
H_\mu(\P)=-\sum_{A\in\P}\mu(A)\log\mu(A),
$$
where $\log$ stands for $\log_2$.
\item The \emph{dynamical entropy of $\P$ with respect to $\mu$ under the action of $T$} is defined by the formula
$$
h_\mu(\P,T)=\lim_{n\to\infty}\frac1n H_\mu(\P^n),
$$
where $\P^n$ stands for the partition 
$$
\bigvee_{i=0}^{n-1}T^{-i}\P = \left\{\bigcap_{i=0}^{n-1}T^{-i}(A_i): \forall_{i\in\{0,1,\dots,n-1\}}\ A_i\in\P\right\}.
$$
\item Finally, the \emph{Kolmogorov--Sinai entropy} of $\mu$ (with respect to the transformation $T$) is defined as
$$
h_\mu(T)=\sup_\P h_\mu(\P,T),
$$
where $\P$ ranges over all finite measurable partitions of $X$.
\end{enumerate} 
By the classical Kolmogorov--Sinai Theorem (\cite{S}), if $\P$ is a generating partition (i.e., such that the partitions $T^{-i}\P$, $i\ge0$, separate points), then $h_\mu(T)=h_\mu(\P,T)$. When the transformation $T$ is fixed, we will abbreviate $h_\mu(T)$ as $h(\mu)$. 

In this paper, we will also use the notion of \tl\ entropy introduced in~\cite{ACM}. 
It is known (see, e.g.,~\cite{M}) that \tl\ entropy is characterized by the so-called \emph{variational principle}, which, for convenience, we will use as definition:
\begin{defn}\label{topen}
Let $(X,T)$ be a \tl\ \ds. The \emph{\tl\ entropy} of the system equals
$$
\htop(X,T)=\sup\{h_\mu(T):\mu\in\M(X,T)\}.
$$
\end{defn}

Let $(X,\mu,T)$ and $(Y,\nu,S)$ be measure-preserving systems. 
We will be using the following classical facts (see, e.g.,~\cite[Facts 4.1.3 and 4.4.3]{Do}):

If $(Y,\nu,S)$ is a continuous factor of $(X,\mu,T)$ then 
\begin{equation}\label{eof}
h(\nu)\le h(\mu).
\end{equation}
and if $\xi$ is a joining of $\mu$ and $\nu$ then
\begin{equation}\label{eoj}
h(\xi)\le h(\mu)+h(\nu).
\end{equation}
If $\xi = \mu\times\nu$ then one has equality in~\eqref{eoj}.

We will also make use of joinings of countably many measures, $\xi=\bigvee_{m\ge1}\mu_m$. In this case the inequality~\eqref{eoj} remains valid in the following form:
\begin{equation}\label{eoj1}
h(\xi)\le \sum_{m\ge1}h(\mu_m).
\end{equation}

\begin{defn}\label{lue}
The \emph{lower} and \emph{upper entropies of a point} $x$ in a \tl\ \ds\ \xt\ are defined as
$$
\underline h(x) = \inf\{h(\mu):\mu\in\M_x\}, \ \ \ \overline h(x) = \sup\{h(\mu):\mu\in\M_x\}.
$$
If $\underline h(x)=\overline h(x)$ then we denote the common value by $h(x)$ and call it the \emph{entropy~of}~$x$.
\end{defn}
In particular, the entropy of a point $x$ is well defined for every point which is generic for some measure $\mu$ (and then $h(x)=h(\mu)$). 

\begin{rem}\label{conule}
If two systems, $(X,T)$ and $(Y,S)$, are \tl ly conjugate via a map $\phi$ then, for any $x\in X$, $\underline h(x)=\underline h(\phi(x))=$ and $\overline h(x)=\overline h(\phi(x))$. Indeed, it follows from Remark~\ref{mxmy} that 
$\phi^*(\M_x)=\M_{\phi(x)}$, and by Remark~\ref{conj}, for each $\mu\in\M_x$ the system $(X,\mu,T)$ is isomorphic to $(Y,\phi^*(\mu),S)$. The claim then follows from the classical fact that isomorphic systems have equal entropies.
\end{rem}

\begin{defn}\label{real}
When the base of expansion $r$ is fixed, by the lower and upper entropies of a real number $x$, $\underline h(x)$ and $\overline h(x)$, respectively, we will understand the lower and upper entropies of the alias $\omega_r(x)$ viewed as an element of the symbolic system $(\{0,1,\dots,r-1\}^\N,\sigma)$.
\end{defn}

We will now introduce, for a fixed base $r$, the notion of a deterministic number~$x$. Similarly to normality and upper/lower entropy, the notion of a determinism hinges on statistical/combinatorial/dynamical properties of the alias $\omega_r(x)$.

There are several equivalent definitions of deterministic sequences, some of which we will only describe briefly, as they are quite intricate and not needed in this work. The essential feature of deterministic sequences is that they have ``low complexity'' for some appropriate notion of complexity.

We will be mostly using the dynamical definition of a deterministic \sq\ introduced by B. Weiss in \cite[Definition~1.6]{W2} (under the name \emph{completely deterministic}). 

\begin{defn} \label{deta} \phantom{a}
Let \xt\ be a \ds. A point $x\in X$ is called \emph{deterministic} if all measures in $\M_x$ (measures quasi-generated by $x$) have entropy zero. We will say that a \sq\ $\omega=(a_n)_{n\ge1}\in\{0,1,\dots,r-1\}^\N$ is \emph{deterministic} if $\omega$ is a deterministic element of the symbolic system $(\{0,1,\dots,r-1\}^\N,\sigma)$.
\end{defn}

For the sake of completeness, we now indicate how deterministic symbolic \sq s can be defined directly, via statistical/combinatorial properties, without referring to dynamical systems.

\begin{defn}\label{complexity}
Let $\omega\in\{0,1,\dots,r-1\}^\N$. Given $\eps\in(0,1)$ and $m\in\N$, by the \emph{$\eps$-complexity of $\omega$ at $m$} we mean the minimal number $C_\omega(\eps,m)$ such that there exists a family of blocks $F\subset\Lambda^m$ of cardinality $C_\omega(\eps,m)$ and a set $\mathbb S\subset\N$ of upper density not exceeding $\eps$, satisfying 
\begin{equation}\label{seup}
\omega|_{[i,i+m-1]}\in F \text{ \ for all $i\notin\mathbb S$}.
\end{equation}
\end{defn}

\begin{rem}\label{inco}
Clearly, if~\eqref{seup} is satisfied for a family $F\subset\Lambda^m$ and a set $\mathbb S\subset\N$ of upper density not exceeding $\eps$ then $C_\omega(\eps,m)\le|F|$.
\end{rem}

\begin{defn}\label{detb}
A \sq\ $\omega\in\{0,1,\dots,r-1\}^\N$ has \emph{subexponential epsilon-complexity} if for any $\eps>0$ there exists an $m\in\N$ such that $C_\omega(\eps,m)<2^{\eps m}$.
\end{defn}

\begin{thm}{\rm(see \cite[Lemma 8.9]{W3} for a slightly different yet equivalent formulation)}\label{detdet}
A \sq\ $\omega\in\{0,1,\dots,r-1\}^\N$ is deterministic if and only if it has subexponential epsilon-complexity.
\end{thm}
Lemma 8.9 is stated in \cite{W3} without a proof. An explicit proof of a more general (and more cumbersome) theorem dealing with the setup of actions of countable amenable groups is given in \cite[Theorem 6.11]{BDV}. For reader's convenience, we include a relatively short proof of Theorem~\ref{detdet} in the Appendix.

We are now in a position to define deterministic real numbers.

\begin{defn}\label{detn}
A real number $x$ \emph{is deterministic in base $r$} if its alias $\omega_r(x)$ is a deterministic \sq\ in $\{0,1,\dots,r-1\}^\N$. The set of real numbers deterministic in base $r$ will be denoted by $\mathcal D(r)$.
\end{defn}

The following proposition provides a class of examples of deterministic numbers. 

\begin{prop}\label{triv} Let $\mathbb S\subset\N$ be a set of density $1$.
Let $y\in\R$ and assume that $\omega_r(y)|_{\mathbb S}$ (the restriction of the alias of $y$ to $\mathbb S$) is periodic. Then $y$ is deterministic in base $r$.    
\end{prop}

\begin{proof}
Assume first that $y'\in\R$ is such that $\omega_r(y')$ is periodic. Clearly $y'\in\mathcal D(r)$. Indeed, the \sq\ $\omega_r(y')$ generates a measure supported by a periodic orbit and this measure has entropy zero. Now, if $\omega_r(y)|_{\mathbb S}=\omega_r(y')$ then $\omega_r(y)$ generates the same measure as $\omega_r(y')$, because the digits in $\omega_r(y)$ appearing along the set $\N\setminus\mathbb S$ of density zero do not alter the frequencies of any blocks. So, $y\in\mathcal D(r)$ as well. 
\end{proof}

As mentioned earlier, Rauzy in \cite{Ra} provided the following remarkable characterization of numbers $y\in\mathcal D(r)$, which served as the main motivation for our work. 
\begin{thm}\label{rrrr}
A real number $y$ is deterministic in base $r$ if and only if, for any $x\in\mathcal N(r)$ one has $x+y\in\mathcal N(r)$. That is,
$$
\mathcal D(r)=\mathcal N^\perp(r).
$$
\end{thm}

\begin{rem}
Prior to Rauzy, in 1969, J.\ Spears and J.\ Maxfield \cite{SM}, proved that numbers $y$ that match our description in Proposition~\ref{triv} belong to $\mathcal N^\perp(r)$.
\end{rem}

Theorem~\ref{rrrr} can be viewed as a third equivalent definition of a deterministic real number. It is worth mentioning that the paper \cite{Ra} gives yet another (fourth) definition (which we will not use in this paper) in terms of a ``noise function''. The noise of a given sequence $(a_n)_{n\ge1}$ is a measure of how difficult it is to predict the value of $a_n$ given information about the ``tail'' $a_{n+1},a_{n+2},\dots,a_{n+s}$ as $s\to\infty$. Deterministic sequences
are those of zero noise (i.e., one can almost always predict with high probability the value $a_n$ given the information about a sufficiently long tail). The proof in \cite{Ra} of the equivalence between the noise-based definition with Definition~\ref{deta} is quite nontrivial.

\section{Rauzy theorem as a phenomenon associated with entropy} \label{Sn}

In order to discuss phenomena associated with Rauzy theorem for real numbers in terms of entropy in dynamical systems, we need to replace the noncompact space of real numbers by a more manageable compact model. This will be done in Subsection~{4.1}.
In subsections~\ref{3.1},~\ref{4.3},~\ref{qx} we present purely dynamical proofs of statements concerning the behavior of lower and upper entropy under algebraic operations, and provide interpretation of these results for real numbers. In particular, we derive the ``necessity'' in Rauzy theorem (Theorem~\ref{rrrr}) in Corollary~\ref{rrr}(1) from entropy inequalities established in Proposition~\ref{x+y}. For completeness, in Subsection~\ref{rev} we prove ``sufficiency'' in Rauzy theorem (admittedly, this prove already depends also on Fourier analysis and does not differ much from Rauzy's original proof).

\subsection{Passing from real numbers to compact dynamical systems}\label{4.1a}
In previous sections the definitions of normality and determinism of a real number $x$ were introduced via the symbolic alias $\omega_r(x)$ viewed as an element of the symbolic space $\{0,1,\dots,r-1\}^\N$. In this manner, we are making a convenient reduction from the non-compact set $\R$ to the compact symbolic space equipped naturally with the shift transformation $\sigma$. 

Since addition of real numbers interpreted in terms of the base $r$ expansions leads to the rather cumbersome addition \emph{with the carry}, we will find it convenient to work with yet another \tl\ system, namely $(\mathbb T,R)$, where $\mathbb T$ is the circle $\R/\Z$ and $R$ is given by $R(t)=rt$, $t\in\mathbb T$. 
The natural bijection between the interval $[0,1)$ and the circle $\mathbb T=\R/\Z$, given by $[0,1)\ni t\mapsto t+\Z\in\R/\Z$ allows us to view, for each real number $x$, its fractional part $\{x\}$ as an element of the circle~$\mathbb T$. With this identification, the mapping $x\mapsto\{x\}$ is in fact a group homomorphism from $\R$~to~$\mathbb T$. More precisely, $\{x+_{_\R}y\}=\{x\}+_{_\mathbb T}\{y\}$, where $+_{_\R}$ and $+_{_\mathbb T}$ are group operations in $\R$ and $\mathbb T$, respectively. In the sequel we will  use ``$+$'' for both $+_{_\R}$ and $+_{_\mathbb T}$, as the group to which the operation refers will be clear from the context. A similar convention will apply to the subtraction sign ``$-$''. 

The systems $(\{0,1,\dots,r-1\}^\N,\sigma)$ and $(\mathbb T,R)$ are linked by an ``almost invertible'' factor map, described below.

\begin{prop}\label{rr}Define the map $\phi_r:\{0,1,\dots,r-1\}^\N\to\mathbb T$ as follows: For $\omega=(a_n)_{n\ge1}\in\{0,1,\dots,r-1\}^\N$ we let  
$$
\phi_r(\omega)=\begin{cases}
0,& \text{ if }a_n = r-1\text{ for all }n\in\N,\\
\sum_{n=1}^\infty \frac{a_n}{r^n},& \text{ otherwise}.
\end{cases}
$$
Then $\phi_r$ is a factor map from the symbolic system $(\{0,1,\dots,r-1\}^\N,\sigma)$ to $(\mathbb T,R)$, and for each nonatomic \im\ $\mu$ on the symbolic system, $\phi_r$ is an isomorphism between the measure-preserving systems 
$(\{0,1,\dots,r-1\}^\N,\mu,\sigma)$ and $(\mathbb T,\phi^*_r(\mu),R)$. In particular, we have the equality $h(\phi_r^*(\mu))=h(\mu)$.
If $\mu$ has atoms then the systems $(\{0,1,\dots,r-1\}^\N,\mu,\sigma)$ and $(\mathbb T,\phi^*_r(\mu),R)$ need not be isomorphic, but still the equality $h(\phi_r^*(\mu))=h(\mu)$ holds.
\end{prop}

\begin{proof}[Proof of Proposition~\ref{rr}]
The fact that $\phi_r\circ\sigma=R\circ\phi_r$ is straightforward, as well as the fact that $\phi_r$ is invertible except on the countable set of \sq s that are eventually 0 or eventually $r-1$. Since this exceptional set is countable, it follows that $\phi_r$ is invertible $\mu$-almost everywhere for any nonatomic measure $\mu$ on the symbolic system. Thus $\phi_r$ an is isomorphism between $(\{0,1,\dots,r-1\}^\N,\mu,\sigma)$ and $(\mathbb T,\phi^*_r(\mu),R)$.
The last statement (for measures $\mu$ with atoms) follows from the fact that finite-to-one factor maps preserve entropy of \im s (see, e.g.,~\cite[Theorem~2.1]{LW}).
\end{proof}

\begin{rem}
If $t\in\mathbb T$ is of the form $\{\frac a{r^n}\}$, where $a\in\N\cup\{0\}$ (and $a$ is not necessarily co-prime with $r$) then $t$ has two preimages via $\phi_r$, one whose digits are eventually $0$'s, and another, whose digits are eventually $r-1$. By convention, the alias $\omega_r(t)$ of $t$ is the \sq\ ending with zeros (exceptionally, one time in Section~\ref{S7}, the other preimage will also be used). For any other $t$, $\omega_r(t)$ is the unique preimage of $t$~by~$\phi_r$. 
\end{rem}


\begin{rem}\label{umome}
Notice that if $\mu$ denotes the uniform Bernoulli measure on $\{0,1,\dots,r-1\}^\N$ then
$\phi_r^*(\mu)$ equals the Lebesgue measure $\lambda$ on $\mathbb T$. It is a classical fact that $\mu$ is the unique \im\ on $(\{0,1,\dots,r-1\}^\N,\sigma)$ of maximal entropy, i.e., such that $h_\mu(\sigma)$ is equal to the \tl\ entropy $\htop(\{0,1,\dots,r-1\}^\N,\sigma)=\log r$ (see , e.g.,~\cite{AW})\footnote{Systems with a unique measure of maximal entropy are often called \emph{intrinsically ergodic, see~\cite{W1}}.}. In view of Proposition~\ref{rr}, it follows that the Lebesgue measure is the unique measure with maximal entropy $\log r$ on $(\mathbb T,R)$.
\end{rem}

\begin{defn}\label{torus} Let the base $r\ge2$ be fixed and let $R$ denote the map $t\mapsto rt$, $t\in\mathbb T$.
\begin{enumerate}
    \item By $\underline h(t)$ and $\overline h(t)$, where $t\in\mathbb T$, we will mean the lower and upper entropies of $t$ in the system $(\mathbb T,R)$. 
    \item An element $t\in\mathbb T$ is said to be \emph{$R$-normal} if it is generic for the Lebesgue measure in the system $(\mathbb T,R)$. The set of $R$-normal elements of $\mathbb T$ will be denoted by $\mathcal N(\mathbb T,R)$.
    \item An element $s\in\mathbb T$ is said to \emph{preserve $R$-normality} if $s+t\in\mathcal N(\mathbb T,R)$ for every $t\in\mathcal N(\mathbb T,R)$. The set of elements of $\mathbb T$ that preserve $R$-normality will be denoted by $\mathcal N^\perp(\mathbb T,R)$.
    \item An element $t\in\mathbb T$ is \emph{$R$-deterministic} if it is a deterministic element in the system $(\mathbb T,R)$. 
    The set of $R$-deterministic elements of $\mathbb T$ will be denoted by $\mathcal D(\mathbb T,R)$.  
    \item An element $s\in\mathbb T$ is said to \emph{preserve $R$-determinsim} if $s+t\in\mathcal D(\mathbb T,R)$ for every $t\in\mathcal D(\mathbb T,R)$. The set of elements of $\mathbb T$ that preserve $R$-determinism will be denoted by $\mathcal D^\perp(\mathbb T,R)$.
\end{enumerate}
\end{defn}
\begin{rem}\label{thesame} In view of Definition~\ref {real} and Propositions~\ref{rr},~\ref{nortop}, we have:
\begin{enumerate}
    \item If $x$ is any real number such that $\{x\}=t$, then $\underline h(x)=\underline h(t)$ and $\overline h(x)=\overline h(t)$, where 
    $\underline h(x)$ and $\overline h(x)$ denote the lower and upper entropy of real numbers with respect to their base $r$-expansions (see Definition~\ref{real}), while $\underline h(t)$ and $\overline h(t)$ denote the lower and upper entropy of a point in the system $(\mathbb T,R)$.
\item $t\in\mathcal N(\mathbb T,R)$ if and only if $x\in\mathcal N(r)$ for any real number $x$ with $\{x\}=t$. 
\item $t\in\mathcal N^\perp(\mathbb T,R)$ if and only if $x\in\mathcal N^\perp(r)$ for any real number $x$ with~$\{x\}=t$.
\item $t\in\mathcal D(\mathbb T,R)$ if and only if $x\in\mathcal D(r)$ for any real number $x$ with $\{x\}=t$.
\item $t\in\mathcal D^\perp(\mathbb T,R)$ if and only if $x\in\mathcal D^\perp(r)$ for any real number $x$ with~$\{x\}=t$.
\end{enumerate}
\end{rem}

We can now rephrase the Rauzy theorem (Theorem~\ref{rrrr}) in terms of the system $(\mathbb T,R)$. The proof follows directly from Theorem~\ref{rrrr} and Remark~\ref{thesame}(3) and (4).

\begin{thm}{\rm(Version of Rauzy theorem for an endomorphism of the circle)}\label{newrauzy}
$$
\mathcal D(\mathbb T,R)=\mathcal N^\perp(\mathbb T,R).
$$
\end{thm}

The following proposition demonstrates that normality and determinism are intrinsically connected to lower and upper entropy. We keep the base $r$ fixed and, as before, $R$ denotes the transformation $t\mapsto rt$, $t\in\mathbb T$.  

\begin{prop}\label{nordet1}\phantom{a}
\begin{enumerate}
	\item A point $x$ in a \tl\ \ds\ $(X,T)$ is deterministic if and only if $h(x)$ exists and equals $0$.
        \item An element $t\in\mathbb T$ is $R$-deterministic if and only if $h(t)$ with respect to the transformation $R$ exists and equals zero.
	\item An element $t\in\mathbb T$ is $R$-normal if and only if $h(t)$ with 
        respect to the transformation $R$ exists and equals $\log r$.
        \item A real number $x$ is normal in base $r$ if and only if $h(x)$ (see Definition~\ref{real}) exists and equals $\log r$.
        \item A real number $x$ is deterministic in base $r$ if and only if $h(x)$ exists and equals $0$.
\end{enumerate}
\end{prop}

\begin{proof}The statements (1) and (2) are obvious. 
The statements~(3), (4) and~(5) follow from Remark~\ref{umome} and Definition~\ref{real}. 
\end{proof}

\begin{rem}\label{minus}
The map $t\mapsto-t$ is a \tl\ conjugacy of the system $(\mathbb T, R)$ with itself, hence, in view of Remark~\ref{conule}, we have $\underline h(-t)=\underline h(t)$ and $\overline h(-t)=\overline h(t)$. In particular, if $t\in\mathbb T$ is $R$-normal or $R$-deterministic then so is $-t$, that is $-\mathcal N(\mathbb T,R)=\mathcal N(\mathbb T,R)$ and $-\mathcal D(\mathbb T,R)=\mathcal D(\mathbb T,R)$. Combining this fact with Remark~\ref{thesame} we get that $-\mathcal N(r)=\mathcal N(r)$ and $-\mathcal D(r)=\mathcal D(r)$.
\end{rem}

\subsection{Behavior of lower and upper entropies under addition}\label{3.1}
In this subsection we continue to work with a fixed (but arbitrary) base $r\ge2$ and with the system $(\mathbb T,R)$, where $R(t)=rt$, $t\in\mathbb T$. Most of the time throughout this subsection, the letters $x$ and $y$ denote elements of $\mathbb T$ rather than real numbers (exceptions: Corollary~\ref{rrr}, Question~\ref{q1} and Proposition~\ref{any}). The symbols $\underline h(x)$ and $\overline h(x)$ stand for the lower and upper entropy of a point in the system $(\mathbb T,R)$.

\begin{prop}\label{x+y}Recall (see Remark~\ref{umome}) that $\htop(\mathbb T,R)=\log r$.
For any $x,y\in\mathbb T$ we have
\begin{gather*}
\max\{0, \underline h(x)-\overline h(y), \underline h(y)-\overline h(x)\}\!\overset{(a)}\le\!\underline h(x+y)\!\overset{(b)}\le\!\min\{\log r, \underline h(x)+\overline h(y), \overline h(x)+\underline h(y)\},\\
\max\{|\underline h(x)-\underline h(y)|,\ |\overline h(x)-\overline h(y)|\}\overset{(c)}\le\overline h(x+y)\overset{(d)}\le\min\{\log r,\ \overline h(x) + \overline h(y)\}.
\end{gather*}
\end{prop}

We remark that Kamae in \cite{K2} introduced a notion of \emph{entropy of a point}, which coincides with our upper entropy of a point, and proved the inequality (d). Note however, that since normality is characterized in terms of lower entropy (see Proposition~\ref{nordet1}(3)), the inequality (d) alone is insufficient to prove even the ``necessity'' of Rauzy theorem (which we do in Corollary~\ref{rrr}(1)).

\begin{proof}[Proof of Proposition~\ref{x+y}]
(a) Let $\eps$ be a positive number and let $\mathcal J$ be a \sq\ along which $x+y\in\mathbb T$ generates (via the transformation $R$) an \im\ $\mu$ with entropy not exceeding $\underline h(x+y)+\eps$. There is a sub\sq\ $\mathcal J'$ of $\mathcal J$ along which the points $x$ and $y$ generate some measures $\nu_x$ and $\nu_y$ on $\mathbb T$, respectively. Clearly, 
$$
\underline h(x)\le h(\nu_x) \text{ \ \ and \ \ } h(\nu_y)\le\overline h(y) . 
$$
The pair $(x+y,y)\in\mathbb T\times\mathbb T$ generates (via the transformation $R\times R$) along some sub\sq\ $\mathcal J''$ of $\mathcal J'$, some joining $\zeta$ of $\mu$ and $\nu_y$. By \eqref{eoj}, we have
$$
h(\zeta)\le h(\mu)+h(\nu_y).
$$ 
The mapping from $\mathbb T\times\mathbb T$ to $\mathbb T$ defined by $(t,u)\mapsto t-u$, $t,u\in\mathbb T$, is continuous, surjective and equivariant:
$$
(R\times R)(t,u)=(Rt,Ru)=(rt,ru)\mapsto(rt-ru)=r(t-u)=R(t-u).
$$
This means that $(\mathbb T,R)$ is a factor of $(\mathbb T\times\mathbb T,R\times R)$ via this map. Since $x$ is the image of $(x+y,y)$, it generates along $\mathcal J''$ some measure which is a factor of $\zeta$. 
On the other hand, as $\mathcal J''$ is a sub\sq\ of $\mathcal J$, we know that $x$ generates $\nu_x$ along $\mathcal J''$. It follows that $(\mathbb T,\nu_x,R)$ is a continuous factor of $(\mathbb T, \zeta,R)$ and hence $h(\nu_x)\le h(\zeta)$. 
We have shown that
$$
\underline h(x)\le h(\nu_x)\le h(\zeta) \le h(\mu)+h(\nu_y) \le \underline h(x+y) +\eps+\overline h(y).
$$
Since $\eps$ is arbitrary, we get
\begin{equation}\label{o1}
\underline h(x)-\overline h(y)\le \underline h(x+y).
\end{equation}
By switching the roles of $x$ and $y$ we also get 
\begin{equation}\label{o2}
\underline h(y)-\overline h(x)\le\underline h(x+y).
\end{equation}
Combining \eqref{o1} and \eqref{o2} we get (a).
\smallskip

(b) Let $\eps$ be a positive number. There exists an increasing \sq\ $\mathcal J$  of natural numbers along which:
\begin{itemize}
	\item[---] $x$ generates a measure $\nu_x$ with entropy not exceeding $\underline h(x)+\eps$, 
	\item[---] $y$ generates some measure $\nu_y$,
	\item[---] $x+y$ generates some measure $\mu$,
	\item[---] the pair $(x,y)$ generates some joining $\xi$ of $\nu_x$ and $\nu_y$,
	\item[---] the pair $(x+y,x)$ generates some joining $\zeta$ of $\mu$ and $\nu_x$ (to be used in the proof of (c)).
\end{itemize}	
The factor map $(t,u)\mapsto t+u$, $t,u\in\mathbb T$,  sends the pair $(x,y)$ to $x+y$, hence the adjacent map on measures sends $\xi$ to $\mu$. This implies that
$$
\underline h(x+y)\le h(\mu)\le h(\xi)\le h(\nu_x)+h(\nu_y)\le\underline h(x) +\eps+ \overline h(y).
$$ 
Since $\eps$ is arbitrary, we have shown that 
\begin{equation}\label{o3}
\underline h(x+y)\le\underline h(x) + \overline h(y).
\end{equation}
By switching the roles of $x$ and $y$ we also get 
\begin{equation}\label{o4}
\underline h(x+y)\le\underline h(y) + \overline h(x).
\end{equation}
Combining \eqref{o3}, \eqref{o4} and the fact that the entropy of any \im\ of the system $(\mathbb T,R)$ cannot exceed $\log r$ (see Remark~\ref{umome}), we obtain (b).

\smallskip
(c) Let $\eps$, $\mathcal J$, $\mu$, $\nu_x$, $\nu_y$ and $\zeta$ be as in the proof of (b). The factor map $(t,u)\mapsto t-u$, $t,u\in\mathbb T$, sends $(x+y,x)$ to $y$, hence the adjacent map on measures sends $\zeta$ to $\nu_y$. Thus
$$
\underline h(y)\le h(\nu_y) \le h(\zeta) \le h(\mu)+h(\nu_x)\le\overline h(x+y) + \underline h(x)+\eps.
$$ 
Since $\eps$ is arbitrary, we obtain
$$
\underline h(y)-\underline h(x)\le\overline h(x+y).
$$
By switching the roles of $x$ and $y$ we also get $\underline h(x)-\underline h(y)\le\overline h(x+y)$, and so we have 
\begin{equation}\label{o5}
|\underline h(y)-\underline h(x)|\le\overline h(x+y).
\end{equation}

Choose again an $\eps>0$ and let $\mathcal J'$ be a \sq\ along which $x$ generates a measure $\nu'_x$ with entropy exceeding $\overline h(x)-\eps$, while $y$, $x+y$ and the pair $(x+y,y)$ generate some measures $\nu'_y$, $\mu'$ and some joining $\zeta'$ of $\mu'$ and $\nu'_y$, respectively. Then the map adjacent to the factor map $(t,u)\mapsto t-u$, $t,u\in\mathbb T$, sends $\zeta'$ to $\nu'_x$ and thus
$$
\overline h(x)-\eps\le h(\nu'_x)\le h(\zeta')\le h(\mu')+h(\nu'_y)\le \overline h(x+y)+\overline h(y),
$$
implying that
$$
\overline h(x)  - \eps- \overline h(y)\le \overline h(x+y).
$$
Since $\eps$ is arbitrary, we get $\overline h(x) - \overline h(y)\le \overline h(x+y)$.
By switching the roles of $x$ and $y$ we also get $\overline h(y) - \overline h(x)\le \overline h(x+y)$, and so 
\begin{equation}\label{o6}
|\overline h(x) - \overline h(y)|\le \overline h(x+y).
\end{equation}
Clearly, (c) follows from \eqref{o5} and \eqref{o6}.

\smallskip(d) For an $\eps>0$ let $\mathcal J$ denote a \sq\ along which $x+y$ generates an \im\ 
$\mu$ with entropy exceeding $\overline h(x+y)-\eps$, while $x$ and $y$ generate some measures $\nu_x$ and $\nu_y$, respectively, and the pair $(x,y)$ generates a joining $\xi$ of $\nu_x$ and $\nu_y$. We have
$$
h(\xi)\le h(\nu_x)+h(\nu_y)\le \overline h(x)+\overline h(y).
$$ 
The map $(t,u)\mapsto t+u$ sends $(x,y)$ to $x+y$ and hence the adjacent map on measures sends $\xi$ to $\mu$. We have shown that
$$
\overline h(x+y)-\eps \le h(\mu)\le h(\xi)\le h(\nu_x)+h(\nu_y)\le \overline h(x)+\overline h(y).
$$
Since $\eps$ is arbitrary, we get 
$$
\overline h(x+y)\le \overline h(x)+\overline h(y).
$$
The inequality $\overline h(x+y)\le\log r$ is obvious, and so we have proved (d).
\end{proof}

\begin{cor}\label{rauzy}\phantom{a} 
The following facts hold for the system $(\mathbb T,R)$:
\begin{enumerate}
  \item Fix $x,y\in\mathbb T$. If $h(x)$ and $h(y)$ exist then 
  $$
  |h(x)-h(y)|\le \underline h(x+y)\le\overline h(x+y)\le h(x)+h(y).
  $$
  \item An element $y\in\mathbb T$ is $R$-deterministic if and only if for any $x\in\mathbb T$ we have 
  $$
  \underline h(x+y) = \underline h(x) \text{ \ and \ }
  \overline h(x+y) = \overline h(x).
  $$ 
  \item $\mathcal D(\mathbb T,R)\subset\mathcal N^\perp(\mathbb T,R)$.
  \item $\mathcal D(\mathbb T,R)\subset\mathcal D^\perp(\mathbb T,R)$, i.e., if $x\in\mathcal D(\mathbb T,R)$ and $y\in\mathcal D(\mathbb T,R)$ then $x+y\in\mathcal D(\mathbb T,R)$. Combining this fact with Remark~\ref{minus} we get that $\mathcal D(\mathbb T,R)$ is a group.
  \item $\mathcal D(\mathbb T,R)\supset\mathcal D^\perp(\mathbb T,R)$ (and thus $\mathcal D(\mathbb T,R)=\mathcal D^\perp(\mathbb T,R)$).
\end{enumerate}
\end{cor}

\begin{proof}
The statements (1), (3) and (4) are obvious. For an $R$-deterministic $y\in\mathbb T$ both equalities in (2) follow from Proposition~\ref{x+y}. If $y$ is not $R$-deterministic (i.e., if $\overline h(y)>0$) then the second equality in (2) fails for example for $x=0$. This also proves (5). It is also possible (but much harder) to explicitly construct an $x$ for which the first equality fails. We skip the tedious construction. (It will follow from Theorem~\ref{rauzy1} that any normal $x$ is an example, however, this is not a consequence of Proposition~\ref{x+y}).
\end{proof}

In view of Remark~\ref{thesame} we have:
\begin{cor}\label{rrr}\phantom{a} 
\begin{enumerate}
  \item (Rauzy, \cite{Ra}, ``necessity'') $\mathcal D(r)\subset\mathcal N^\perp(r)$.
  \item $\mathcal D(r)=\mathcal D^\perp(r)$. The set $\mathcal D(r)$ is a subgroup of $(\R,+)$.
\end{enumerate}
\end{cor}
\medskip
We now introduce the notion of independence of generic points in \ds s.

\begin{defn}\label{indd}
Let $(X_1,T_1), (X_2,T_2), \dots (X_k,T_k)$ be \tl\ \ds s and let $x_i\in X_i$ be generic for a $T_i$-\im\ $\mu_i$ on $X_i$, $i=1,2,\dots,k$. We say that the elements $x_1, x_2,\dots, x_k$ are \emph{independent} if the $k$-tuple $(x_1,x_2,\dots,x_k)$ is generic in the product system $(X_1\times X_2,\dots\times X_k,T_1\times T_2,\dots\times T_k)$ for the product measure $\mu_1\times\mu_2\times\dots\times\mu_k$.
\end{defn}

\begin{defn}\label{indrn}\phantom{.}
\begin{itemize}
\item[(a)] Real numbers $x_1,x_2,\dots,x_k$ are said to be \emph{$r$-independent} if their aliases $\omega_r(x_1),\omega_r(x_2),\dots,\omega_r(x_k)$, viewed as elements of the symbolic system\break 
$(\{0,1,\dots,r-1\},\sigma)$, are independent.
\item[(b)] Elements $t_1,t_2,\dots,t_k\in\mathbb T$ are said to be \emph{$R$-independent} if they are independent in the system $(\mathbb T,R)$.
\end{itemize}
\end{defn}

\begin{rem}\label{lkj}
Invoking the map $\phi_r:\{0,1,\dots,r-1\}^\N\to\mathbb T$ it can be seen that real numbers $x_1,x_2,\dots,x_k$ are $r$-independent if and only if their fractional parts $\{x_1\},\{x_2\},\dots,\{x_k\}$, are \emph{$R$-independent}.  
\end{rem}

Independence of symbolic \sq s can be expressed in terms of  frequencies of simultaneous occurrences of blocks\footnote{Independence in this setting has been introduced by Rauzy for arbitrary sequences in a compact metric space, see \cite[Chapter 4, Section 4, page 91]{Ra1}}. For simplicity, consider just two \sq s $\omega_1,\omega_2\in\{0,1,\dots,r-1\}^\N$, and let $B_1\in\{0,1,\dots,r-1\}^{k_1}$, $B_2\in\{0,1,\dots,r-1\}^{k_2}$ be two blocks. We say that the pair of blocks $(B_1,B_2)$ \emph{occurs in the pair of \sq s $(\omega_1,\omega_2)$ at a position $n$} if $B_1$ occurs in $\omega_1$ starting at the position $n$ and, simultaneously, $B_2$ occurs in $\omega_2$ starting at the position $n$. 
In analogy to Definition~\ref{frequency}, we will say that the frequency of the pair of blocks $(B_1,B_2)$ in the pair of \sq s $(\omega_1,\omega_2)$ exists if the density
$$
d(\{n: (B_1,B_2) \text{ occurs in } (\omega_1,\omega_2) \text{ at the position }n\})
$$ 
exists. We will denote this frequency by $\mathsf{Fr}(B_1,B_2,\omega_1,\omega_2)$.
With this terminology, the symbolic \sq s $\omega_1, \omega_2$ are independent if, for any blocks $B_1\in\Lambda^{k_1}$,  $B_2\in\Lambda^{k_2}$, we have
\begin{itemize}
\item the frequency  $\mathsf{Fr}(B_1,\omega_1)$ of $B_1$ in $\omega_1$ exists,
\item the frequency  $\mathsf{Fr}(B_2,\omega_2)$ of $B_2$ in $\omega_2$ exists,
\item the frequency $\mathsf{Fr}(B_1,B_2,\omega_1,\omega_2)$ of the pair of blocks $(B_1,B_2)$ in the pair of \sq s $(\omega_1,\omega_2)$ exists and satisfies 
$$
\mathsf{Fr}(B_1,B_2,\omega_1,\omega_2)=\mathsf{Fr}(B_1,\omega_1)\mathsf{Fr}(B_2,\omega_2).
$$
\end{itemize}

\begin{exam}Recall that by $\omega_r(x)$ we denote the alias of a real number $x$ in base~$r$. Let $x\in\R$ be normal in base $4$ and let $y, z$ be real numbers satisfying, for each $n\in\N$,
$$
(\omega_2(y))_n=\lfloor\tfrac12(\omega_4(x))_n\rfloor, \ \ (\omega_2(z))_n=(\omega_4(x))_n\ (\!\!\!\!\!\!\mod2).
$$
Then $y$ and $z$ are normal in base 2 and $2$-independent.  
\end{exam}
\begin{proof}As easily verified, the map $\pi:\{0,1\}^\N\times\{0,1\}^\N\to\{0,1,2,3\}^\N$ given by 
$$
\pi(\omega,\omega')=(b_n)_{n\ge1}, \text{ where } = 2a_n + a'_n, \ \omega=(a_n)_{n\ge1}, \ \omega'=(a'_n)_{n\ge1}
$$
is continuous, bijective, and commutes with the shift. So, it is a \tl\ conjugacy between the product system $(\{0,1\}^\N\times\{0,1\}^\N,\sigma\times\sigma)$ and the shift on four symbols $(\{0,1,2,3\}^\N,\sigma)$. 
Further, we have $\pi^*(\mu_2\times\mu_2)=\mu_4$ (where $\mu_r$ stands for the uniform Bernoulli measure on $\{0,1,\dots,r-1\}^\N$). Finally, $\pi(\omega_2(y),\omega_2(z))=\omega_4(x)$. The fact that $x$ is normal in base $4$ is equivalent to $\omega_4(x)$ being generic for $\mu_4$. Hence, the pair $(\omega_2(y),\omega_2(z))=\pi^{-1}(\omega_4(x))$ is generic for ${\pi^*}^{-1}(\mu_4)=\mu_2\times\mu_2$. This means that $\omega_2(y)$ and $\omega_2(z)$ are normal and (by Definition~\ref{indd}) independent, as elements of the system $(\{0,1\}^\N,\sigma)$, which further means that $z,y$ are normal in base 2 and $2$-independent.
\end{proof}

For a fixed base $r$, if $x,y\in\mathbb T$ are $R$-independent then the lower bound in Corollary~\ref{rauzy}(1) can be significantly sharpened:

\begin{prop}\label{x+yi}
If $x,y\in\mathbb T$ are $R$-independent then 
$$
\max\{h(x),h(y)\}\le h(x+y).
$$
\end{prop}

\begin{proof} By the definition of independent points, $x$ and $y$ are generic for some \im s $\mu,\nu\in\M(\mathbb T,R)$, respectively, while the pair $(x,y)$ is generic for $\mu\times\nu\in\M(\mathbb T\times\mathbb T,R\times R)$. In particular, $h(x)$ and $h(y)$ are well defined (as, correspondingly, $h(\mu)$ and $h(\nu)$). The point $x+y\in\mathbb T$ is the image of $(x,y)$ via the factor map $(t,u)\mapsto t+u$, $t,u\in\mathbb T$, therefore, by Remark~\ref{mxmy}, $x+y$ is also generic for some measure, and hence $h(x+y)$ is well defined as well. Note that, on the one hand, the factor map $(t,u)\mapsto(t+u,u)$, $t,u\in\mathbb T$, sends $(x,y)$ to $(x+y,y)$, and on the other hand, the factor map $(t,u)\mapsto (t-u,u)$, $t,u\in\mathbb T$, sends $(x+y,y)$ back to $(x,y)$. Using the inequalities \eqref{eof} (two times) and \eqref{eoj}, we obtain:
$$
h(x,y)=h(x+y,y)\le h(x+y)+h(y).
$$
By independence of $x$ and $y$, we also have $h(x,y)=h(x)+h(y)$. So, we have shown that $h(x)\le h(x+y)$. By switching the roles of $x$ and $y$, we also get
$h(y)\le h(x+y)$. 
\end{proof}

\begin{cor}\label{indep}\phantom{.}
\begin{itemize}
\item[(i)]If $x, y\in\mathbb T$ are $R$-independent and $x\in\mathcal N(\mathbb T,R)$ then $x+y\in\mathcal N(\mathbb T,R)$ regardless of $y$ (use $h(x+y)\ge \max\{h(x),h(y)\}=\log r$, and Proposition~\ref{nordet1}~(3)). In particular, the sum of two independent $R$-normal elements of $\mathbb T$ is $R$-normal.
\item[(ii)]If $x, y\in\R$ are $r$-independent and $x\in\mathcal N(r)$ then $x+y\in\mathcal N(r)$ (regardless of $y$). In particular, the sum of two $r$-independent real numbers normal in base~$r$ is normal in base $r$.
\end{itemize}
\end{cor}

Independence is not necessary for the sum of $R$-normal elements of the circle $\mathbb T$ to be $R$-normal. For example, whenever $x$ is $R$-normal then $x+x=2x$ is also $R$-normal (see Proposition~\ref{xq}), while the pair of points $(x,x)$ generates the diagonal joining, which makes them far from independent. 

\begin{ques}\label{q1} It follows, via Remark~\ref{thesame}, from Corollaries~\ref{indep} and~\ref{xq} (see below) that there are two extreme cases when the sum of two real numbers $x,y\in\mathcal N(r)$ belongs to $\mathcal N(r)$: (a) when $x$ and $y$ are $r$-independent, and (b) when $x=qy$ for some rational number $q\neq -1$. Is there a  succinct necessary and sufficient condition for the pair of two numbers $x,y\in\mathcal N(r)$ to have their sum also in $\mathcal N(r)$?  
\end{ques}

For completeness of the picture, we provide a short proof of the following well-known fact:

\begin{prop}\label{any}
Any real number $x$ can be represented in uncountably many different ways as a sum of two numbers normal in base $r$. 
\end{prop}

\begin{proof}
The set $\mathcal N(r)$ is of full Lebesgue measure on $\R$. By invariance of the Lebesgue measure under symmetry and translation, the set $x-\mathcal N(r)=\{x-y: y\in\mathcal N(r)\}$ is also of full Lebesgue measure, which implies that $(x-\mathcal N(r))\cap \mathcal N(r)$ is of full Lebesgue measure. So, there exists uncountably many numbers $x_1\in\mathcal N(r)$ such that $x_1\in x-\mathcal N(r)$ and hence for some $x_2\in\mathcal N(r)$ (depending on the choice of $x_1$) we have $x_1=x-x_2$. Then $x=x_1+x_2$, as required.
\end{proof}

\subsection{Attainability of the bounds for lower and upper entropy}\label{4.3}
We now present a series of examples to illustrate the behavior of lower and upper entropy under addition. In particular, we will show that all the estimates established in Proposition~\ref{x+y} are sharp. We will be utilizing the system $(\mathbb T,R)$, where the map $R$ is given by $t\mapsto rt$, $t\in\mathbb T$, $r\ge2$.

\begin{exam}\label{ea}
We begin with a simple example in which $h(x), h(y)$ and $h(x+y)$ exist and $|h(x)-h(y)|=h(x+y)<h(x)+h(y)$, i.e.,
the entropy of the sum achieves its lower bound (given by  Corollary~\ref{rauzy}(1)) but not the upper bound. Let $r=2$. Let $x\in\mathbb T$ be generic (under the transformation $R$) for a measure of positive entropy $h$ (so, $h(x)=h$). 
The map $x\mapsto-x$ is a \tl\ conjugacy of the system $(\mathbb T, R)$ with itself, hence, for $y=-x$, we have $h(y)=h(x)=h$. Now, $x+y=0$, which is fixed under $R$, and hence $h(x+y)=0$. Thus 
$$
|h(x)-h(y)|=0=h(x+y)<2h=h(x)+h(y).
$$
\end{exam}

\begin{exam}\label{sm}
In this example we deal with the equality 
\begin{equation}\label{==}
h(x+y)=h(x)+h(y).
\end{equation}
Note that the equality \eqref{==} holds if at least one of $x, y$ is $R$-deterministic. It is of interest to inquire whether the equality \eqref{==} can hold when both $x$ and $y$ are not $R$-deterministic. We will answer this question in the positive. Note that if $x$ is not $R$-deterministic and $y$ is $R$-normal then $h(x)+h(y)>\log r$ and \eqref{==} cannot hold. We will show that \eqref{==} can hold  for independent $x,y\notin\mathcal D(\mathbb T,R)\cup\mathcal N(\mathbb T,R)$ as well as for dependent $x',y'\notin\mathcal D(\mathbb T,R)\cup\mathcal N(\mathbb T,R)$.

(i) Take $r=4$ and let $x\in\mathbb T$ be an element such that its alias in base $4$, $\omega_4(x)$, contains only the digits $0,1$, and let $y\in\mathbb T$ be an element such that $\omega_4(y)$ contains only the digits $0,2$. Because both $\omega_4(x)$ and $\omega_4(y)$ use only 2 out of 4 symbols, $x$ and $y$ are not $4$-normal. 
Observe that on such pairs $(x,y)$ the factor map $\mathbb T\times\mathbb T\to\mathbb T$ given by $(t,u)\mapsto t+u$, $t,u\in\mathbb T$, is invertible. Indeed, the digit $1$ occurs in $\omega_4(x)$ precisely at the coordinates where $\omega_4(x+y)$ has digits $1$ or~$3$, and likewise, the digits $2$ in $\omega_4(y)$ occur precisely at the coordinates where $\omega_4(x+y)$ has digits $2$ or~$3$, and so $x+y$ determines the pair $(x,y)$. This implies that whenever $x$, $y$, $x+y$ and the pair $(x,y)$ are generic for some measures $\mu$, $\nu$, $\zeta$ and $\xi$ (the latter is a measure on $\mathbb T\times\mathbb T$), respectively, then the systems $(\mathbb T, \zeta, R)$ and $(\mathbb T\times\mathbb T,\xi, R\times R)$ (here $R$ is given by $t\mapsto 4t$, $t\in\mathbb T$) are isomorphic, and hence
$$
h(x+y) = h(x,y).
$$
Since the digits 0 and 1 in $\omega_4(x)$ (as well as 0 and 2 in $\omega_4(y)$) are distributed completely arbitrarily, we can find elements $x$ and $y$ as above so that $h(x)=h(\mu)>0$ and $h(y)=h(\nu)>0$ (implying that $x,y\notin\mathcal D_{\mathbb T}(4)$), and moreover, by judiciously choosing the positions of the digits in the aliases of $x$ and $y$, we can arrange these aliases to be independent (and hence $x,y$ to be $R$-independent). In this case, we have $h(x,y)=h(x)+h(y)$ and the desired equality $h(x+y)=h(x)+h(y)$ holds. 

(ii) We keep $r=4$. We will construct $R$-dependent $x'$ and $y'$ using $x,y$ from~(i). The element $x'$ is obtained by placing successive digits of $\omega_4(x)$ at even coordinates and filling the odd coordinates with zeros. We create $y'$ analogously using the digits of $\omega_4(y)$. Note that under $\sigma^2$ (the shift by two positions) the \sq\ $\omega_4(x')$ generates a measure $\mu'$ on $\{0,1,2,3\}^\N$ such that the system $(\{0,1,2,3\}^\N,\sigma^2,\mu')$ is isomorphic to system
$(\{0,1,2,3\}^\N,\sigma,\mu)$ ($\mu$ is the measure generated by $\omega_4(x)$ under the shift~$\sigma$). The classical formula for entropy, $h_\mu(T^k)=kh_\mu(T)$ (see, e.g.,~\cite[Fact~2.4.19]{Do}), implies that $h(x')=\frac12 h(x)$. Similarly, $h(y')=\frac12 h(y)$ and $h(x'+y')=\frac12 h(x+y)$, so the equality $h(x'+y')=h(x')+h(y')$ holds. It remains to show that the elements $x'$ and $y'$ are not $R$-independent.  Let $\mu'$, $\nu'$ and $\xi'$ denote the measures generated by $x'$, $y'$ and the pair $(x',y')$, respectively. Then 
$$
\mu'([1])=\tfrac12\mu([1]),\ \  \nu'([2])=\tfrac12\nu([2]) \text{ \ \ and \ \ }\xi'\left(\begin{bmatrix}1\\2\end{bmatrix}\right)=\tfrac12\xi\left(\begin{bmatrix}1\\2\end{bmatrix}\right),\footnote{Here $\begin{bmatrix}1\\2\end{bmatrix}$ denotes the cylinder in 
$(\{0,1,2,3\}\times\{0,1,2,3\})^\N$ corresponding to the block $\begin{matrix}1\\2\end{matrix}$\,, one of 16 blocks of length $1$ over the alphabet $\{0,1,2,3\}\times\{0,1,2,3\}$.}
$$
where $\mu,\nu$ and $\xi$ are, as in the example (i), the measures generated by $x,y$ and the pair $(x,y)$, respectively.
By $R$-independence of $x$ and $y$, we have $\xi\left(\begin{bmatrix}1\\2\end{bmatrix}\right)=\mu([1])\nu([2])$, which implies that $\xi'\left(\begin{bmatrix}1\\2\end{bmatrix}\right)=\tfrac12\mu([1])\nu([2])$, which is strictly larger than $\mu'([1])\nu'([2])$.
\end{exam}

\begin{rem}
In Example~\ref{sm} we were utilizing base $4$. With some extra work, one can create similar examples in base $2$. For instance, to get an example as in Example~\ref{sm}(i), consider elements $x,y\in\mathbb T$ such that $\omega_2(x)$ is built of sufficiently separated (say, by at least 20 zeros) repetitions of the block $11$ while $\omega_2(y)$ is built analogously with the blocks $101$. Then, by inspecting the digits of $\omega_2(x+y)$, one can locate all occurrences of $11$ in $\omega_2(x)$ and all occurrences of $101$ in $\omega_2(y)$, and so $x+y$ determines the pair $(x,y)$. From here, we can argue as in Example~\ref{sm}, including the modification leading to example (ii).
\end{rem}

We are going now to present an example which illustrates the behavior of lower and upper entropy under addition when the entropy of either $x\in\mathbb T$ or $y\in\mathbb T$ does not exist, i.e., either $\underline h(x)<\overline h(x)$ or $\underline h(y)<\overline h(y)$.
We start with introducing a concise notation for the expressions that appear in Proposition~\ref{x+y}:
\begin{align}\label{for}
\begin{split}
    \text{lower bound for $\underline h(x_1+x_2)$:\ \ }&\underline{LB}=\max\{0,\ \underline h(x_1)-\overline h(x_2),\ \underline h(x_2)-\overline h(x_1)\},\\
    \text{upper bound for $\underline h(x_1+x_2)$:\ \ }&\underline{U\!B}=\min\{\log r,\ \underline h(x_1)+\overline h(x_2),\ \overline h(x_1)+\underline h(x_2)\},\\
    \text{lower bound for $\overline h(x_1+x_2)$:\ \ }&\overline{LB}=\max\{|\underline h(x_1)-\underline h(x_2)|,\ |\overline h(x_1)-\overline h(x_2)|\},\\
    \text{upper bound for $\overline h(x_1+x_2)$:\ \ }&\overline{U\!B}=\min\{\log r,\ \overline h(x_1)+\overline h(x_2)\}.
\end{split}
\end{align}
Note that if $\underline{LB}<\underline{U\!B}$ then in the double inequality 
$\underline{LB}\le \underline h(x_1+x_2)\le \underline{UB}$ only one equality can be achieved.
A similar observation applies to $\overline{LB}\le \overline h(x_1+x_2)\le \overline{UB}$ when $\overline{LB}<\overline{U\!B}$. This leads to four extreme cases, and each of them can be demonstrated by an example. We will provide just one, for the most delicate situation when the (smaller) lower entropy achieves its upper bound while the (larger) upper entropy achieves its lower bound. The remaining three examples are similar (and in fact easier). 

\begin{exam} There exist $x_1,x_2\in\mathbb T$ such that $\underline h(x_1+x_2)<\overline h(x_1+x_2)$ and
$$
\underline{LB}<\underline h(x_1+x_2)=\underline{U\!B}\text{ \ while \ }\overline{LB}=\overline h(x_1+x_2)<\overline{U\!B}.
$$

\smallskip\noindent
Let $r=2$. Let $\omega\in\{0,1\}^\N$ be generic for the Bernoulli measure $\mu_{\bar p}$ with $\bar p=(\frac15,\frac45)$ and let $z=\phi_2(\omega)$ (that is, $z\in\mathbb T$ is such that its binary alias, $\omega_2(z)$, matches $\omega$). 
Then 
$$
h(z)=h(-z) =-\tfrac15\log(\tfrac15)-\tfrac45\log(\tfrac45)=: H(\bar p),
$$
which is a positive number smaller than $\frac34\log 2$ (this will be used later). Let 
\begin{equation}\label{sos}
\mathbb S=\bigcup_{n\ge1}\{(2n)!+1,(2n)!+2\dots,(2n+1)!\}.
\end{equation}
Then 
$$
\mathbb S^{\mathsf c}=\{1\}\cup\bigcup_{n\ge1}\{(2n-1)!+1,(2n-1)!+2,\dots,(2n)!\}.
$$
We will also use the periodic set $\mathbb A=3\N$ and its complement $\mathbb A^{\mathsf c}=(3\N-1)\cup(3\N-2)$. 
By \cite[Theorem 4]{K1}, the \sq\ $\omega_2(z)|_\mathbb A$ (the restriction of $\omega_2(z)$ to the periodic \sq\ $\mathbbm 1_{\mathbb A}$) is also generic for $\mu_{\bar p}$ and hence has entropy $H(\bar p)$. Note that under $\sigma^3$ (the shift by three positions) the \sq\ $\omega_2(z)\cdot \mathbbm 1_{\mathbb A}$, where the multiplication of binary \sq s is understood coordinatewise, generates a measure $\mu$ such that the system $(\{0,1\}^\N,\mu,\sigma^3)$ is isomorphic to $(\{0,1\}^\N,\nu,\sigma)$, where
$\nu$ is generated (under $\sigma$) by $\omega_2(z)|_\mathbb A$. Now, the classical formula for entropy, $h_\mu(T^k)=kh_\mu(T)$, implies that the entropy of $\omega_2(z)\cdot \mathbbm 1_{\mathbb A}$ equals one third of the entropy of $\omega_2(z)|_{\mathbb A}$, i.e.,
$$
h(\omega_2(z)\cdot\mathbbm1_{\mathbb A})=\frac{H(\bar p)}3.
$$
By a similar argument, we have $h(\omega_2(-z)\cdot\mathbbm1_{\mathbb A^{\mathsf c}})=\frac{2H(\bar p)}3$. 
We let $x_1\in\mathbb T$ be the element whose binary alias is $\omega_2(z)\cdot\mathbbm1_{\mathbb A}$. Then 
$$
\underline h(x_1)=\overline h(x_1)=\frac{H(\bar p)}3.
$$
We define $x_2\in\mathbb T$ as the element whose alias is $\omega_2(-z)\cdot \mathbbm1_{\mathbb S}$. The alias of $x_2$ is comprised of alternating blocks, of rapidly increasing lengths, coming from the \sq s $\omega_2(-z)$ and $\bar 0$ (the \sq\ of zeros). The values $\underline h(x_2)$ and $\overline h(x_2)$ will be established with the help of the following lemma, whose proof will be given after we complete the example.

\begin{lem}\label{ratunku}
Let $\mathbb S\subset\N$ be the set defined in \eqref{sos}. Suppose $s,t\in\{0,1\}^\N$ are generic (under the shift transformation $\sigma$) for some measures $\mu$ and $\nu$, respectively. Let 
$$
u= s\cdot \mathbbm1_{\mathbb S}+t\cdot\mathbbm1_{\mathbb S^{\mathsf c}}.
$$
Then
$$
\underline h(u)=\min\{h(s), h(t)\} \text{ \ and \ }\overline h(u)=\max\{h(s), h(t)\}.
$$
\end{lem}
By applying Lemma~\ref{ratunku} to $u=\omega_2(x_2)$, we get 
$$
\underline h(x_2)=0 \text{ \ and \ } \overline h(x_2)=H(p).
$$
Substituting the values of $\underline h(x_1), \overline h(x_1), \underline h(x_2), \overline h(x_2)$ into \eqref{for} we obtain
\begin{align*}
&\underline{LB}=\max\{0,\tfrac{H(p)}3-H(p), 0-\tfrac{H(p)}3\}=0,\\
&\underline{U\!B}=\min\{\log2,\ \tfrac{H(p)}3+H(p),\ \tfrac{H(p)}3+0\}=\tfrac{H(p)}3,\\
&\overline{LB}=\max\{|\tfrac{H(p)}3-0|, |\tfrac{H(p)}3-H(p)|\}=\tfrac{2H(p)}3,\\
&\overline{U\!B}=\min\{\log2,\ \tfrac{H(p)}3+H(p)\}=\tfrac{4H(p)}3.
\end{align*} 

Note that coordinatewise addition of binary \sq s with disjoint supports produces binary \sq s. Thus we can write 
$$
\omega_2(x_1) = \omega_2(z)\cdot\mathbbm1_{\mathbb A}\cdot\mathbbm1_{\mathbb S}+\omega_2(z)\cdot\mathbbm1_{\mathbb A}\cdot\mathbbm1_{\mathbb S^{\mathsf c}}
$$ 
and 
$$
\omega_2(x_2) = \omega_2(-z)\cdot\mathbbm1_{\mathbb A}\cdot \mathbbm1_{\mathbb S}+\omega_2(-z)\cdot\mathbbm1_{\mathbb A^{\mathsf c}}\cdot \mathbbm1_{\mathbb S}.
$$

Also note that, whenever $s,t\in\mathbb T$ are such that $\omega_2(s)$ and $\omega_2(t)$ have disjoint supports, then $\omega_2(s+t)=\omega_2(s)+\omega_2(t)$. Since $\omega_2(z)$ and $\omega_2(-z)$ have disjoint supports, so do $\omega_2(x_1)$ and $\omega_2(x_2)$, therefore 
        \begin{multline*}
        \omega_2(x_1+x_2)=\omega_2(z)\cdot\mathbbm1_{\mathbb A}\cdot\mathbbm1_{\mathbb S}+\omega_2(z)\cdot\mathbbm1_{\mathbb A}\cdot\mathbbm1_{\mathbb S^{\mathsf c}}
        +\omega_2(-z)\cdot\mathbbm1_{\mathbb A}\cdot \mathbbm1_{\mathbb S}+\omega_2(-z)\cdot\mathbbm1_{\mathbb A^{\mathsf c}}\cdot \mathbbm1_{\mathbb S}=\\
        =\omega_2(z)\cdot\mathbbm1_{\mathbb A}\cdot\mathbbm1_{\mathbb S^{\mathsf c}}+(\omega_2(-z)\cdot\mathbbm1_{\mathbb A^{\mathsf c}}+\mathbbm1_{\mathbb A})\cdot\mathbbm1_{\mathbb S}.
        \end{multline*}

The rightmost formula shows that the binary alias of $x_1+x_2$ is built of alternating blocks, of rapidly growing lengths, coming from the \sq s $\omega_2(z)\cdot\mathbbm1_{\mathbb A}$ and $\omega_2(-z)\cdot\mathbbm1_{\mathbb A^{\mathsf c}}+\mathbbm1_{\mathbb A}$.
As we have already shown, the measure generated by the \sq\ $\omega_2(z)\cdot\mathbbm1_{\mathbb A}$ has entropy $\frac{H(p)}3$.
Since the \sq s $\omega_2(-z)\cdot\mathbbm1_{\mathbb A^{\mathsf c}}$ and $\mathbbm1_{\mathbb A}$ have disjoint supports, and the periodic \sq\ $\mathbbm1_A$ is deterministic, using Corollary~\ref{rauzy}~(2) we obtain
$$
h(\omega_2(-z)\cdot\mathbbm1_{\mathbb A^{\mathsf c}}+\mathbbm1_{\mathbb A})=h(\omega_2(-z)\cdot\mathbbm1_{\mathbb  A^{\mathsf c}})=\frac{2H(p)}3.
$$
Lemma~\ref{ratunku} now implies that 
$$
\underline h(x_1+x_2)=\frac{H(p)}3 \text{ \ and \ }\overline h(x_1+x_2)=\frac{2H(p)}3,
$$
and the desired relations hold.
\end{exam}

\begin{proof}[Proof of Lemma~\ref{ratunku}]
On the interval $[1,(2n+1)!]$, $u$ differs from $s$ only on the subinterval $[1,(2n)!]$. Since $\frac{(2n)!}{(2n+1)!}\to0$, this difference becomes negligible for large $n$, implying that along the \sq\ $\mathcal J_1=((2n+1)!)_{n\ge1}$, $u$ generates the same measure as does $s$, i.e., the measure $\mu$. By a similar argument, along the \sq\ $\mathcal J_2=((2n)!)_{n\ge1}$, $u$ generates $\nu$. Now, let $\mathcal J=(j_k)_{k\ge1}$ be any \sq\ along which $u$ generates some measure. By passing to a sub\sq, we may assume that either all $j_k$ fall in the intervals of the form $[(2n)!+1,(2n+1)!]$ or all $j_k$ fall in the intervals of the form $[(2n-1)!+1,(2n)!]$. Suppose that the first case holds (the argument for the other case is identical) and for each $k\ge1$ denote by $n_k$ the integer such that 
$j_k\in[(2n_k)!+1,(2n_k+1)!]$. By passing to a sub\sq\ again, we can assume that the fractions
$$
\frac{(2n_k)!}{j_k}
$$
tend to a number $\alpha\in[0,1]$. Then, for large $k$, the block $u|_{[1,j_k]}$ occurring in $u$ over the interval $[1,j_k]$ is a concatenation of the blocks $B_k^{(1)}=u|_{[1,(2n_k)!]}$ and $B_k^{(2)}=u|_{[(2n_k)!+1,j_k]}=s|_{[(2n_k)!+1,j_k]}$. The numbers $(2n_k)!$ form a sub\sq\ of $\mathcal J_2$ so the blocks $B_k^{(1)}$ generate the measure $\nu$. If the blocks $B_k^{(2)}$ have bounded lengths, they can be ignored and $u$ is generic for $\nu$. Otherwise, by passing to a sub\sq\ one last time we may assume that the blocks $B_k^{(2)}$ generate  some measure $\xi$. Since $s$ is generic for $\mu$, the blocks $B_k^{(3)}=s|_{[1,j_k]}$ generate $\mu$. The same holds for the blocks $B_k^{(4)}=s|_{[1,(2n_k)!]}$. 
But $B_k^{(3)}$ is a concatenation of the blocks $B_k^{(4)}$ and $B_k^{(2)}$ where the proportion of lengths
$\frac{|B_k^{(4)}|}{|B_k^{(3)}|}$ tends to $\alpha$. Therefore $\mu=\alpha\mu+(1-\alpha)\xi$ (see Remark~\ref{uo}). Clearly, this implies that either $\alpha=1$ or $\xi=\mu$. 
Eventually, since $u|_{[1,j_k]}$ is a concatenation of the blocks $B_k^{(1)}$ (approximating $\nu$) and $B_k^{(2)}$ (approximating $\mu$, unless $\alpha=1$) and the fractions of lengths $\frac{|B_k^{(1)}|}{|B_k|}$ tend to $\alpha$,  
the measure generated by $u$ along $\mathcal J$ equals $\alpha\nu+(1-\alpha)\mu$ (also when $\alpha=1$). 

By the affinity property of entropy (see, e.g.,~\cite[Theorem~2.5.1]{Do}), we obtain 
$$
h(\alpha\nu+(1-\alpha)\mu)=\alpha h(\nu)+(1-\alpha)h(\mu),
$$
which is a number between $h(\nu)$ and $h(\mu)$. This completes the proof of the lemma.
\end{proof}

\subsection{Multiplication by rationals preserves lower and upper entropy}\label{qx}

It has been proved by Wall \cite{Wa} that if $x\in\R$ is normal in base $r$ and $q\neq0$ is rational then $qx$ is normal in base $r$. It is also true that if $q\neq0$  is rational and $y\in\R$ is deterministic in base $r$ then so is $qy$. Indeed, note that, for any real numbers $x,y,q$, $q\neq0$, we have
$$
qy+x = q(y+\tfrac1qx).
$$
Assume now that $x$ is normal in base $r$ and $q$ is rational. By Wall's theorem $\frac1qx$ is normal in base~$r$. Assuming in addition that $y$ is deterministic in base $r$, we have, by (the necessity in) Rauzy theorem (Theorem~\ref{rrrr}) that $y+\frac1qx$ is normal in base~$r$. Applying Wall's theorem again, we get that $q(y+\frac1qx)=qy+x$ is normal in base $r$. By the sufficiency in Rauzy theorem, $qy$ is deterministic. 

We will now demonstrate that the above facts have deeper dynamical underpinnings. In view of Remark~\ref{thesame}, it is natural (and adequate) to work with the system $(\mathbb T,R)$ where $R(t)=rt$. 

We start with the trivial observation that since multiplication by an integer can be defined in terms of addition and negation,
and the passage $x\mapsto\{x\}$ from $\R$ to $\mathbb T$ is a group homomorphism, we have
$n\{x\}=\{nx\}$ for any $n\in\N$.

Now, division of an element $x\in\mathbb T$ by a positive integer $m$ has multiple outcomes, as there are multiple elements $y\in\mathbb T$ such that $my = x$. We will be using the following notation: for $x\in\mathbb T$ and a rational number $q=\frac nm$, by $qx$ we will denote any element $y\in\mathbb T$ such that $my = nx$. It will be clear from the context, that this ambiguity does not affect the correctness of our statements and proofs. 


\begin{prop}\label{xq} Consider the system $(\mathbb T,R)$ where $R$ is the map $t\mapsto rt$, $t\in\mathbb T$.
Let $q$ be any nonzero rational number. Then 
\begin{enumerate}
\item For any $x\in\mathbb T$ we have
$$
\underline h(qx)=\underline h(x), \text{\ \ and\ \ } \overline h(qx)=\overline h(x).
$$
\item In particular, if $x$ is deterministic or normal then $qx$ is, respectively, deterministic or normal.
\end{enumerate} 
\end{prop}
\begin{proof}Statement (2) follows, with the help of  Proposition~\ref{nordet1}(2) and (3), from statement (1). It remains to prove (1).

First, we will show that, for $n\in\Z\setminus\{0\}$, the mapping $t\mapsto nt$, $t\in\mathbb T$, preserves lower and upper entropy. Observe that this mapping is a \tl\ factor map from $(\mathbb T,R)$ to itself, and hence it sends $\M_x$ onto $\M_{nx}$. Moreover, this map preserves entropy of \im s, since it is finite-to-one (see e.g.,  \cite[Theorem~2.1]{LW}).  

As a consequence, the sets of entropy values $\{h(\mu):\mu\in\M_x\}$ and $\{h(\mu):\mu\in\M_{nx}\}$ coincide. In view of Definition~\ref{lue}, it follows that
\begin{equation}\label{hxn}
\underline h(nx)=\underline h(x) \text{\ \ and\ \ } \overline h(nx)=\overline h(x).
\end{equation}

Now let $q=\frac nm$ be rational with $m\in\N$ and let $y=qx$, i.e., $y\in\mathbb T$ satisfies $my = nx$. By \eqref{hxn} we obtain
\begin{align*}\label{hxn1}
\underline h(qx)&=\underline h(y)=\underline h(my)=\underline h(nx)=\underline h(x) \text{\ \ and\ \ }\\ 
\overline h(qx)&=\overline h(y)=\overline h(my)=\overline h(nx)=\overline h(x).
\end{align*}
\end{proof}

\begin{cor}\label{affin}
Let $q\neq0$ be rational and let $y\in\R$ be deterministic in base~$r$. Then the mapping $L_{q,y}:\R\to\R$ given by $x\mapsto qx+y$ preserves both lower and upper entropy of real numbers. In particular, it preserves both normality and determinism in base $r$.
\end{cor}

\begin{rem}\label{group}
Recall that by Corollary~\ref{rrr}(2), the set $\mathcal D(r)$ of numbers deterministic in base $r$ is a group. The family $\mathcal L_{\text{rat,det}}=\{L_{q,y}: q\in\Q\setminus\{0\}, \,y\in\mathcal D(r)\}$ is also a group. Indeed, 
$$
L_{q,y}^{-1}=L_{\tfrac1q,-y\tfrac1q}. 
$$
Since  $-y\tfrac1q$ is deterministic by Corollary~\ref{affin}, we have $L_{q,y}^{-1}\in\mathcal L_{\text{rat,det}}$.
Further, if $q'\in\Q\setminus\{0\}, \,y'\in\mathcal D(r)$ then
$$
(L_{q',y'}\circ L_{q,y})(x) = (qx+y)q'+y' = qq'x + q'y +y'.
$$
Now, $qq'\in\Q\setminus\{0\}$, while $q'y+y'\in\mathcal D(r)$ by Corollary~\ref{affin}. Thus,
$$
L_{q',y'}\circ L_{q,y} = L_{qq',q'y+y'}\in\mathcal L_{\text{rat,det}}.
$$
Finally, note that the maps in $\mathcal L_{\text{rat,det}}$ preserve not only normality and determinism in base $r$ but also, by Theorem~\ref{goal}, $r$-independence of numbers normal in base $r$.
\end{rem}

\begin{rem}Proposition~\ref{xq} allows to prove that a normal number $x$ plus a rational number $q$ is normal, without referring to the more complicated Proposition~\ref{x+y}. Indeed, this fact is trivial if $q=0$. Otherwise, $x+q = q(\frac xq+1)$, where $\frac xq$ is normal by Proposition~\ref{xq}, addition of $1$ does not affect normality, and multiplication by $q$ preserves normality by Proposition~\ref{xq} again.
\end{rem}

\subsection{The ``sufficiency'' in Rauzy theorem}\label{rev}
In this subsection, we will provide a proof of the sufficiency direction in the Rauzy theorem (Theorem~\ref{rrrr}). Unlike the proof of necessity (Corollary~\ref{rrr}(1)) which employs notions of joinings, factors, and entropy, the proof of sufficiency relies mostly on techniques of harmonic analysis. 

\begin{thm}\label{rauzy1} \emph{(Rauzy)} For any base $r\ge2$ we have $\mathcal N^\perp(r)\subset\mathcal D(r)$.
\end{thm}

\begin{proof}The proof is essentially the same as that of \cite[Lemma 4]{Ra}.
First of all, by Remark~\ref{thesame}, it suffices to conduct the proof in the framework of the system $(\mathbb T,R)$ where $R(t)=rt$, $t\in\mathbb T$. It will be convenient to pass to a \tl ly conjugate model $(\mathsf T,\mathsf R)$ of $(\mathbb T,R)$, where $\mathsf T$ is the unit circle in the complex plane, i.e., $\mathsf T=\{z:|z|=1\}$ and $\mathsf R$ is given by $z\mapsto z^r$, $z\in\mathsf T$. An element $z\in\mathsf T$ corresponds to an element of $\mathcal N(\mathbb T,R)$ if and only if it is generic under $\mathsf R$ for the normalized Lebesgue measure on $\mathsf T$ (which we keep denoting by $\lambda$). In this case we will say that $z$ is $\mathsf R$-normal. Likewise, an element $z\in\mathsf T$ corresponds to an element of $\mathcal D(\mathbb T,R)$ if and only if it is deterministic in the system $(\mathsf T,\mathsf R)$ (we will then say that $z$ is $\mathsf R$-deterministic).

We need to show that if $y\in\mathsf T$ has the property that $xy$ is $\mathsf R$-normal for any $\mathsf R$-normal $x\in\mathsf T$ then $y$ is $\mathsf R$-deterministic. In other words, we need to show that any measure $\nu$ generated (under $\mathsf R$) by $y$ along any sub\sq\ $\mathcal J=(n_k)_{k\ge1}$ has entropy zero. This will be done by showing that $\nu$ is disjoint from $\lambda$. Indeed, since $h(\lambda)=\log r>0$ and two measures of positive entropy are never disjoint\footnote{
According to the well-known Sinai's factor theorem \cite{S1}, any system of positive entropy $h$ has a Bernoulli factor of any entropy less than or equal to $h$. So, two systems of positive entropy have a common nontrivial factor, and hence are not disjoint.\label{foo6}}, the disjointness will imply that $h(\nu)=0$.

In order to show that $\lambda$ and $\nu$ are disjoint, we will verify that for any pair of continuous complex functions $f, g$ on $\mathsf T$ and any joining $\xi$ of $\lambda$ and $\nu$, we have 
\begin{equation}\label{pp}
\int f(z_1)g(z_2)\,d\xi(z_1,z_2)=\int f(z)\,d\lambda(z) \cdot \int f(z)\,d\nu(z).
\end{equation}
Clearly, it suffices to show \eqref{pp} for a linearly uniformly dense family of continuous functions, and we will choose the family of characters $\chi_n$ given by $\chi_n(z)= z^n$, $z\in\mathsf T$, $n\in\Z$. This reduces the problem to showing that
\begin{equation}\label{qq}
\forall_{n,m\in\Z}\ \ \int \chi_n(z_1)\chi_m(z_2)\,d\xi(z_1,z_2)=\int \chi_n(z)\,d\lambda(z) \cdot \int \chi_m(z)\,d\nu(z).
\end{equation}
Note that since $\chi_0\equiv 1$, the equation \eqref{qq} holds trivially if either $n=0$ or $m=0$. Now assume that $n\neq0$ and $m\neq0$. Since, for $n\neq0$, we have $\int \chi_n(z)\,d\lambda(z)=0$, the right hand side of \eqref{qq} equals $0$ and the problem reduces to showing that the left hand side of \eqref{qq} vanishes:
\begin{equation}\label{qq1}
\int\chi_n(z_1)\chi_n(z_2)d\xi(z_1,z_2)=\int z_1^nz_2^m\,d\xi(z_1,z_2)=0.
\end{equation}

By a result of Kamae (see \cite[Theorem 2]{K1}), there exists an $\mathsf R$-normal element $x\in\mathsf T$ such that the pair $(x,y)$ generates $\xi$ (under $\mathsf R\times\mathsf R$) along a sub\sq\ of $\mathcal J=(n_k)_{k\ge1}$. For brevity, we will denote this sub\sq\ again by $(n_k)_{k\ge1}$. So, for any continuous function $F$ on $\mathsf T\times\mathsf T$ we have
$$
\lim_k\frac1{n_k}\sum_{j=0}^{n_k-1}F(x^{jr},y^{jr}) = \int F\,d\xi.
$$
In particular, for $F(z_1,z_2)=z_1^nz_2^m$, we obtain
\begin{equation}\label{eq:1}
\int z_1^nz_2^m\,d\xi(z_1,z_2)=\lim_k\frac1{n_k}\sum_{j=0}^{n_k}x^{njr}y^{mjr}.
\end{equation}

Since $x$ is $\mathsf R$-normal, Proposition~\ref{xq} implies that so is $x^{\frac nm}$. Recall that $y$ is assumed to have the property that $xy$ is $\mathsf R$-normal for any $\mathsf R$-normal $x\in\mathsf T$. Thus $x^{\frac nm}y$ is $\mathsf R$-normal, i.e., it generates $\lambda$ (under $\mathsf R$). Hence, for any continuous function $f$ on $\mathsf T$, we have
\begin{equation}\label{ee}
\lim_k\frac1k\sum_{j=0}^{k-1}f((x^{\frac nm}y)^{jr}) = \int f\,d\lambda.
\end{equation}
Taking $f(z)=z^m$ and observing that \eqref{ee} holds also along the sub\sq\ $(n_k)_{k\ge1}$, we get
\begin{equation}\label{eq:2}
\lim_k\frac1{n_k}\sum_{j=0}^{n_k-1}x^{njr}y^{mjr} = \int z^m\,d\lambda(z) = 0.
\end{equation}
Combining \eqref{eq:1} with \eqref{eq:2} we obtain the desired equality \eqref{qq1}.
\end{proof}

\section{Multidimensional Rauzy theorem}\label{S5}
The main result of this section, Theorem~\ref{k-dim}, generalizes Rauzy theorem (Theorem~\ref{rrrr}) to vectors in $\R^m$. Such a generalization can be also derived from our Theorem~\ref{raufg}, but the proof in this section is much more straightforward.

\begin{defn}\label{mve}Let $m\in\N$ and let $\bar r=(r_1,r_2,\dots,r_m)$ with $r_i\in\N,\ r_i\ge2$, $i\in\{1,2,\dots,m\}$. By the alias of a vector $\bar x=(x_1,x_2,\dots,x_m)\in\R^m$ in base $\bar r$ we will understand a ``multirow'' \sq\ $\omega_{\bar r}(\bar x)$ having $m$ rows, where for each $i\in\{1,2,\dots,m\}$, the $i$th row is comprised of the alias of $x_i$ in base $r_i$.
\end{defn}
Occasionally we will find it convenient to identify the multirow \sq s $\omega_{\bar r}(\bar x)$ appearing in the above definition with \sq s over the alphabet 
$$
\Lambda_{\bar r}=\{0,1,\dots,r_1\}\times\{0,1,\dots,r_2\}\times\cdots\times\{0,1,\dots,r_m\},
$$
where each element of the alphabet $\Lambda_{\bar r}$ is viewed as a column of height $m$. 

\begin{defn}\label{k-nor} \phantom{a}
\begin{itemize}
    \item A vector $\bar x=(x_1,x_2,\dots,x_m)\in\R^m$ will be called \emph{normal in base $\bar r$} if every block $\bar B=(\bar b_1,\bar b_2\dots,\bar b_k)$ with $\bar b_j\in\Lambda_{\bar r}$, $j\in\{1,2,\dots,k\}$, appears in $\omega_{\bar r}(\bar x)$ with frequency $(r_1r_2\cdots r_m)^{-k}$. The set of vectors normal in base $\bar r$ will be denoted by $\mathcal N(\bar r)$.
    \item A vector $\bar y\in\R^m$  \emph{preserves normality in base $\bar r$} if $\bar x +\bar y$ is normal in base $\bar r$ for every $\bar x$ normal in base $\bar r$. The set of vectors that preserve normality in base $\bar r$ will be denoted by $\mathcal N^\perp(\bar r)$.
   \item A vector $\bar y=(y_1,y_2,\dots,y_m)\in\R^m$ is \emph{deterministic in base $\bar r$} if,  for each $i\in\{1,2,\dots,m\}$, $y_i$ is deterministic in base $r_i$.
   The set of vectors deterministic in base $\bar r$ will be denoted by $\mathcal D(\bar r)$.
\end{itemize}
\end{defn}

\begin{rem}\label{above} The following useful observations are straightforward:
\begin{enumerate}
    \item A vector $\bar x$ is normal in base $\bar r$ if and only if its alias in base $\bar r$, $\omega_{\bar r}(\bar x)$, is generic for the uniform Bernoulli measure in the symbolic system $(\Lambda_{\bar r}^\N,\sigma)$.
    \item A vector $\bar x=(x_1,x_2,\dots,x_m)\in\R^m$ is normal in base $\bar r$ if and only if the vector of fractional parts, $\{\bar x\}=(\{x_1\},\{x_2\},\dots,\{x_m\})$, is generic for the $m$-dimensional Lebesgue measure on $\mathbb T^m$ in the system $(\mathbb T^m,\bar R)$, where $\bar R$ is given by $\bar R(t_1,t_2,\dots,t_m)=(r_1t_1,r_2t_2,\dots,r_mt_m)$, $(t_1,t_2,\dots,t_m)\in\mathbb T^m$. 
    \item The $m$-dimensional Lebesgue measure on $\mathbb T^m$ is $\bar R$-\inv, has entropy $\sum_{i=1}^m\log r_i$, and is the unique measure of maximal entropy (this follows by the same argument as in Remark~\ref{umome} using the factor map between the the symbolic system $(\Lambda_{\bar r}^\N,\sigma)$ and $(\mathbb T^m,\bar R)$ which sends the uniform Bernoulli measure to the $m$-dimensional Lebesgue measure on $\mathbb T^m$). Thus, a vector $\bar x$ is normal in base $\bar r$ if and only if $h(\{\bar x\})=\sum_{i=1}^m\log r_i$ (in the system $(\mathbb T^m,\bar R)$).
    \item A vector $\bar x=(x_1,x_2,\dots,x_m)\in\R^m$ is normal in base $\bar r$ if and only if, for each $i=1,2\dots,m$, $x_i$ is normal in base $r_i$ and the fractional parts $\{x_i\}$, viewed as elements of the respective systems $(\mathbb T, R_i)$ with $R_i$ defined by $t\mapsto r_it$, $t\in\mathbb T$, are independent (see Definition~\ref{indd}). 
    \item A vector $\bar y\in\R^m$ is deterministic in base $\bar r$ if and only if its alias in base $\bar r$, $\omega_{\bar r}(\bar y)$, is deterministic in the symbolic system $(\Lambda_{\bar r}^\N,\sigma)$.
    \item A vector $\bar y\in\R^m$ is deterministic in base $\bar r$ if and only if the vector of fractional parts, $\{\bar y\}=(\{y_1\},\{y_2\},\dots,\{y_m\})$, is deterministic in $(\mathbb T^m,\bar R)$, if and only if $h(\{\bar y\})=0$.
\end{enumerate}
\end{rem}

\begin{thm}\label{k-dim}
A vector $\bar y$ is deterministic in base $\bar r$ if and only if, for any $\bar x\in\mathcal N(\bar r)$ one has $\bar x+\bar y\in\mathcal N(\bar r)$. That is,
$$
\mathcal D(\bar r)=\mathcal N^\perp(\bar r).
$$
\end{thm}

\begin{proof}
Let $\bar x=(x_1,x_2,\dots,x_m)$ be normal in base $\bar r$. Then, by Remark~\ref{above}(3), $h(\{\bar x\})=\sum_{i=1}^m\log r_i$. Let $\bar y$ be deterministic in base $\bar r$.
Since 
$$
\{\bar x\}=(\{\bar x\}+\{\bar y\})+(-\{\bar y\}), 
$$
by the same argument as in the proof of Theorem~\ref{x+y}(a) and (b), we have 
\begin{equation}\label{tt}
h(\{\bar x\}+\{\bar y\})-h(-\{y\})\le h(\{\bar x\})\le h(\{\bar x\}+\{\bar y\})+h(-\{y\}).
\end{equation}
By Remark~\ref{minus}, $-\bar y$ is deterministic, and hence by Remark~\ref{above}(6), $h(-\{\bar y\})=0$. Now, by \eqref{tt}, we get $h(\{\bar x\}+\{\bar y\})=h(\{\bar x\})=\sum_{i=1}^m\log r_i$, which, by Remark~\ref{above}~(3) implies normality of $\{\bar x\}+\{\bar y\}$ in base~$\bar r$.

In the opposite direction, if $\bar y$ preserves normality in base $\bar r$ then, for each $i\in\{1,2,\dots,m\}$ and any $x_i$ normal in base $r_i$, $x_i+y_i$ is normal in base $r_i$. By Rauzy theorem (see Theorem~\ref{rrrr}), we get that, for any $i\in\{1,2,\dots,m\}$, $y_i$ is deterministic in base $r_i$ and so $\bar y$ is deterministic in base~$\bar r$. 
\end{proof}

\begin{rem}The goal of this remark is to explain that Theorem~\ref{k-dim} is a nontrivial generalization of the Rauzy theorem (Theorem~\ref{rrrr}).
In view of Remark~\ref{above}(1) and (5), normality and determinism in base $\bar r$ of vectors in $\R^m$ are equivalent to, respectively, normality and determinism of their aliases in the symbolic space $\Lambda_{\bar r}^\N$ (where $\Lambda_{\bar r}$ has $r=r_1r_2\cdots r_m$ symbols). By labeling the elements of $\Lambda_{\bar r}$ as $\{0,1,\dots,r-1\}$ (in any order), the \sq s in the symbolic space $\Lambda_{\bar r}^\N$ can be interpreted as aliases of real numbers in base $r$, and it is tempting to try to interpret Theorem~\ref{k-dim} as a special case of Theorem~\ref{rrrr}. This however does not work since addition of vectors in $\R^m$, $m>1$, does not correspond to the addition of numbers with aliases described above.
\end{rem}

For $m>1$, Theorem~\ref{k-dim} has an interesting corollary which roughly says that addition of deterministic numbers preserves independence of normal numbers.

\begin{thm}\label{goal}
Let $x_i$ be normal in base $r_i$, $i\in\{1,2,\dots,m\}$, and suppose that the fractional parts $\{x_i\}$ are independent (as elements of the respective systems $(\mathbb T,R_i)$, where $R_i$ is given by $t\mapsto r_it$, $t\in\mathbb T$). Let $y_i$ be deterministic in base $r_i$, $i\in\{1,2,\dots,m\}$. Then the numbers $\{x_i+y_i\}$ regarded as elements of the systems $(\mathbb T,R_i)$, are independent.
\end{thm}

\begin{proof}
By Remark~\ref{above}(4), the vector $\bar x=(x_1,x_2,\dots,x_m)$ is normal in base $\bar r=(r_1,r_2,\dots,r_m)$, while, directly by Definition~\ref{k-nor}, the vector $\bar y=(y_1,y_2,\dots,y_m)$ is deterministic in base $\bar r$. Theorem~\ref{k-dim} implies that the vector $\bar x+\bar y$ is normal in base $\bar r$, which, again via Remark~\ref{above}(4), concludes the proof of the theorem.
\end{proof}

\section{Generalizations to endomorphisms of compact metric groups}\label{S4}

As it was revealed in the previous sections, Rauzy theorem (Theorem~\ref{rrrr}) has natural dynamical underpinnings and it is of independent interest to establish a general ergodic framework for dealing with various aspects of normality and determinism.
In this section we extend some of the results obtained in Section~\ref{Sn}, in particular Proposition~\ref{x+y}, Corollary~\ref{rauzy}, Theorem~\ref{k-dim}, Proposition~\ref{xq} and partly Theorem~\ref{rauzy1}, to a more general setup.
We want to stress that unlike Sections~\ref{S3} and~\ref{Sn}, which were 
geared towards Rauzy-like theorems in $\R$ and $\R^n$, this section focuses 
on  phenomena associated with dynamics on compact groups. 

The generalizations obtained in this section are of two-fold nature. First, we deal with dynamics induced by ergodic endomorphisms of arbitrary infinite compact metrizable groups, and second, we employ general averaging schemes which involve F\o lner \sq s in the (amenable) semigroup $(\N,+)$. 

This section is comprised of four subsections. In Subsection~\ref{5.1} we introduce 
the background material concerning F\o lner \sq s in $\N$ (viewed as an additive semigroup) and define the notion of determinism along a F\o lner \sq. In Subsection~\ref{fnor} we define normality along a F\o lner \sq\ and generalize Proposition~\ref{x+y}, Corollary~\ref{rauzy} (in particular, ``necessity'' in Rauzy theorem) and Proposition~\ref{xq} to finite entropy ergodic endomorphisms of compact metrizable groups. In Subsection~\ref{sred},~\ref{eee} and~\ref{4.1} we prove generalizations of ``sufficiency'' for some classes of endomorphisms of compact groups including toral endomorphisms. 


\subsection{Determinism along a F\o lner \sq}\label{5.1}

\begin{defn}\label{folner}
A \sq\ of finite subets of $\N$, $\F=(F_n)_{n\ge1}$, is called \emph{a F\o lner \sq} if
\begin{equation}\label{fol}
\lim_{n\to\infty}\frac{|F_n\cap(F_n+1)|}{|F_n|}=1.
\end{equation}
\end{defn}

Note that in general the sets $F_n$ are not required to be nested nor to cover $\N$.

Let \xt\ be a \tl\ \ds. Let $\F=(F_n)_{n\ge1}$ be a F\o lner \sq\ in $\N$. Let $\mu$ be any probability measure of $X$. A point $x\in X$ is called $\F$-generic for $\mu$ if  
$$
\lim_n \frac1{|F_n|}\sum_{i\in F_n}\delta_{T^ix} =\mu \text{ \ \ (in the weak* topology).}
$$
Points $x\in X$ for which this convergence holds along a sub\sq\ $(n_k)_{k\ge1}$ are called \emph{$\F$-quasi-generic for $\mu$}. Given a point $x\in X$, the set of measures which are $\F$-quasi-generated by $x$ will be denoted by $\M_\F(x)$. By compactness of the weak* topology on the set of probability measures on $X$, $\M_\F(x)$ is nonempty for any $x\in X$. Due to the F\o lner property \eqref{fol}, all measures in $\M_\F(x)$ are $T$-\inv. 

\begin{rem}\label{sim}
Note that a point $x\in X$ is \emph{$\F$-generic for $\mu$} if and only if $\M_\F(x)=\{\mu\}$.
\end{rem}

\begin{defn}Let $(X,T)$ be a \ds\ and let $\F$ be a F\o lner \sq\ in $\N$.
The \emph{$\F$-lower} and \emph{$\F$-upper entropies of a point} $x\in X$ are defined as follows:
$$
\underline h_\F(x) = \inf\{h(\mu):\mu\in\M_\F(x)\}, \ \ \ \bar h_\F(x) = \sup\{h(\mu):\mu\in\M_\F(x)\}.
$$
Clearly, $\underline h_\F(x)\le\bar h_\F(x)$. In case of equality we denote the common value by $h_\F(x)$ and call it the \emph{$\F$-entropy} of $x$.
\end{defn}

We can now define the notion of $\F$-determinism in any \ds\ $(X,T)$:

\begin{defn}\label{Fdeter}
Let $(X,T)$ be a \ds\ and let $\F$ be a F\o lner \sq\ in $\N$.  
A point $x\in X$ is $\F$-deterministic if $\bar h_\F(x)=0$. 
\end{defn}

\subsection{Normality along a F\o lner \sq}\label{fnor}

Let $X$ be an infinite compact metric group and let $\lambda_X$ denote the normalized Haar measure on $X$.\footnote{On a compact metric group the normalized left and right Haar measures coincide, see e.g., \cite[Theorem~15.13]{HR}.}
A homomorphism $T:X\to X$ is called an \emph{endomorphism} if it is continuous and surjective. The dynamical system $(X,T)$ will be called an \emph{algebraic system}.
By surjectivity of $T$ and uniqueness of the Haar measure, $T$ preserves $\lambda_X$. 
We will say that $T$ is \emph{ergodic} if $\lambda_X$ is ergodic with respect to $T$. 

Throughout the rest of this section we will assume that $X=(X,+)$ is an infinite compact metric group. We will use the additive notation since in Subsections~\ref{sred} and~\ref{4.1} we will be dealing with Abelian groups. However, the theorems of this subsection are valid without the commutativity assumption.

\begin{prop}\label{ie} Let $(X,T)$ be an ergodic algebraic system. Then the Haar measure $\lambda_X$ has positive (possibly infinite) entropy. If $\lambda_X$ has finite entropy then $\lambda_X$ is the unique measure of maximal entropy.
\end{prop}

\begin{proof}
The first claim of the theorem was proved by S.\ A.\ Juzvinski\u\i\ in \cite{Ju}. 
The second claim in the case of automorphisms was proved by K.\ Berg in \cite{B}, so we only need to make a reduction to the invertible case. This is done using the standard technique of natural extensions. Let 
$$
\bar X = \{(x_n)_{n\in\Z}:\forall_{n\in\Z}\  x_{n+1}=T(x_n)\}\subset X^\Z.
$$
The space $\bar X$, equipped with the coordinatewise addition and the product topology, is a compact metrizable group and if $\bar T$ denotes the left shift transformation,
given by $\bar T((x_n)_{n\in\Z})=(x_{n+1})_{n\in\Z}$, then the projection $\pi_0$ on the zero coordinate is a factor map from the system $(\bar X, \bar T)$ onto $(X,T)$ (this is where surjectivity of $T$ is necessary). The system $(\bar X,\bar T)$ is called the \emph{natural extension} of $(X,T)$. The mapping $\bar T$ is an automorphism of $\bar X$, therefore it preserves the Haar measure $\lambda_{\bar X}$ on $\bar X$. The map $\pi^*_0:\M_{\bar T}(\bar X)\to\M_T(X)$ given by 
$$
\pi_0^*(\bar\mu)(A)=\bar\mu(\pi_0^{-1}(A)), \ \text{where $A$ is a Borel subset of $X$, $\bar\mu\in\M_{\bar T}(\bar X)$}
$$
is surjective (see Section~\ref{S2}). The natural extension preserves ergodicity and entropy, i.e., $\bar\mu$ is ergodic if and only if $\pi_0^*(\bar\mu)$ is ergodic (see, e.g.,~\cite[Theorem 1, page 241]{KFS}) and $h(\pi_0(\bar\mu))=h(\bar\mu)$, for any $\bar T$-\im\ $\bar\mu$ (see, e.g.,~\cite[Fact 6.8.12]{Do}).
Clearly, $\pi_0(\lambda_{\bar X})=\lambda_X$ and since $\lambda_X$ is ergodic, so is $\lambda_{\bar X}$. If we assume that $\lambda_X$ has finite entropy, then $\lambda_{\bar X}$ has finite entropy as well. So, we have made a reduction to the invertible case. The result in question now follows from \cite[Corollary~1.1]{B} and \cite[Theorem~2.1]{B}, where it is proved for automorphisms. 
\end{proof}

\begin{defn}\label{Fnor} Let $(X,T)$ be an algebraic system. Let $\F$ be a F\o lner \sq\ in~$\N$.
A point $x\in X$ is \emph{$\F$-normal} if it is $\F$-generic for the Haar measure~$\lambda_X$. 
We denote
\begin{itemize}
\item $\mathcal N_\F(X,T)$ -- the set of $\F$-normal elements in the system $(X,T)$,
\item $\mathcal D_\F(X,T)$ -- the set of $\F$-deterministic elements in the system $(X,T)$,
\item $\mathcal N^\perp_\F(X,T)=\{y\in X: \forall_{x\in \mathcal N_\F(X,T)}\  x+y\in\mathcal N_\F(X,T)\}$  (the set of \emph{$\F$-normality preserving} elements in $(X,T)$).
\item $\mathcal D^\perp_\F(X,T)=\{y\in X: \forall_{x\in \mathcal D_\F(X,T)}\  x+y\in\mathcal D_\F(X,T)\}$  (the set of \emph{$\F$-determinism preserving} elements in $(X,T)$).
\end{itemize}
\end{defn}

\begin{ques}\label{q0}
Suppose we define $\mathcal N^\perp_\F(X,T)$ and $\mathcal D^\perp_\F(X,T)$ using $y+x$ (instead of $x+y$). Would these be, correspondingly, the same notions?
\end{ques}

Note that if the measure-preserving system $(X,\lambda_X,T)$ has finite entropy then $x$ is $\F$-normal if and only if $\underline h_\F$ attains at $x$ its maximal value on $X$. Observe also that $\mathcal N^\perp_\F(X,T)$ is an \inv\ subgroup of $X$. 

We can now formulate a general version of the main two results of Subsection~\ref{3.1}. The proofs are straightforward adaptations of the corresponding proofs in that subsection and will be omitted. 

\begin{prop}\label{x+y1}{\rm(cf. Proposition~\ref{x+y})} Let $(X,T)$ be an ergodic algebraic system. Let $\F$ be a F\o lner \sq\ in $\N$. If $(X,T)$ has finite \tl\ entropy then 
\begin{multline}\label{inn}
\max\{0, \underline h_\F(x)-\overline h_\F(y), \underline h_\F(y)-\overline h_\F(x)\}\le \underline h_\F(x+y)\le\\
\min\{\htop(X,T),\underline h_\F(x)+\overline h_\F(y), \overline h_\F(x)+\underline h_\F(y)\},
\end{multline}
\begin{multline}\label{innn}
\max\{|\underline h_\F(x)-\underline h_\F(y)|,\ |\overline h_\F(x)-\overline h_\F(y)|\}\le\overline h_\F(x+y)\le\\
\min\{\htop(X,T),\overline h_\F(x) + \overline h_\F(y)\}.
\end{multline}
\end{prop}

\begin{cor}\label{rauzy2}{\rm(cf. Corollary~\ref{rauzy})} Under the assumptions of Proposition~\ref{x+y1} we have:
\begin{enumerate}
  \item If $h_\F(x)$ and $h_\F(y)$ exist then 
  $$
  |h_\F(x)-h_\F(y)|\le \underline h_\F(x+y)\le\overline h_\F(x+y)\le h_\F(x)+h_\F(y).
  $$
  \item $y\in\mathcal D_\F(X,T)$ if and only if for any $x\in X$ we have 
  $\underline h_\F(x+y) =\underline h_\F(y+x) = \underline h_\F(x) \text{ \ and \ }
  \overline h_\F(x+y) =\overline h_\F(y+x) = \overline h_\F(x).$
  \item $\mathcal D_\F(X,T)=\mathcal D_\F^\perp(X,T)$,
  \item $\mathcal D_\F(X,T)\subset\mathcal N_\F^\perp(X,T)$.
\end{enumerate}
\end{cor}

\begin{rem}\label{remo}
The assumption in the formulation of Proposition~\ref{x+y1}, that $(X,T)$ has finite \tl\ entropy, is needed to ensure that the formulas~\eqref{inn} and ~\eqref{innn} do not lead to the indeterminate form $\infty-\infty$. This assumption is also needed for the inclusion (4) in Corollary~\ref{rauzy2}, since the proof uses the implication $\underline h_\F(x)=\htop(X,T) \implies x\in\mathcal N_\F(X,T)$, which does not need to hold when the \tl\ entropy is infinite. On the other hand, the equality (3) in Corollary~\ref{rauzy2} holds without finite entropy assumption, because the indeterminate form $\infty-\infty$ does not occur in~\eqref{inn} or~\eqref{innn} when at least one of the points $x,y$ is deterministic. Note that this equality answers positively the part of Question~\ref{q0} concerning $\mathcal D_\F^\perp(X,T)$.
\end{rem}

\begin{defn}\label{afm}
Let $(X,T)$ and $(Y,S)$ be algebraic systems.  A surjective group homomorphism $\pi:X\to Y$ such that $\pi\circ T=S\circ\pi$ is called an \emph{algebraic factor map} and the system $(Y,S)$ is called an \emph{algebraic factor} of $(X,T)$.
\end{defn}

\begin{prop}\label{xq0} 
Let $(X,T)$ be an ergodic algebraic system and let $(Y,S)$ be an algebraic factor of $(X,T)$ via an algebraic factor map $\pi:X\to Y$. Let $\F$ be a F\o lner \sq\ in $\N$. Then
\begin{enumerate}[(i)]
\item $\pi(\mathcal D_\F(X,T))\subset \mathcal D_\F(Y,S)$, 
\item $\pi(\mathcal N_\F(X,T))\subset \mathcal N_\F(Y,S)$.
\end{enumerate}
\end{prop}

\begin{proof}
Since $\pi$ is a factor map from \xt\ onto \ys, (i) is obvious. Next, $\pi$ induces a map $\pi^*$ (see \eqref{star}) from the set of $T$-\im s onto the set of $S$-\im s. Since $\pi$ is a surjective group homomorphism, $\pi^*$ sends the Haar measure $\lambda_X$ on $X$ to the Haar measure $\lambda_Y$ on $Y$. If $x\in\mathcal N_\F(X,T)$, it is $\F$-generic for $\lambda_X$ and hence $\pi(x)$ is $\F$-generic for the measure $\pi^*(\lambda_X)=\lambda_Y$ (Remark~\ref{mxmy} is valid also for $\F$-quasi-generic points), and thus $\pi(x)\in\mathcal N_\F(Y,S)$.
\end{proof}

The following result on ``lifting quasi-generic points'' is needed in the proof of Proposition~\ref{newcor} which provides an amplification of Proposition~\ref{xq0}. 
Proposition~\ref{newcor} will be utilized in Section~\ref{sred} in the proofs of Corollary~\ref{redu1}, Theorem~\ref{zredu} and Theorem~\ref{raufg}.

\begin{thm}\label{dwt}
Let $(X,T)$ be an ergodic algebraic system and let $\nu$ be a $T$-\im\ on $X$. Let $y\in X$ be $\F$-quasi-generic for the measure $\nu$. Let $\xi=\lambda_X\vee\nu$ be a joining of the Haar measure $\lambda_X$ with $\nu$. Then there exists an $\F$-normal point $x\in X$ such that the pair $(x,y)$ is $\F$-quasi-generic for $\xi$.
\end{thm}

\begin{proof}[Sketch of proof] For an automorphism $T$ and the standard F\o lner \sq\ in $\N$ (i.e., $\F=(F_n)_{n\ge1}$ where $F_n=\{1,2,\dots,n\}$, $n\ge1$) the statement follows directly from \cite[Theorem 1.3]{DW} (see also \cite[Proposition~4]{K2}) and \cite[Corollary on page~345]{Da}. To obtain Theorem~\ref{dwt} in full generality one needs to extend \cite[Corollary]{Da} to endomorphisms and extend \cite[Theorem 1.3]{DW} to arbitarary F\o lner \sq s in $\N$. The passage to endomorphisms can be done via the standard natural extensions technique, similar to that utilized in the proof of Proposition~\ref{ie} above. The adaptation of \cite[Theorem 1.3]{DW} relies on the fact that a general F\o lner \sq\ in $\N$ is equivalent\footnote{Two F\o lner \sq s $(F_n)_{n\ge1}$ and $(F'_n)_{n\ge1}$ are \emph{equivalent} if $\lim_{n\to\infty}\frac{|F_n\triangle F'_n|}{|F_n|}=0$.} to a F\o lner \sq\ $\F=(F_n)_{n\ge1}$ where the sets $F_n$ are unions of long intervals 
(see \cite[Lemma~8.2]{BDM}) and on a careful modification of the constructions in Section~3 of \cite{DW} in which ``density one'' is replaced by ``$\F$-density one'' and ``generic for $\mu$'' is replaced by ``$\F$-generic for $\mu$''. 
\end{proof}

\begin{prop}\label{newcor}
Let $(X,T)$ be an ergodic algebraic system and let $(Y,S)$ be an algebraic factor of $(X,T)$ via an algebraic factor map $\pi:X\to Y$. Let $\F$ be a F\o lner \sq\ in $\N$. Then
\begin{enumerate}[(i)]
\item $\pi(\mathcal N_\F(X,T))= \mathcal N_\F(Y,S)$, 
\item $\pi(\mathcal N^\perp_\F(X,T))\subset \mathcal N^\perp_\F(Y,S)$.
\end{enumerate}
\end{prop}

\begin{proof}
Consider the mapping $\bar\pi:X\to X\times Y$ defined by 
$$
\bar\pi(x) = (x,\pi(x)).
$$
The measure $\xi=\bar\pi^*(\lambda_X)$ is a joining (often called a \emph{factor joining}) of the ergodic measures $\lambda_X$ and $\pi^*(\lambda_X)=\lambda_Y$. Theorem~\ref{dwt} implies that any $\F$-normal point $y\in(Y,S)$ lifts with respect to $\bar\pi$ to an $\F$-normal pair $(x,y)\in(X\times Y,T\times S)$. Then $x$ is normal in the system $(X,T)$ and $y=\pi(x)$. We have shown that 
$$
\mathcal N_\F(Y,S)\subset\pi(\mathcal N_\F(X,T)),
$$
which, combined with Proposition~\ref{xq0}(ii), proves (i). 

Now suppose $x'\in\mathcal N^\perp_\F(X,T)$ and take any $y\in\mathcal N_\F(Y,S)$. By (i), there exists an $x\in\mathcal N(X,T)$ such that $y=\pi(x)$. Then, by (i) again, we have
$$
\pi(x') + y = \pi(x')+\pi(x)=\pi(x'+x)\in\pi(\mathcal N_\F(X,T))=\mathcal N_\F(Y,S),
$$
and hence $\pi(x')\in\mathcal N^\perp_\F(Y,S)$.
\end{proof}

The following question naturally presents itself:

\begin{ques}\label{ques}
Let $(X,T)$ be an ergodic algebraic system and let $\F$ be a F\o lner \sq\ in $\N$. Is it true that $\mathcal N_\F^\perp(X,T)\subset\mathcal D_\F(X,T)$?
\end{ques}

In the next section, after introducing some preparatory notation and facts, we provide the positive answer to this question for some classes of Abelian algebraic systems including toral endomorphisms.\footnote{The statement of \cite[Theorem on page 264]{K1}, can be interpreted as a positive answer to Question~\ref{ques} for certain Abelian groups. In particular, on page~268 in \cite{K1}, the author mentions (without proof) two specific instances of applicability of his theorem including hyperbolic endomorphisms of multidimensional tori $\mathbb T^n$ (\cite[Example~7]{K1}). However, it seems that the proof of \cite[Theorem on page 264]{K1} contains some gaps. First, we do not understand the interpretation of Furstenberg's theorem on the lack of disjointenss for positive entropy systems, and second, we could not fill a missing argument concerning averaging of a non-invariant measure along its orbit.}

\subsection{Preliminary results on endomorphisms of compact Abelian groups}\label{sred}
In this section we will restrict our attention to \emph{Abelian algebraic systems}, i.e., algebraic systems $(X,T)$ where $X$ is an infinite compact metrizable Abelian group. This will allow us to use the Pontryagin duality theory. 

 Recall that \emph{characters} on $X$ are continuous maps $\chi:X\to\{z\in\mathbb C:|z|=1\}$ satisfying $\chi(x+y)=\chi(x)\chi(y)$, $x,y\in X$. Note that product of characters is a character and so is the inverse (equivalently complex conjugate) of a character. The Pontryagin dual $\widehat X$ is the multiplicative group consisting of all characters. The characters separate points (see~\cite[Theorem 22.17]{HR}) and, since $X$ is compact, no proper subgroup of $\widehat X$ has this property (see, e.g., \cite[Theorem~1.3]{CR} and use the fact that compact topology is the weakest among Hausdorff topologies).
\smallskip

At first we will reduce the problem to algebraic systems for which there exists a character which separates orbits. The idea of Lemma~\ref{redu} is taken from the proof of~\cite[Lemma~4]{K2}. 

\begin{defn}\label{simp}
An Abelian algebraic system $(X,T)$ will be called \emph{simple} if there exists a (nontrivial) character $\chi$ on $X$ which separates orbits, i.e., for any $x,x'\in X, \ x\neq x'$ there exists an $n\ge 0$ such that $\chi(T^nx)\neq\chi(T^nx')$.
\end{defn}

We remark that while any endomorphism of the circle $(\mathbb T,R)$ is obviously simple (because the map $x\mapsto e^{2\pi ix}$ is a character which separates points), the higher-dimensional tori $\mathbb T^n$ admit both simple and not simple ergodic endomorphisms. We justify this claim by the following examples where $X=\mathbb T^2$.

\begin{exam}\label{s-ns}
Let $X=\mathbb T^2$ be the two-dimensional torus. Consider the endomorphisms 
$T(x,y)=(2x,3y)$ and $S(x,y)=(2x,2y)$. Both $T$ and $S$ are surjective and ergodic (because the matrices representing $T$ and $S$ have no eigenvalues which are roots of unity (see~\cite[page 623]{Ha} or~\cite[Corollary 2.20]{EW}). Yet, as the following considerations demonstrate, $(X,T)$ is simple while $(X,S)$ is not. Let $\chi$ be the character given by $\chi(x,y)=e^{2\pi i(x+y)}$. Two points $(x,y)$ and $(x',y')$ are not separated by $\chi$ if and only if 
\begin{equation}\label{siml}
x+y=x'+y'.
\end{equation}
Next, $\chi(T(x,y))=\chi(T(x',y'))$ if and only if 
\begin{equation}\label{sim2}
2x+3y=2x'+3y'. 
\end{equation}
If~\eqref{siml} and~\eqref{sim2} hold simultaneously, then the two points are identical. Any pair of distinct points is separated by either $\chi$ or $\chi\circ T$, and thus the system $(X,T)$ is indeed simple.
To see that $(X,S)$ is not simple, fix a character $\chi$ on $X$ and note that it has the form $\chi(x,y)=e^{2\pi i(kx+ly)}$, for some $k,l\in\Z$. If $k=0$ then $\chi$ does not separate the orbits of points of the form $(x,0)$. Similarly, if $l=0$ then $\chi$ does not separate orbits of points of the form $(0,y)$. If $k=\pm l$ then $\chi$ does not separate the orbits of points of the form $(x,y)$ with $x=\mp y$. Thus we can assume that $k\neq0$, $l\neq0$, and either $|k|\neq 1$ or $|l|\neq 1$ (or both). However, in this case the points $(\frac1k,\frac1l)$ and $(0,0)$ are different while $\chi$ does not separate their orbits.    
\end{exam}

\begin{lem}\label{redu}
Let $(X,T)$ be an Abelian algebraic system and let $\F$ be a F\o lner \sq\ in $\N$. Choose an element $y_0\in X\setminus\mathcal D_\F(X,T)$. Then there exists an algebraic factor map $\pi:(X,T)\to(X',T')$, where $(X',T')$ a simple algebraic system, such that $\pi(y_0)\in X'\setminus\mathcal D_\F(X',T')$. 
\end{lem}

\begin{proof}
By Definition~\ref{Fdeter}, $y_0$ is $\F$-quasi-generic for a $T$-\inv\ measure $\nu$ on $X$ such that 
$$
h_\nu(T)>0.
$$

It is well known that, under our assumptions, $\widehat X$, the Pontryagin dual of $X$, is infinite countable. So, we can write $\widehat X=\{\chi_0,\chi_1,\chi_2,\dots\}$, where  $\chi_0\equiv1$ is the trivial character. For a fixed $m\ge1$, let $\pi_m:X\to\mathbb T^{\N\cup\{0\}}$ be given by 
\begin{equation}\label{iscon}
\pi_m(x)= \mathbf x_m=(\mathbf x_{m,n})_{n\ge0}\in\mathbb T^{\N\cup\{0\}},\text{ where }\mathbf x_{m,n}=\chi_m(T^nx).
\end{equation}
The image $\mathbf X_m=\pi_m(\mathbb T^d)$ is clearly a compact Abelian group.
The map $\pi_m$ is an algebraic factor map from the system $(X,T)$ onto $(\mathbf X_m,\sigma_m)$, where $\sigma_m$ is the shift transformation~given~by
$$
(\sigma_m(\mathbf x_m))_n=\mathbf x_{m,n+1}, \ n\ge0.
$$
It is clear by construction that each of the systems $(\mathbf X_m,\sigma_m)$ is simple (with $\chi_m$ playing the role of $\chi$ in Definition~\ref{simp}).

Remark~\ref{mxmy} (which is valid also for $\F$-quasi-generic points) implies that
the element $\pi_m(y_0)\in\mathbf X_m$ is $\F$-quasi-generic for the $\sigma_m$-\im\ $\boldsymbol\nu_m=\pi^*_m(\nu)$. We will show now that $h_{\boldsymbol\nu_m}(\sigma_m)>0$ for at least one index~$m$. This will conclude the proof, because then $\pi_m(y_0)$, being $\F$-quasi-generic for $\boldsymbol\nu_m$, is not $\F$-deterministic, so the algebraic factor map $\pi=\pi_m$ (with $X'=\mathbf X_m$) satisfies the claim of the theorem. 

Consider the mapping $\bar\pi:X\to\prod_{m\ge1}\mathbf X_m$ given by
$$
\bar\pi(x)=(\pi_m(x))_{m\ge1}.
$$
This map is obviously continuous and satisfies $\bar\pi\circ T=\sigma\circ\bar\pi$, where $\sigma$ is the natural product transformation on $\prod_{m\ge1}\mathbf X_m$, $\sigma=\sigma_1\times\sigma_2\times\cdots$. Since characters separate points of $X$, $\bar\pi$ is also injective, and thus it is a \tl\ conjugacy between the algebraic systems $(X,T)$ and $(\mathbf X,\sigma)$, where $\mathbf X=\bar\pi(X)$.
This implies that $\bar\pi$ is also a measure-theoretic isomorphism between the measure-preserving systems $(X,\nu,T)$ and $(\mathbf X,\boldsymbol\nu,\sigma)$, where $\boldsymbol\nu=\bar\pi^*(\nu)$. In particular, we have $h_{\boldsymbol\nu}(\sigma)=h_\nu(T)>0$. Since for each $m\ge1$, the marginal of $\bar\pi^*(\nu)$ on $\mathbf X_m$ equals $\boldsymbol\nu_m$, we can view $\boldsymbol\nu$ as a countable joining $\bigvee_{m\ge1}\boldsymbol\nu_m$. The inequality~\eqref{eoj1} for countable joinings implies that
$$
0<h_{\boldsymbol\nu}(\sigma)\le \sum_{m\ge1}h_{\boldsymbol\nu_m}(\sigma_m),
$$
and thus there exists an $m\ge1$ such that $h_{\boldsymbol\nu_{m}}(\sigma_{m})>0$, as claimed.
\end{proof}

\begin{cor}\label{redu1}
Let $\mathfrak X$ be a class of ergodic Abelian algebraic systems such that, whenever $(X,T)\in\mathfrak X$, all algebraic factors of $(X,T)$ also belong to $\mathfrak X$. 
Let $\F$ be a F\o lner \sq\ in $\N$. If the inclusion
\begin{equation}\label{ooo}
\mathcal N_\F^\perp(X,T)\subset\mathcal D_\F(X,T)
\end{equation}
holds for all simple systems in $\mathfrak X$ then it holds for all systems in $\mathfrak X$.
\end{cor}
\begin{proof}
Let $(X,T)\in\mathfrak X$ and suppose there exists an element $y_0\in\mathcal N^\perp_\F(X,T)$ which is not $\F$-deterministic in $(X,T)$. By Lemma~\ref{redu}, there exists an algebraic factor map $\pi:(X,T)\to(X',T')$ onto a simple algebraic system such that $\pi(y_0)$ is not $\F$-deterministic in $(X',T')$. By Corollary~\ref{newcor}(ii), $\pi(y_0)\in\mathcal N^\perp_\F(X',T')$. Since $\mathfrak X$ is closed under algebraic factors, we have $(X',T')\in\mathfrak X$. We have arrived at a contradiction with the assumption that~\eqref{ooo} holds for all simple systems in the class $\mathfrak X$.
\end{proof}

\begin{defn}\label{subfac}
Let $(X,T)$ be an Abelian algebraic system. A \emph{polynomial in variable $T$} is a map $P:X\to X$ of the form 
\begin{equation}\label{poly}
P = a_0 T^0+a_1 T+ a_2T^2+\dots+a_kT^k,
\end{equation}
where $T^0$ is the identity map and $a_l\in\Z$ for $l=0,1,\dots,k$, $k\ge0$. The zero polynomial will be denoted by $P_0$.
\end{defn}

\begin{rem}\label{fsg}\phantom{.}
\begin{enumerate}[(a)]
\item Note that different polynomials may represent the same map. For example, if $T(x)=2x$ on $\mathbb T$ then $2kT^0-kT^1=P_0$ for any $k\in\Z$.
\item A polynomial $P$ in $T$ need not be surjective, even when $P\neq P_0$. For instance, if $X=\mathbb T^2$ and $T(x,y)=2x+3y$ then $P=-2T^0+T^1$ maps any point $(x,y)\in\mathbb T^2$ to $(0,y)$, so $P(X)$ is a one-dimensional subtorus of $X$. 
\item The image $P(X)$ is a \emph{$P$-subgroup} of $X$, that is, it is a closed $T$-\inv\ subgroup of $X$ and $P$ is an algebraic factor map from $(X,T)$ to the algebraic system $(P(X),T|_{P(X)})$.
\end{enumerate}
\end{rem}

\begin{defn}\label{ujj}
Given an ergodic Abelian algebraic system $(X,T)$, we let $\boldsymbol\U$ denote the (at most countable) collection of all \underline{non-surjective} polynomials in $T$.\footnote{If the system is ergodic then the family $\boldsymbol\U$ is either infinite countable or consists of just the trivial map $P_0$. Indeed,  Suppose $\boldsymbol\U$ is finite and contains a (not surjective) polynomial $P\neq P_0$. Then, for any nontrivial character $\gamma$ on $P(X)$, the map $\chi=\gamma\circ P$ is a nontrivial character on $X$ and for any $n\ge 0$ we have $\chi\circ T^n =\gamma\circ P\circ T^n$. Clearly, $P\circ T^n$ is a not surjective polynomial, and hence it belongs to $\boldsymbol\U$. So, $\chi$ has a finite orbit under the composition with~$T$, which implies that $T$ is not ergodic (see, e.g.,~\cite[Theorem~1]{Ha}).}  
Let 
$$
\mathbf Y =\{(P(x))_{P\in\boldsymbol\U}, x\in X\} \subset \prod_{P\in \boldsymbol\U}P(X) \ \ \text{(Cartesian product)},
$$
and define the map $\mathbf P:X\to \mathbf Y$
by 
$$
\mathbf P(x)=(P(x))_{P\in\boldsymbol\U}, \ \ x\in X.
$$
\end{defn}
Clearly, $\mathbf P$ is an algebraic factor map from $(X,T)$ to the algebraic system $(\mathbf Y,\mathbf T)$, where $\mathbf T$ denotes the product transformation $T\times T\times\dots$ restricted to $\mathbf Y$  (in the trivial case when $\boldsymbol\U=\{P_0\}$ we have $\mathbf Y=\{0\}$ and we let $\mathbf T$ be the identity map). 
\smallskip

The next theorem together with Corollary~\ref{trivy} answers Question~\ref{ques} for some classes of Abelian algebraic systems. 

\begin{thm}\label{zredu}
Let $(X,T)$ be a simple ergodic Abelian algebraic system and let $(\mathbf Y,\mathbf T)$ be as in Definition~\ref{ujj}. Let $\lambda_{\mathbf Y}$ denote the Haar measure on $\mathbf Y$. 
If 
\begin{equation}\label{incl}
h_{\lambda_X}(T)>h_{\lambda_{\mathbf Y}}(\mathbf T)
\end{equation}
then, for any F\o lner \sq\ $\F$ in $\N$,~\eqref{ooo} holds, i.e., 
$$
\mathcal N_\F^\perp(X,T)\subset\mathcal D_\F(X,T).
$$
\end{thm}

\begin{proof}
Suppose there exists an element $y_0\in\mathcal N^\perp_\F(X,T)$ which is not $\F$-deterministic. Then $y_0$ is $\F$-quasi-generic of an \im\ $\nu$ on $X$ satisfying $h_\nu(T)>0$.
Recall that $\mathbf P$ is an algebraic factor map from $(X,T)$ onto $(\mathbf Y,\mathbf T)$ and note that $\lambda_{\mathbf Y}= \mathbf P^*(\lambda_X)$ (recall that, by convention, $\mathbf P^*$ is the map from $\M(X,Y)\to\M(\mathbf Y,\mathbf T)$ induced by $\mathbf P$, see~\eqref{star}). The inequality $h_{\lambda_X}(T)>h_{\lambda_{\mathbf Y}}(\mathbf T)$ can be interpreted in terms of conditional entropy as follows:
$$
h_{\lambda_X}(T|\Sigma)>0,
$$
where $\Sigma = \{\mathbf P^{-1}(B): B\text{ is a Borel set in }\mathbf Y\}$.
By Sinai's theorem (\cite{S1}) and Thouvenot's relative factor theorem (\cite{T}, see also~\cite{Se}), the measure-preserving systems $(X,\lambda_X,T)$ and $(X,\nu,T)$ have a common Bernoulli factor $(Z,\zeta,S)$ which is independent (with respect to $\lambda_X$) of~$\Sigma$. That is, if we let $\phi_1:(X,\lambda_X,T)\to (Z,\zeta,S)$ and $\phi_2:(X,\nu,T)\to(Z,\zeta,S)$ denote the respective (measure-theoretic) factor maps then any complex functions of the form $f\circ\phi_1$ and $g\circ \mathbf P$, where $f\in L^2_0(\zeta)$ and $g\in L^2_0(\lambda_{\mathbf Y})$, 
($L^2_0(\mu)$ stands for the orthocomplement of constant functions in $L^2(\mu)$)
are orthogonal in $L^2(\lambda_X)$. Let $\xi$ be any joining $\lambda_X\vee\nu$ over the common factor $(Z,\zeta,T)$, which means that $\xi$-almost all pairs $(x,y)\in X\times X$ satisfy $\phi_1(x)=\phi_2(y)$.\footnote{At least one joining over the common factor (so-called relatively independent joining) always exists, see, e.g.,~\cite[page 800]{R}.} Let $f\in L^2_0(\zeta)$ be a non-constant function on $Z$. Denote $f_1=f\circ\phi_1$ and $f_2=f\circ\phi_2$. These are non-constant complex functions on $X$ which satisfy $f_1(x)=f_2(y)$ for $\xi$-almost all pairs $(x,y)$, and hence
\begin{equation}\label{depe}
    \int f_1(x)\bar f_2(y)\,d\xi(x,y)>0.
\end{equation}
Since $(X,T)$ is simple, there exists a character $\chi$ on $X$ which separates orbits. Then, for any polynomial $P(x)=a_0x+a_1Tx+\cdots+a_kT^kx$, the function  
\begin{equation}\label{cha}
    \chi(P(x))=\chi(a_0 x)\chi(a_1 Tx)\cdots\chi(a_kT^kx),
\end{equation}
is a character on $X$. Now, the family $\Theta$ of all characters of this form separates points (because the characters $\chi\circ T^n$, $n\ge0$, do), and clearly it is a group (with multiplication). So $\Theta=\widehat X$, the dual group of $X$. The characters are linearly uniformly dense in $C(X)$, and hence linearly dense in both $L^2(\lambda_X)$ and $L^2(\nu)$. Therefore, we can approximate $f_1$ in $L^2(\lambda_X)$ and $\bar f_2$ in $L^2(\nu)$ arbitrarily well by linear combinations of the characters on $X$. Since $\int f_1(x)\,d\lambda_X = \int f\,d\zeta=0$, $f_1$ is orthogonal in $L^2(\lambda_X)$ to the trivial character, so this character can be omitted in the combinations approximating $f_1$. Moreover, since $f_1$ is lifted from $L^2_0(\zeta)$ with respect to $\phi_1$, it is orthogonal to any function lifted from $L^2(\lambda_{\mathbf Y})$ with respect to~$\mathbf P$. So, in the combinations of characters approximating $f_1$ we can also omit all nontrivial characters obtained by lifting characters on $\mathbf Y$ with respect to $\mathbf P$. 
Thus, there exist linear combinations of characters on $X$, say $g_1$ and $g_2$, 
where $g_1$ avoids any (trivial and non-trivial) characters lifted from $\mathbf Y$ with respect to $\mathbf P$, such that
\begin{equation*}
\int g_1(x)g_2(y)\,d\xi(x,y)>0.
\end{equation*} 
This inequality in turn implies that there exist two characters on $X$, say $\chi_1$ and $\chi_2$, with $\chi_1$ non-trivial and not lifted from $\mathbf Y$ with respect to $\mathbf P$, such that
\begin{equation}\label{new}
\int \chi_1(x)\chi_2(y)\,d\xi(x,y)\neq 0.
\end{equation}
(in fact, $\chi_2$ cannot be trivial either, because $\int \chi_1\,d\lambda_X = 0$).
By~\eqref{cha}, there are polynomials $P$ and $Q$ in $T$, such that
\begin{equation*}
\chi_1=\chi\circ P,\ \ \ 
\chi_2=\chi\circ Q.
\end{equation*}

If $P$ was not surjective (i.e., if it belonged to $\boldsymbol\U$), we would have $P=\pi\circ\mathbf P$ where $\pi$ is the natural projection of $\mathbf Y$ onto $P(X)$. Then $\chi_1$ would equal $\chi|_{P(X)}\circ\pi\circ\mathbf P$, so it would be the character $\chi|_{P(X)}\circ\pi$ on $\mathbf Y$ lifted with respect to $\mathbf P$, a contradiction. We conclude that $P$ is surjective. 

By Theorem~\ref{dwt}, there exists an $\F$-normal element $x_0\in X$ such that the pair $(x_0,y_0)$ is $\F$-quasi-generic for $\xi$. We let $\F'=(F_{n_k})_{k\ge1}$ denote the sub\sq\ of $\F$ such that $(x_0,y_0)$ is $\F'$-generic for $\xi$. Then the (non-vanishing) integral in~\eqref{new} becomes 
\begin{multline}\label{depe2}
\lim_{k\to\infty}\frac1{|F_{n_k}|}\sum_{n\in F_{n_k}}\chi(P(T^n( x_0)))\chi(Q(T^n(y_0)))=\\
\lim_{k\to\infty}\frac1{|F_{n_k}|}\sum_{n\in F_{n_k}}\chi(T^n(P(x_0)+Q(y_0))).
\end{multline}

Since $P$ is surjective, in virtue of Corollary~\ref{newcor}(i), we have $P(x_0)\in\mathcal N_\F(X,T)$. On the other hand, $Q(y_0)\in\mathcal N^\perp_\F(X,T)$ (here we cannot use Corollary~\ref{newcor}(ii), instead we use the fact that $Q$ is a polynomial in variable $T$ and that $\mathcal N_\F^\perp(X,T)$ is a $T$-invariant subgroup of $X$). So, $P(x_0)+Q(y_0)\in\mathcal N_\F(X,T)$ and the right hand side of~\eqref{depe2} equals the integral of the nontrivial character $\chi$ with respect to the Haar measure~$\lambda_X$. Since such an integral equals $0$ we have a contradiction with~\eqref{new}, which ends the proof.
\end{proof}

\begin{cor}\label{trivy}
If $(X,T)$ is an ergodic Abelian algebraic system such that any proper $P$-subgroup of $X$ (see Remark~\ref{fsg}(c)) is finite, then~\eqref{ooo} holds, that is, for any F\o lner \sq\ $\F$ in $\N$, we have
$$
\mathcal N_\F^\perp(X,T)\subset\mathcal D_\F(X,T).
$$
\end{cor}
\begin{proof}
It is obvious that if an Abelian algebraic system $(X,T)$ has the property that all its proper $P$-subgroups are finite then the same property have all algebraic factors of $(X,T)$. So, the class $\mathfrak X$ of ergodic Abelian algebraic systems with this property satisfies the assumption of Corollary~\ref{redu1}. Clearly, for any system in this class we have $h_{\lambda_{\mathbf Y}}(\mathbf T)=0$, which is strictly less than $h_{\lambda_X}(T)$, i.e.,~\eqref{incl} holds. By Theorem~\ref{zredu}, any simple system $(X,T)$ in the class $\mathfrak X$ satisfies~\eqref{ooo} and by Corollary~\ref{redu1}, any system in the class $\mathfrak X$ satisfies~\eqref{ooo}, as claimed.
\end{proof}

\subsection{Applications to direct products of $\Z/p\Z$ ($p$ prime) and solenoids}\label{eee}
In this subsection we apply Corollary~\ref{trivy} to two particular classes of Abelian groups.

Let $\Lambda_p=\{0,1,2,\dots,p-1\}$, where $p$ is prime. 
On $\Lambda_p^\N$ consider the operation $+$ of the coordinatewise addition modulo $p$. Clearly, this operation is continuous ($\Lambda_p^\N$ is isomorphic, as a \tl\ group, to the infinite direct product $(\Z/p\Z)^\N$). The Haar measure on $\Lambda_p^\N$ in the product measure $\mu^\N$, where $\mu$ is the normalized counting measure on $\Lambda$ (note that $\mu^\N$ coincides with the uniform Bernoulli measure on one-sided \sq s over $\Lambda$). The shift $\sigma$ is an endomorphism of $\Lambda_p^\N$, and the Abelian algebraic system $(\Lambda_p^\N,\sigma)$ is ergodic.

For this system we can proof a theorem fully analogous to the Rauzy theorem (Theorem~\ref{newrauzy}).

\begin{thm}\label{rtforzp}
$$
\mathcal D_\F(\Lambda_p^\N,\sigma)=\mathcal N_\F^\perp(\Lambda_p^\N,\sigma),
$$
\end{thm}
\smallskip

The proof is preceded by a lemma.

\begin{lem}\label{fajne}Let $p$ be a prime.
Any proper closed shift-\inv\ subgroup $H$ of the group $\Lambda_p^\N$ is finite.
\end{lem}

\begin{proof} Let us call a block $B\in\Lambda_p^k$ ($k\in\N$) \emph{$H$-admissible} if $B$ appears in some element of $H$. Since $H$ is a subgroup of $\Lambda_p^\N$, it contains the \sq\ $\mathbf 0$ consisting of only zeros. If $H=\{\mathbf 0\}$ then it is finite and the proof ends. Otherwise some nonzero element $a\in\Lambda_p$, viewed as a block of length $1$, is $H$-admissible. Then, for any $n\in\N$, the number $na$ (mod $p$), viewed as a block of length $1$ over $\Lambda_p$ is also $H$-admissible. Since $p$ is prime, the numbers $na$ (mod $p$) represent all $b\in\Lambda_p$. We have shown that any block of length $1$ is $H$-admissible. Since $H$ is a proper closed and shift-\inv\ subset of $\Lambda_p^\N$, there exists a maximal number $k_0\in\N$ such that all blocks of length $k_0$ are $H$-admissible (and at least one block of length $k_0+1$ is not $H$-admissible). Since each block $B\in\Lambda_p^{k_0}$ is $H$-admissible, it has an $H$-admissible continuation $Ba$, $a\in\Lambda_p$. Suppose that $0^{k_0}a$ is $H$-admissible, where $a\in\Lambda_p$, $a\neq 0$. Then, arguing as before, we get that $0^{k_0}b$ is $H$-admissible for any $b\in\Lambda_p$. Let $B\in\Lambda_p^{k_0}$ be arbitrary and let $Ba$ be an $H$-admissble continuation of $B$. By shift-invariance of $H$, the sum of two $H$-admissible blocks is $H$-admissible. In particular, we can add the $H$-admissible blocks $0^{k_0}b$ and $Ba$ and obtain that the block $Bc$, where $c=b+a$, is $H$-admissible. Since $b$ is an arbitrary element of $\Lambda_p$, so is $c$. In this manner, we obtain that all blocks of length $k_0+1$ are $H$-admissible. This contradicts the definition of $k_0$. We conclude that the only $H$-admissible continuation of $0^{k_0}$ is $0^{k_0+1}$. Now suppose that some block $B$ of length $k_0$ has two different $H$-admissible continuations $Bb$ and $Bc$ with $b\neq c\in\Lambda_p$. Then, by subtraction, we find that the block $0^{k_0}a$ is $H$-admissible, where $a=b-c\neq 0$, a possibility that has just been eliminated. We have shown that any block of length $k_0$ has a unique $H$-admissible continuation. This implies that any $x\in H$ is determined by the block $x|_{[1,k_0]}$, and hence $|H|=|\Lambda_p|^{k_0}$.
\end{proof}

\begin{proof}[Proof of Theorem~\ref{rtforzp}]
By Lemma~\ref{fajne}, any proper closed shift-\inv\ subgroup (in particular, any proper $P$-subgroup) of $\Lambda_p^\N$ is finite. Now,
Corollary~\ref{trivy} implies that for any F\o lner \sq\ $\F$ in $\mathbb N$, we have $\mathcal N_\F^\perp(\Lambda_p^\N,\sigma)\subset\mathcal D_\F(\Lambda_p^\N,\sigma)$. Since $\htop(\Lambda_p^\N,\sigma)=\log p<\infty$, Corollary~\ref{rauzy2}(4) gives the opposite inclusion. 
\end{proof}

We continue to consider the group $(\Lambda_p^\N,+)$. 
Any polynomial $P$ in $\sigma$ is a continuous homomorphism of $\Lambda_p^\N$. Suppose that $P$ is surjective (then $P$ is an endomorphism of $\Lambda_p^\N$), and moreover, suppose that the Abelian algebraic system $(\Lambda_p^\N,P)$ is ergodic. Since $P$ commutes with $\sigma$, it is easy to see that every proper closed $P$-\inv\ subgroup of $X$ is also shift-\inv\ and hence Corollary~\ref{trivy} applies to $(\Lambda_p^\N,P)$. Invoking the fact that $\htop(\Lambda_p^\N,P)<\infty$, we arrive at the following result. 

\begin{cor}\label{oyo}
For any ergodic polynomial $P$ in $\sigma$ on $\Lambda_p^\N$ (where $p$ is a prime), and any F\o lner \sq\ $\F$ in $\N$, one has
$$
\mathcal D_\F(\Lambda_p^\N,P)=\mathcal N_\F^\perp(\Lambda_p^\N,P). 
$$
\end{cor}

\begin{rem}\phantom{.}
\begin{enumerate}[(i)]
\item It can be shown (using the Pontryagin dual $\widehat X$ and~\cite[Theorem~1]{Ha}) that a polynomial $P$ in $\sigma$ is ergodic (in particular, surjective) if and only if it is not of the trivial form $P=a_0\sigma^0$, $a_0\in\Z$ (recall that $\sigma^0$ stands for the identity map). We skip the proof. 
\item Note that polynomials $P$ in $\sigma$ coincide with \emph{algebraic cellular automata}, i.e., cellular automata given  given by 
$$
(P(x))_n = \sum_{k=0}^N a_k x_{n+k}\!\!\mod p, \ \ x=(x_n)_{n\in\N}.
$$
where $N\ge0$ and $a_k\in\Z$, $k=0,1,\dots,N$.
\item If we consider $\Lambda_p^\Z$ rather than $\Lambda_p^\N$ then  Lemma~\ref{fajne} holds as well. Since now the shift transformation $\sigma$ is invertible, we may include as polynomials in $\sigma$ all maps of the form
$$
P=a_{-k}\sigma^{-k}+a_{-k+1}\sigma^{-k+1}+\cdots+a_0\sigma^0+a_1\sigma+\cdots+a_k\sigma^k,
$$
where $k\ge0$, $a_{-k},\dots,a_k\in\Z$. 
With this modification, the analogs of Corollary~\ref{oyo} and item (i) of this remark hold. 
Again, we skip the details.
\end{enumerate}
\end{rem}

In our next example, Corollary~\ref{trivy} is applied to prove an analog of Rauzy theorem (Theorem~\ref{newrauzy}) for the so-called solenoids.

\begin{defn}\label{sole} Let $\mathbf p=(p_k)_{k\ge1}$ be a \sq\ of (not necessarily distinct) prime numbers. The \emph{solenoid with base $\mathbf p$} is the compact Abelian group defined as follows. Let
$$
\mathbb S_{\mathbf p} = \{(t_k)_{k\ge 1}\in\mathbb T^\N: t_k=p_kt_{k+1}\!\!\mod 1\}.
$$
The set $\mathbb S_{\mathbf p}$ is endowed with the operation of addition inherited from the direct product $\mathbb T^\N$. 
\end{defn}
In other words, $\mathbb S_{\mathbf p}$ is the \tl\ group obtained as the inverse limit $\overset{\longleftarrow}{\lim}(\mathbb T_k,p_k)$ of the circle groups $\mathbb T_k=\mathbb T$ with the bonding maps defined as multiplications by $p_k$, as shown in the following diagram
$$
\mathbb T\overset{p_1}{\longleftarrow}\mathbb T\overset{p_2}{\longleftarrow}\mathbb T\overset{p_3}{\longleftarrow}\cdots.
$$

It is well known that solenoids are connected (and in fact, they are indecomposable continua).
For more details concerning solenoids we refer the reader to \cite[Chapter VI]{HR}, \cite{AF} and \cite{H}. 

Denote 
$$
\N_{\mathbf p}=\{1\}\cup\{p_{k_1}p_{k_2}\cdots p_{k_m},\ k_1<k_2<\dots<k_m,\ m\in\N\}
$$
and let $\Q_{\mathbf p}$ be the set
of rational numbers which in some (perhaps reducible) form have denominators in $\N_{\mathbf p}$. 
The Pontryagin dual $\widehat{\mathbb S}_{\mathbf p}$ equals the discrete (additive) group $\Q_{\mathbf p}$. Any endomorphism of $\mathbb S_{\mathbf p}$ is dual to an endomorphism of $\Q_{\mathbf p}$. One can show that the group of endomorphisms of $\Q_{\mathbf p}$ is generated by multiplications by nonzero integers and fractions of the form $\frac1p$, where $p$ is a prime that appears in the \sq\ $\mathbf p$ infinitely many times. Every endomorphism of the solenoid $\mathbb S_{\mathbf p}$ is ergodic except when it is dual to the multiplication by either $1$ or~$-1$. Any ergodic endomorphism of $\mathbb S_{\mathbf p}$ has positive and finite entropy (see, e.g., \cite[Theorem~1]{Ju1}).

\begin{thm}\label{sol}
Let $T$ be an ergodic endomorphism of a solenoid $\mathbb S_{\mathbf p}$. Then the analog of Theorem~\ref{newrauzy} holds:
$$
\mathcal N^\perp_\F(\mathbb S_{\mathbf p},T) = \mathcal D_\F(\mathbb S_{\mathbf p},T).
$$
\end{thm}

\begin{proof}By Corollary~\ref{rauzy2}~(4) we have $\mathcal D_\F(\mathbb S_{\mathbf p},T)\subset\mathcal N^\perp_\F(\mathbb S_{\mathbf p},T)
$. In view of Corollary~\ref{trivy}, in order to prove the reverse inclusion, it suffices to notice that the only proper $P$-subgroup of $\mathbb S_{\mathbf p}$ is the trivial subgroup. This follows from \tl\ properties of the solenoid. Since $\mathbb S_{\mathbf p}$ is compact and connected, so is any of its $P$-subgroups. Now, any proper compact connected subset of a solenoid is either a point or an arc (see e.g., \cite[Theorem~2]{H}) and it is well-known that no topological group is homeomorphic to an arc.\footnote{Any arc has the \emph{fixed-point property}, while for any nontrivial \tl\ group $G$ the map $x\mapsto x+x_0$ (where $x_0\neq 0$) is a homeomorphism of $G$ without any fixed points.} So, the only possible $P$-subgroup of $\mathbb S_{\mathbf p}$ is the trivial subgroup.
\end{proof}

\subsection{Rauzy theorem for toral endomorphisms}\label{4.1}

In this subsection, we will prove Theorem~\ref{raufg} which is an analog of the Rauzy theorem (Theorem~\ref{newrauzy}) for ergodic toral endomorphisms. This result is in a way deeper than the results in the preceding subsection, because multidimensional tori do not satisfy the assumption of Corollary~\ref{trivy}, and additional effort is needed to deal with non-surjective polynomials in $T$ which have infinite image. 

\smallskip
Fix an integer $d\ge1$ and consider the $d$-dimensional torus $\mathbb T^d$. Its elements are vectors $x=(x_1,x_2,\dots,x_d)$ with entries in the circle $\mathbb T=\R/\Z$. Any surjective endomorphism of $\mathbb T^d$ is given by the formula $x\mapsto Ax$, where $A$ is a nonsingular integer $d\times d$-matrix and $x=(x_1,x_2,\dots,x_d)\in\mathbb T^d$ is written as a column vector. 
We will denote this endomorphism by the same letter $A$ and call it a \emph{toral endomorphism}.\footnote{An iconic example in this class is the map of $\mathbb T^2$ given by $(a,b)\mapsto(2a+b,a+b)$.} As in any algebraic system, the Haar measure $\lambda$ on $\mathbb T^d$ is preserved under $A$, and hence the Lebesgue measure (which is the completion of the Haar measure and will be denoted by $\lambda$ as well) is also preserved. The measure-preserving system $(\mathbb T^d,\lambda,A)$ is ergodic if and only if $A$ does not have roots of unity among its eigenvalues (see~\cite[page 623]{Ha} or~\cite[Corollary 2.20]{EW}). In the ergodic case, by Proposition~\ref{ie}, the system has positive entropy. It also follows from Proposition~\ref{ie} that since the entropy is finite\footnote{We have $\htop(X,T)=\sum_i\max\{0,\log|\lambda_i|\}$, where the sum ranges over all eigenvalues of $A$; see, e.g.,~see, e.g.,~\cite[Theorem~1]{Ju1}.}, $\lambda$ is the unique measure of maximal entropy.

\begin{lem}\label{red}
Let $A:\mathbb T^d\to\mathbb T^d$ be an ergodic toral endomorphism and let $(Y,S)$ be a nontrivial algebraic system such that there exists an algebraic factor map\break$\pi:(\mathbb T^d,A)\to(Y,S)$. 
Then 
\begin{enumerate}
\item $Y$ is (isomorphic to) a $d'$-dimensional torus with $d'\le d$ and $S$ is an ergodic toral endomorphism.
\item
We have either
    \begin{enumerate}[(a)]
        \item $h_\mu(A)=h_{\pi^*(\mu)}(S)$, for all \im s $\mu\in\M_A(\mathbb T^d)$, or
        \item $\htop(\mathbb T^d,A)>\htop(Y,S)$.
    \end{enumerate}
\end{enumerate}
\end{lem}

\begin{proof}
(1) A compact Abelian group is (isomorphic to) the $d$-dimensional torus~$\mathbb T^d$ if and only if its Pontryagin dual is isomorphic to the additive group $\Z^d$. The algebraic factor $\pi$ induces an injective embedding of the dual of $Y$ in the dual of $\mathbb T^d$, $\widehat\pi:\widehat Y\to\widehat{\mathbb T}^d$, by composition: $\widehat\pi(\chi)=\chi\circ\pi$, $\chi\in\widehat Y$. Thus, $\widehat Y$ is isomorphic to a (nontrivial) subgroup of $\Z^d$. Any such subgroup is isomorphic to $\Z^{d'}$ for some $d'\in\{1,2,\dots,d\}$. So, $Y$ is (isomprphic to) a $d'$-dimensional torus and $S$ is an ergodic toral endomorphism.

\smallskip\noindent(2) Let $H\subset\mathbb T^d$ denote the kernel of $\pi$. Then $H$ is a closed $A$-\inv\ subgroup of $\mathbb T^d$. 
By \cite[Theorem 2]{Ju2}, we have 
\begin{equation}\label{oto}
\htop(\mathbb T^d,A)=\htop(\mathbb T^d/H,A^H)+\htop(H,A|_H),
\end{equation}
where $A^H$ stands for the map induced by $A$ on $X/H$. Since $\pi$ and the natural projection $\mathbb T^d\mapsto \mathbb T^d/H$ have the same kernel, the factors $(Y,S)$ and $(\mathbb T^d/H,A^H)$ are \tl ly conjugate, and hence \begin{equation}\label{ot}
\htop(Y,S)=\htop(\mathbb T^d/H,A^H).
\end{equation}
Any proper closed subgroup of the $d$-dimensional torus is either finite or it is (isomorphic to) a product of a $d'$-dimensional torus, where $1\le d'<d$, with a finite group. If $H$ is finite then $\pi$ is finite-to-one and thus it satisfies (a). Suppose $H= Z\times G$ where $Z$ is (isomorphic to) a $d'$-dimensional torus with $1\le d'< d$ and $G$ is a finite group. Since $H$ is \inv\ under $A$, the torus $Z$ is \inv\ under $A^k$ for some $k\in\N$. The matrix $A^k$ is nonsingular, so it preserves dimension, which implies that $A^k|_Z$ is surjective, i.e., it is a toral endomorphism. Ergodicty of $A$ is equivalent to the lack of eigenvalues that are roots of unity, in particular, and it implies the ergodicity of $A^k$. Since any eigenvalue of $A^k|_Z$ is also an eigenvalue of $A^k$, $A^k|_Z$ is ergodic. By Proposition~\ref{ie}, $(Z,A^k|_Z)$ has a positive entropy, and therefore $(H,A|_H)$ also has a positive entropy, which, considering~\eqref{ot} and~\eqref{oto}, implies~(b).
\end{proof}

We are now in a position to present the main theorem of this subsection.
\begin{thm}\label{raufg}
Let $A:\mathbb T^d\to\mathbb T^d$ $(d\ge1)$ be an ergodic toral endomorphism. Let $\F=(F_n)_{n\ge1}$ be a F\o lner \sq\ in~$\N$. Then 
$$
\mathcal D_\F(\mathbb T^d,A)=\mathcal N_\F^\perp(\mathbb T^d,A).
$$
\end{thm}

\begin{proof}
Since $(\mathbb T^d,A)$ has finite \tl\ entropy, in view of Corollary~\ref{rauzy2}(4), we only need to prove $N_\F^\perp(\mathbb T^d,A)\subset\mathcal D_\F(\mathbb T^d,A)$. By Lemma~\ref{red}(1), the class $\mathfrak X$ of ergodic toral endomorphisms is closed under the operation of taking algebraic factors. So, by Corollary~\ref{redu1}, it suffices to prove the theorem for simple ergodic toral endomorphisms.

The proof uses induction on the dimension $d$ of the torus. 
The theorem holds for $d=1$, in which case it reduces to Theorem~\ref{rauzy1} (although the formulation of Theorem~\ref{rauzy1} concerns real numbers, the proof is done for the circle $\mathbb T$). Fix $d\ge2$ and suppose that the theorem holds for any simple ergodic toral endomorphism of dimension $d'\in\{1,2,\dots,d-1\}$. Consider a simple ergodic toral endomorphism $(\mathbb T^d,A)$ and suppose that there exists a nondeterministic element $y_0\in\mathcal N^\perp_\F(\mathbb T^d,A)$. Let $\nu$ be an \im\ on $\mathbb T^d$ which has positive entropy and is $\F$-quasi-generated by $y_0$. Let $\mathbf P:(\mathbb T^d,A)\to(\mathbf Y,\mathbf T)$ be the algebraic factor map introduced in Definition~\ref{ujj}. By Lemma~\ref{red}(2), we have either 
\begin{enumerate}
\item[(a)] $h_\nu(A)=h_{\mathbf P^*(\nu)}(\mathbf T)$ (and hence the latter is strictly positive), 
\end{enumerate}
or
\begin{enumerate}
\item[(b)] $\htop(\mathbb T^d,A)>\htop(\mathbf Y,\mathbf T)$. 
\end{enumerate}
Since in any algebraic system the Haar measure is a measure of maximal entropy, the condition~(b) implies that $h_\lambda(A)>h_{\lambda_{\mathbf Y}}(\mathbf T)$ and~\eqref{ooo} follows directly from Theorem~\ref{zredu}. We will focus on the case (a). 
Since the measure-preserving system $(\mathbf Y,\mathbf P^*(\nu),\mathbf T)$ is a countable joining of the systems $(P(\mathbb T^d),P^*(\nu),A|_{P(\mathbb T^d)})$, where $P$ ranges over all non-surjective polynomials in variable $A$, by the inequality~\eqref{eoj1}, there exists a non-surjective polynomial $P_0$ such that $h_{P_0^*(\nu)}(A|_{P_0(\mathbb T^d)})>0$. This implies that the element $P_0(y_0)$, which is clearly $\F$-quasi-generic for $P_0^*(\nu)$, is not $\F$-deterministic. On the other hand, by Corollary~\ref{newcor}~(ii), we have $P_0(y_0)\in\mathcal N^\perp_\F(P_0(\mathbb T^d),A|_{P_0(\mathbb T^d)})$. Note that $P_0(\mathbb T^d)$, being a proper closed subgroup of $\mathbb T^d$, and also being connected as a continuous image of $\mathbb T^d$, is a proper subtorus of $\mathbb T^d$, and hence its dimension is less than $d$. The system $(P_0(\mathbb T^d),A|_{P_0(\mathbb T^d)})$ is not only a factor but also a subsystem of $(\mathbb T^d,A)$. Since it is a factor, it is ergodic. Since it is a subsystem, it is simple (the property of being simple is obviously inherited by algebraic subsystems of algebraic systems). By the inductive assumption, $P_0(y_0)$ should be $\F$-deterministic, which is a contradiction.
\end{proof}

\begin{rem}\label{sol2} Theorem~\ref{sol} has a natural extension to higher-dimensional solenoids. These are defined, for $d\ge 2$, as $d$-dimensional, connected, compact abelian groups (equivalently, as dual groups of subgroups of $\Q^d$, see~\cite{LiWa}). Any $d$-dimensional solenoid can also be constructed as an inverse limit of the $d$-dimensional tori. For higher dimensional solenoids, Rauzy theorem holds as well. The proof is similar to that for $d$-dimensional tori and relies on the fact that any proper closed connected subgroup of a $d$-dimensional solenoid is a solenoid of a lower dimension. 
\end{rem}

\section{Negative results for $p$-normality when $p\neq\frac12$}\label{S6}

While the definition of normality of a real number $x$ in base 2 deals with equal weights associated to the digits 0 and 1 in the binary alias of $x$
(recall that the binary alias $\omega_2(x)$ of a real number $x$ was introduced in Section~\ref{S2} as the \sq\ of digits in the binary expansion of the fractional part $\{x\}$ of $x$), one can also consider $p$-normality\footnote{In the notation ``$p$-normal'', $p$ is a number strictly between $0$ and $1$, while in the similarly looking notation ``$r$-normal'', $r$ is a natural number larger than $1$, so there should be no confusion.}, i.e., a more general situation where, for some $p\in(0,1)$, the digit $1$ has weight $p$ and the digit $0$ has weight $1-p$. It is natural to ask whether an analog of Rauzy theorem (Theorem~\ref{rrrr}) still holds for $p$-normality. In this section, we will show that, for $p\neq\frac12$, an analog of Theorem~\ref{rrrr}, as well as the analogs of Proposition~\ref{nordet1}~(4), Corollary~\ref{rrr}~(1), Corollary~\ref{indep} and Proposition~\ref{qx}~(2), fail dramatically, meaning that not only there are counterexamples to these statements, but there are actually no non-trivial examples for which the ``$p$-analogs'' hold.

In order to obtain results for real numbers $x\in\R$, we will first conduct the proofs for either elements $t$ of the system $(\mathbb T,R)$ where $R(t)=2t$ and for \sq s $\omega$ viewed as elements of the symbolic system $(\{0,1\}^\N,\sigma)$, where addition of \sq s involves the carry. This addition will be denoted by the symbol $\pl$. Formally, if $\omega,\tau\in\{0,1\}^\N$, 
$\omega=(a_n)_{n\ge1}, \tau=(b_n)_{n\ge1}$ then $\omega\pl\tau=(c_n)_{n\ge1}$, where
$$
c_n = 
\begin{cases}
a_n+b_n\!\!\mod2,& \text{ if }\sum_{i>n}^\infty \frac{a_i+b_i}{2^i}\le\frac1{2^n},\\
a_n+b_n+1\!\!\mod2,& \text{ otherwise}.
\end{cases}
$$
If $n$ is such that $\sum_{i>n}^\infty \frac{a_i+b_i}{2^i}>\frac1{2^n}$, we will say that \emph{the carry} occurs at the coordinate~$n$.

The factor maps $x\mapsto\{x\}$ (the fractional part) from $\R$ to $\mathbb T$ and $\omega\mapsto \phi_2(\omega)$ from $\{0,1\}^\N$ to $\mathbb T$ (see Proposition~\ref{rr}) will allow us to transfer the results from $(\mathbb T,R)$ and $(\{0,1\}^\N,\sigma)$ to the reals. We start with the formal definitions of $p$-normality in the three setups: for \sq s, for elements of the circle, and for real numbers.

\begin{defn}\label{pns}Let $p\in(0,1)$.
\begin{enumerate} 
\item A \sq\ $\omega\in\{0,1\}^\N$ is $p$-normal if every finite block $B=(b_1,b_2,\dots,b_k)\in\{0,1\}^k$ appears in $\omega$ with frequency $p^s(1-p)^{k-s}$, where $s\in\{0,1,\dots,k\}$ is the number of $1$'s appearing in $B$. Equivalently, $\omega$ is $p$-normal if it is generic (under the shift $\sigma$) for the $(p,1-p)$-Bernoulli measure $\mu_p$ on $\{0,1\}^\N$. 
\item An element $t\in\mathbb T$ is $p$-normal if $t=\phi_2(\omega)$ for some $p$-normal \sq\ $\omega\in\{0,1\}^\N$.
\item A real number $x$ is called $p$-normal if its fractional part $\{x\}$ is a $p$-normal element of $\mathbb T$, equivalently, if its binary alias $\omega_2(x)$ is a $p$-normal \sq.
\end{enumerate}
\end{defn}

\begin{rem}\label{ren}
If $\omega\in\{0,1\}^\N$ is $p$-normal then, since $\omega$ is generic for the $(p,1-p)$-Bernoulli measure, the entropy of $\omega$ with respect to the shift (see Definition~\ref{lue}), $h(\omega)$, exists and equals 
        $$
        H(p)=-p\log p-(1-p)\log(1-p).
        $$
By Proposition~\ref{rr}, an element $t\in\mathbb T$ is $p$-normal if and only if it is generic (under the transformation $R$) for a measure $\lambda_p$ such that the system $(\mathbb T,R,\lambda_p)$ is isomorphic to $(\{0,1\}^\N,\sigma,\mu_p)$, where $\mu_p$ is the
$(p,1-p)$-Bernoulli measure. So $h(t)$ also exists and equals $H(p)$. By Remark~\ref{thesame}~(1), the entropy $h(x)$ of a $p$-normal real number $x$ exists and equals $H(p)$ as well.
\end{rem}

Recall that, by Corollary~\ref{nordet1}(4), a number $x$ is normal in base 2 (i.e., $\frac12$-normal) if and only if $h(x)=\log2=H(\frac12)$. The following proposition shows that for $p\neq\frac12$ the situation is quite different. 

\begin{prop}
For any $p\in(0,1),\ p\neq\frac12$, there exists a real number $x$ such that $h(x)=H(p)$ but $x$ is not $p$-normal.   
\end{prop}
\begin{proof}
We first note that for any $p\in(0,1)$ there exists an ergodic system $(Y,\nu,S)$ not isomorphic to the $(p,1-p)$-Bernoulli system but having entropy $H(p)$. For example, one can take any system that is a product of the $(p,1-p)$-Bernoulli system with a nontrivial ergodic zero-entropy system. This product system is ergodic by disjointness of Bernoulli systems and zero entropy systems, and clearly has a nontrivial zero-entropy factor while all nontrivial factors of Bernoulli systems are isomorphic to Bernoulli systems and hence have positive entropy (see \cite{O}).
If $p\neq\frac12$ we have $h(\nu)=H(p)<\log 2$. Now, we can invoke Krieger's generator theorem \cite{Kr}, which states that for any integer $r\ge 2$ and any ergodic measure-preserving system $(X,\mu,T)$ with entropy $h(\mu)<\log r$ is isomorphic to $(\{0,1,\dots,r\}^\N,\mu',\sigma)$ for some ergodic $\sigma$-\im\ $\mu'$ on $\{0,1,\dots,r\}^\N$. 
In our case, this theorem implies that there exists an ergodic measure $\nu'$ on $\{0,1\}^\N$ such that the system $(\{0,1\}^\N,\nu',\sigma)$ is isomorphic to $(Y,\nu,S)$. So, $\nu'$ has entropy $H(p)$ and $(\{0,1\}^\N,\nu',\sigma)$ is not isomorphic to any Bernoulli system. Let $\omega\in\{0,1\}^\N$ be generic under $\sigma$ for $\nu'$. Any real number $x$ whose binary alias $\omega_2(x)$ equals $\omega$ satisfies the claim of the proposition.
\end{proof}

\begin{prop}\label{pnor+q} \emph{(cf.\ Corollary~\ref{rauzy}(3)).}
Let $y\in\R$ be a deterministic number such that its fractional part $\{y\}\in\mathbb T$ is not generic under the transformation $R(t)=2t$, $t\in\mathbb T$, for the Dirac measure $\delta_0$ concentrated at $0$. If $x\in\R$ is $p$-normal for $p\neq\frac12$ then $x+y$ is {\bf not} $p$-normal and, moreover, it is not $p'$-normal for any $p'\in(0,1)$. 
\end{prop}

\begin{rem}\label{65} The following argument shows that the assumption that $\{y\}$ is not generic for $\delta_0$ cannot be dropped. Suppose that $\{y\}$ is generic for $\delta_0$. Then the pair $(\{x\},\{y\})$ is generic in the product system $(\mathbb T\times\mathbb T,R\times R)$ for the product measure $\lambda_p\times\delta_0$. Indeed, since $\delta_0$ is concentrated at one point, it is clear that $\lambda_p\times\delta_0$ is the only joining of $\lambda_p$ and $\delta_0$. (Alternatively, one can use disjointness between Bernoulli systems and zero entropy systems.) Then, by Remark~\ref{mxmy}, $\{x\}+\{y\}$ (summation in $\mathbb T$) is generic for the measure $\nu$ on $\mathbb T$ which is the image of $\lambda_p\times\delta_0$ via the factor map $(t,s)\mapsto t+s$, $t,s\in\mathbb T$. But for $(\lambda_p\times\delta_0)$-almost every pair $(t,s)$ we have $s=0$, so $\nu=\lambda_p$, which implies that $\{x\}+\{y\}$ is $p$-normal. By Definition~\ref{pns}, the real number $x+y$ is $p$-normal. 
\end{rem}
\begin{rem}\label{66} Note that if a number $y\in\R$ has the property that $\{y\}$ is generic for $\delta_0$ then the binary alias $\omega_2(y)$ of $y$ consists essentially of very long blocks of $0$'s and very long blocks of $1$'s (long blocks of $0$'s are responsible for elements of the orbit of $\{y\}$ approaching $0$ counterclockwise, while long blocks of $1$'s are responsible for elements of the orbit of $\{y\}$ approaching $0$ clockwise). More precisely, the following holds: $\{y\}$ is generic for $\delta_0$ if and only if the block $01$ (in fact, any finite block containing both $0$ and $1$) occurs in the binary alias $\omega_2(y)$ of $y$ with frequency zero. Indeed, one implication follows immediately from the fact that the cylinder $[01]$ has $\delta_0$-measure zero. For the other implication suppose that $01$ occurs in $\omega_2(y)$ with frequency zero. Then $\omega_2(y)$ consists essentially (i.e., after dropping a sub\sq\ of density zero) of arbitrarily long constant blocks (either just $0$'s or just $1$'s). This implies that any measure which is quasi-generated in the system $(\{0,1\}^\N,\sigma)$ by $\omega_2(y)$ is a convex combination of $\delta_{\bar 0}$ and $\delta_{\bar 1}$ (the measures concentrated at the constant \sq s $\bar 0 = 000\dots$ and $\bar 1 = 111\dots$). But since the map $\phi_2:\{0,1\}^\N\to\mathbb T$ (see Proposition~\ref{rr}) sends both $\bar 0$ and $\bar 1$ to $0$, the adjacent map $\phi_2^*$ on \im s sends both $\delta_{\bar 0}$ and $\delta_{\bar 1}$ to $\delta_0$. Since the adjacent map is affine it sends the convex hull spanned by these two measures to $\delta_0$. Now, since $\{y\}=\phi_2(\omega_2(y))$, we obtain that $\{y\}$ is generic for $\delta_0$.
\end{rem}

\begin{proof}[Proof of Proposition~\ref{pnor+q}]
Assume that $p>\frac12$ (the proof for $p<\frac12$ is similar and is omitted). We begin with the observation that, by Corollary~\ref{rauzy}(2) and Remark~\ref{thesame}~(1), $h(x+y)=h(x)=H(p)$. So, if $x+y$ was $p'$-normal for some $p'$ then, by Remark~\ref{ren}, we would have $H(p')=H(p)$ and hence either $p'=p$ of $p'=1-p$. We will exclude both possibilities.

By Definition~\ref{pns}, it is enough to show that the binary alias $\omega_2(x+y)$ of the sum $x+y$ is neither $p$-normal nor $(1-p)$-normal in $\{0,1\}^\N$. We let $\mu_p$ denote the $(p,1-p)$-Bernoulli measure on $\{0,1\}^\N$. Choose $l\in\N$ so that 
\begin{equation}\label{q1q}
\left(\frac{1-p}p\right)^l<p.
\end{equation}
If follows from Remark~\ref{66} (and from the assumption made on $y$) that there exists a block $B$ ending with $0$, in which $1$ occurs $l$ times, and such that $\nu([B])>0$ for some measure $\nu$ quasi-generated by $\omega_2(y)$ along a sub\sq\ $(n_k)_{k\ge1}$. We denote by $N$ the length of $B$ (note that $N>l$). Since $y$ is deterministic, $h(\nu)=0$. By disjointness of Bernoulli systems from zero entropy systems, the pair $(\omega_2(x),\omega_2(y))$ is quasi-generic (generic along $(n_k)_{k\ge1}$) for the product measure $\mu_p\times\nu$. Suppose $x+y$ is $p$-normal. Then the pair $(\omega_2(x+y),\omega_2(y))$ also generates (along $(n_k)_{k\ge1}$) the product measure $\mu_p\times\nu$. This implies that the pair of blocks $(1^N,B)$ occurs in $(\omega_2(x+y),\omega_2(y))$ with frequency, evaluated along $(n_k)_{k\ge1}$, equal to $p^N\cdot\nu([B])$. More precisely, we have
\begin{multline*}p^N\cdot\nu([B])=\\
\lim_{k\to\infty}\frac 1{n_k}|\{n\in[1,n_k]: \text{ $(1^N,B)$ occurs in $(\omega_2(x+y),\omega_2(y))$ at the position $n$}\}|.
\end{multline*}

On the other hand, the pair of blocks $(1^N,B)$ occurs in $(\omega_2(x+y),\omega_2(y))$ at some position $n$ if and only if $B$ occurs in $\omega_2(y)$ starting at the position $n$ and one of the following two mutually exclusive cases takes place:
\begin{enumerate}
\item in the summation $\omega_2(x)\pl\omega_2(y)$ the carry does not occur at the position $n+N-1$ and $\omega_2(x)|_{[n,n+N-1]}=\tilde B$, where $\tilde B$ is defined by $\tilde B(i)=1-B(i)$, $i=1,2,\dots,N$,
\item in the summation $\omega_2(x)\pl\omega_2(y)$ the carry occurs at the position $n+N-1$ and  $\omega_2(x)|_{[n,n+N-1]}=\tilde B'$, where $\tilde B'$ coincides with $\tilde B$ at all coordinates except that at the last coordinate it has $0$ (while $\tilde B$ has there a $1$). 
\end{enumerate}
Here is the illustration for the case (1):
\begin{align*}
\omega_2(y) = & \dots\underset{B}{\underbrace{0100100110}}\,00\dots\\
\omega_2(x) = & \dots\underset{\tilde B}{\underbrace{1011011001}}\,10\dots\\
\omega_2(x+y) = &\dots 1111111111\dots 
\end{align*}
Here is the illustration for the case (2): 
\begin{align*}
\omega_2(y) = & \dots\underset{B}{\underbrace{0100100110}}\,1\dots\\
\omega_2(x) = & \dots\underset{\tilde B'}{\underbrace{1011011000}}\,1\dots\\
\omega_2(x+y) = &\dots 1111111111\dots 
\end{align*}

In either case, whenever the pair of blocks $(1^N,B)$ occurs in $(\omega_2(x+y),\omega_2(y))$ at some position $n$, then, in $(\omega_2(x),\omega_2(y))$, at the position $n$, there occurs the pair of blocks $(\tilde B'',B)$, where $\tilde B''$ is the block of length $N-1$ obtained from $\tilde B$ by dropping the last digit~$1$. Since $(\omega_2(x),\omega_2(y))$ generates (along $(n_k)_{k\ge1}$) the product measure $\mu_p\times\nu$, the pair of blocks $(\tilde B'',B)$ occurs in $(\omega_2(x),\omega_2(y))$ with frequency, evaluated along $(n_k)_{k\ge1}$, equal to $p^{N-l-1}(1-p)^{l}\cdot\nu([B])$. We have obtained the inequality
\begin{equation}\label{spr}
p^{N-l-1}(1-p)^{l}\cdot\nu([B])\ge p^N\cdot\nu([B]),
\end{equation}
and thus $(1-p)^l\ge p^{l+1}$, i.e., $(\frac{1-p}p)^l\ge p$, which is a contradiction with \eqref{q1q}. This contradiction implies that $x+y$ is not $p$-normal. 

The proof that $x+y$ is not $(1-p)$-normal is similar, with one modification:
we choose $B$ so that it ends with a $1$ (rather than $0$) and contains $l+1$ digits $1$ (including the last digit of $B$). Then, arguing as in the preceding case, we obtain that the occurrence of the pair of blocks $(0^N,B)$ in $(\omega_2(x+y),\omega_2(y))$ implies the occurrence of the pair of blocks block $(\tilde B'',B)$ in $(\omega_2(x), \omega_2(y))$ ($\tilde B''$ is defined as before, as the ``mirror'' of $B$ with the last symbol dropped). If $x+y$ was $(1-p)$-normal, the measure generated by $x+y$ would assign to the cylinder $[0^N]$ the value $p^N$ and we would obtain again the inequality \eqref{spr}, which leads to a contradiction.
\end{proof}

\begin{prop}\label{pq}\emph{(cf.\ Corollary~\ref{indep})}. Fix $p\in(0,1)$, $p\neq\frac12$.
If $x,y\in\R$ are independent in base $2$ (see Definition~\ref{indrn}) $p$-normal numbers then $x+y$ is not $p'$-normal for any $p'\in(0,1)$. 
\end{prop}

\begin{proof}
Let $\xi_p$ be the \im\ on $\{0,1\}^\N$ which is the factor of $\mu_p\times\mu_p$ via the map $(\omega,\tau)\mapsto \omega\pl\tau$ from $\{0,1\}^\N\times\{0,1\}^\N$ onto $\{0,1\}^\N$ (this map is continuous except on a countable set, hence it is a measurable factor map). Since $x$ and $y$ are independent, so are their fractional parts $\{x\},\{y\}$, and so are the binary aliases $\omega_2(x), \omega_2(y)$ (see Remark~\ref{lkj} and Definition~\ref{indrn}), which implies that the pair $(\omega_2(x),\omega_2(y))$ is generic in $\{0,1\}^\N\times\{0,1\}^\N$ for $\mu_p\times\mu_p$, and hence the image of this pair via the factor map $(\omega,\tau)\mapsto \omega\pl\tau$, i.e., $\omega_2(x)\pl\omega_2(y)$, is generic for $\xi_p$. To prove the statement in question we will first show that $\xi_p\neq\mu_p$ and then that $\xi_p\neq\mu_{p'}$ for any other $p'\in(0,1)$. 

Recall that $\phi_2:\{0,1\}^\N\to\R$ is defined in Proposition~\ref{rr} (except at one point which we can disregard) by 
\begin{equation}\label{repeated}
\phi_2((a_n)_{n\ge1})=\sum_{n=1}^\infty \frac{a_n}{2^n}\in[0,1).
\end{equation}

Let us view $\{0,1\}^\N\times\{0,1\}^\N$ as a probability space equipped with the measure $\mu_p\times\mu_p$. The elements of this space are pairs $(\omega,\tau)$, where 
$\omega=(a_n)_{n\ge1}$ and $\tau=(b_n)_{n\ge1}$ are elements of $\{0,1\}^\N$. 
Consider the following two events (i.e., subsets of this probability space):
\begin{itemize}
\item $\mathfrak A=\{(\omega,\tau)\in \{0,1\}^\N\times\{0,1\}^\N: \phi_2(\omega)+\phi_2(\tau)\ge1\}$,
\item $\mathfrak B=\{(\omega,\tau)\in \{0,1\}^\N\times\{0,1\}^\N: \phi_2(\sigma(\omega))+\phi_2(\sigma(\tau))\ge1\}$, 
\end{itemize}
where $+$ stands for the usual addition of real numbers. 

Let $P$ denote the probability of the event $\mathfrak B$, i.e.,
$$
P=(\mu_p\times\mu_p)(\mathfrak B).
$$
Further, let us also consider the partition the space $\{0,1\}^\N\times\{0,1\}^\N$ by the following eight events ($\mathfrak B^{\mathsf c}$ denotes the complement of $\mathfrak B$): 
\begin{enumerate}
    \item[] $\mathfrak C_1=\{(\omega,\tau)\in\mathfrak B^{\mathsf c},\ a_1=0,\ b_1=0\}$,
    \item[] $\mathfrak C_2=\{(\omega,\tau)\in\mathfrak B^{\mathsf c},\ a_1=0,\ b_1=1\}$,
    \item[] $\mathfrak C_3=\{(\omega,\tau)\in\mathfrak B^{\mathsf c},\ a_1=1,\ b_1=0\}$,
    \item[] $\mathfrak C_4=\{(\omega,\tau)\in\mathfrak B^{\mathsf c},\ a_1=1,\ b_1=1\}$,
    \item[] $\mathfrak C_5=\{(\omega,\tau)\in\mathfrak B,\ a_1=0,\ b_1=0\}$,
    \item[] $\mathfrak C_6=\{(\omega,\tau)\in\mathfrak B,\ a_1=0,\ b_1=1\}$,
    \item[] $\mathfrak C_7=\{(\omega,\tau)\in\mathfrak B,\ a_1=1,\ b_1=0\}$,
    \item[] $\mathfrak C_8=\{(\omega,\tau)\in\mathfrak B,\ a_1=1,\ b_1=1\}$.
\end{enumerate}

Let $q=1-p$ and $Q=1-P$. Because the event $\mathfrak B$ is independent of the events $\{(\omega,\tau):a_1=1\}$ and $\{(\omega,\tau):b_1=1\}$ (which clearly are also independent of each other), the probabilities of the events $\mathfrak C_1,\mathfrak C_2,\dots,\mathfrak C_8$ are $Qq^2$, $Qpq$, $Qpq$, $Qp^2$, $Pq^2$, $Ppq$, $Ppq$, $Pp^2$, respectively. Observe that $\mathfrak C_4\cup\mathfrak C_6\cup\mathfrak C_7\cup\mathfrak C_8\subset\mathfrak A$. Indeed, if $(\omega,\tau)\in\mathfrak C_4\cup\mathfrak C_8$ then $\phi_2(\omega)\ge\frac12$ and $\phi_2(\tau)\ge\frac12$, so $\phi_2(\omega)+\phi_2(\tau)\ge1$. If $(\omega,\tau)\in\mathfrak C_6$ then although $\phi_2(\omega)<\frac12$, the fact that $(\omega,\tau)\in\mathfrak B$ implies that $\frac12-\phi_2(\omega)\le\phi_2(\tau)-\frac12$, and hence $\phi_2(\omega)+\phi_2(\tau)\ge1$ as well.
By a similar argument, we have $\mathfrak C_1\cup\mathfrak C_2\cup\mathfrak C_3\cup\mathfrak C_5\subset\mathfrak A^{\mathsf c}$, which implies that 
\begin{equation}\label{sum}
\mathfrak C_4\cup\mathfrak C_6\cup\mathfrak C_7\cup\mathfrak C_8=\mathfrak A.
\end{equation}

By invariance of $\mu_p\times\mu_p$ under $\sigma\times\sigma$ and since $\mathfrak B=(\sigma\times\sigma)^{-1}(\mathfrak A)$, we have $(\mu_p\times\mu_p)(\mathfrak A)=(\mu_p\times\mu_p)(\mathfrak B)=P$. Thus, by summing the probabilities of the events $\mathfrak C_4$, $\mathfrak C_6$, $\mathfrak C_7$ and $\mathfrak C_8$, we obtain the following equation:
$$
P=Qp^2+P(2pq+p^2).
$$
After substituting $Q=1-P$, we get
$$
P=p^2+2Ppq,
$$
which implies 
\begin{equation}\label{n1}
P=\frac{p^2}{p^2+q^2},\ \ Q=\frac{q^2}{p^2+q^2}.
\end{equation}

Given a pair $(\omega,\tau)$, let $\rho=\omega\pl\tau$, $\rho=(c_n)_{n\ge1}\in\{0,1\}^\N$. By a reasoning similar to the one above derivation of \eqref{sum}, one can check that $c_1=1$ if and only $(\omega,\tau)\in \mathfrak C_2\cup\mathfrak C_3\cup\mathfrak C_5\cup\mathfrak C_8$. Recall that $\xi_p$ is the image of $\mu_p\times\mu_p$ via the factor map $(s,t)\mapsto s\pl t$. Thus 
\begin{multline}\label{noo}
p':=\xi_p(\{\rho:c_1=1\})=(\mu_p\times\mu_p)(\{(\omega,\tau):c_1=1\})=2Qpq+P(p^2+q^2)\\
=p^2+\frac{2pq^3}{p^2+q^2}.
\end{multline}

Now, the equation $p'=p$ has in $(0,1)$ only one solution, $p=\frac12$. Indeed, we have
\begin{multline*}
p^2+\frac{2pq^3}{p^2+q^2}=p\iff p+\frac{2q^3}{p^2+q^2}=1\iff\frac{2q^3}{p^2+q^2}=q\iff 2q^2=p^2+q^2\iff\\\iff p=q=\tfrac12.
\end{multline*}
So, unless $p=\frac12$, $p'$ is different from $p$ and then $\xi_p\neq\mu_p$, which implies that $x+y$ is not $p$-normal. 

But \emph{a priori} $\xi_p$ could equal $\mu_{p'}$ and hence $x+y$ could be $p'$-normal (indeed, since $p'=\xi_p(\{c:c_1=1\})$, $\mu_{p'}$ is the only possible Bernoulli measure which $\xi_p$ could match). We will presently see that this is not the case. In fact, we will prove that $\xi_p$ is not a Bernoulli measure, because
the coordinates $c_1$ and $c_2$ (viewed as 0-1-valued random variables on the probability space $(\{0,1\}^\N,\xi_p)$) are not independent. More precisely, we will show that
$$
p_0':=\xi_p(\{\rho:c_1=1\}|c_2=0)\neq \xi_p(\{\rho:c_1=1\})=p'.
$$

We have
$$
\xi_p(\{\rho:c_1=0\})=1-p'=q^2+\frac{2qp^3}{p^2+q^2}.
$$
Observe that $\mathfrak A^{\mathsf c}\cap\{(\omega,\tau):c_1=0\}=\mathfrak C_1$. By independence, the probability of $\mathfrak C_1$ equals $Qq^2=\frac{q^4}{p^2+q^2}$. Dividing this number by $1-p'$ we get the conditional probability of $\mathfrak A^{\mathsf c}$ with respect to the event $\{(\omega,\tau):c_1=0\}$:
$$
(\mu_p\times\mu_p)(\mathfrak A^{\mathsf c}|c_1=0)=\frac{\frac{q^4}{p^2+q^2}}{1-p'}=\frac{q^2}{p^2+q^2+2\frac{p^3}q}.
$$
Using invariance of $\mu_p\times\mu_p$ under $\sigma\times\sigma$ again, we also get
\begin{align}
&Q_0:=(\mu_p\times\mu_p)(\mathfrak B^{\mathsf c}|c_2=0)=\frac{q^2}{p^2+q^2+2\frac{p^3}q}, \label{oj1}\\
&P_0:=(\mu_p\times\mu_p)(\mathfrak B^{\mathsf c}|c_2=1)=1-Q_0=\frac{p^2+\frac{2p^3}q}{p^2+q^2+2\frac{p^3}q}.\label{oj2}
\end{align}

We have
\begin{multline*}
p_0':=\xi_p(\{\rho:c_1=1\}|c_2=0)=(\mu_p\times\mu_p)(\mathfrak C_2\cup\mathfrak C_3\cup\mathfrak C_5\cup\mathfrak C_8|c_2=0).
\end{multline*}
The conditional probabilities of the sets $\mathfrak C_2$, $\mathfrak C_3$, $\mathfrak C_5$ and $\mathfrak C_8$ are equal to $Q_0pq$, $Q_0pq$, $P_0q^2$ and $P_0p^2$, respectively. Summing up these probabilities, we obtain a formula similar to \eqref{noo}:
\begin{equation}\label{nooo}
p_0'=2Q_0pq+P_0(p^2+q^2).
\end{equation}

Thus, we can write
\begin{equation}\label{no11}
p'-p_0'=2pq(Q-Q_0)+(p^2+q^2)(P-P_0).
\end{equation}
Since $2pq+(p^2+q^2)=1$, the right hand side of \eqref{no11} can be viewed as a convex combination of the numbers $(Q-Q_0)$ and $(P-P_0)$. Note that $(P-P_0)=-(Q-Q_0)$, i.e., these numbers lie symmetrically around zero. By comparing \eqref{n1} and \eqref{oj1} we see that $(Q-Q_0)>0$ (and hence $(P-P_0)<0$). This means that the convex combination representing $p'-p_0'$ equals zero exclusively when the coefficients $2pq$ and $(p^2+q^2)$ are both equal to $\frac12$. But this happens only when $p=\frac12$, otherwise $2pq<\frac12$ (and hence $(p^2+q^2)>\frac12$), therefore $p_0'>p'$, which ends the proof.
\end{proof}

\begin{rem}
Using the same type of calculations (albeit much more tedious), one can show that if $x$ is $p_1$-normal, $y$ is $p_2$-normal ($p_1,p_2\in(0,1)$), and $x, y$ are independent, then, unless either $p_1=\frac12$ or $p_2=\frac12$ (in which case $x+y$ is normal by Corollary~\ref{indep}), $x+y$ is not $p'$-normal for any $p'\in(0,1)$.   
\end{rem}

\begin{thm}\label{xq1}\emph{(cf.\ Proposition~\ref{xq}(2)).} 
Let $x\in\R$ be $p$-normal with $p\neq\frac12$ and let $n$ be a positive integer which is not a power of $2$. Then $nx$ and $\frac xn$ are not $p$-normal. 
\end{thm}

\begin{rem}\label{2xa} If $n=2^k$ with $k\in\N$ then $x$ is $p$-normal if and only if so is $nx$, if and only if so is $\frac nx$. To see this note that the binary alias $\omega_2(nx)$ of $nx$ equals $\sigma^n(\omega_2(x))$, where $\omega_2(x)$ is the binary alias of $x$. Since the shift preserves $p$-normality (by both image and preimage), we conclude that $nx$ is $p$-normal if and only if so is $x$. Now let $y=\frac xn$. Then $x=ny$ and, by the preceding argument, $y=\frac xn$ is $p$-normal if and only if $ny=x$ is $p$-normal.
\end{rem}

The proof of Theorem~\ref{xq1} makes use of the following theorem by Dan Rudolph \cite{Ru}:
\begin{thm}\label{Rud}
Let $R, S:\mathbb T\to\mathbb T$ be defined by $R(t)=mt$, $S(t)=nt$, where $m>1$ and $n>1$ are relatively prime natural numbers. Let $\mu$ be a measure on $\mathbb T$ invariant and ergodic under the semigroup generated by $R$ and $S$. Then either $\mu=\lambda$ or $\mu$ has entropy zero with respect to $R$ and with respect to $S$.
\end{thm}

\begin{proof}[Proof of Theorem~\ref{xq1}] Since we are dealing with binary aliases, we will apply Ru\-dolph's theorem to $m=2$. Next, we claim that we can restrict to numbers $n$ that are odd (and larger than~1). Indeed, we can represent any $n>1$ as $2^kn'$, where $k\ge0$ and $n'>1$ is odd. Then $nx=n'x'$ and $\frac xn=\frac {x''}{n'}$, where $x'=2^kx$ and $x''=\frac x{2^k}$ are $p$-normal by Remark~\ref{2xa}. 

In view of Definition~\ref{pns}(3) we can replace $x$ by its fractional part $\{x\}=t_0\in\mathbb T$ and work with the system $(\mathbb T,R,\lambda_p)$ isomorphic to the Bernoulli system $(\{0,1\}^\N,\sigma,\mu_p)$ via the map $\phi_2:\{0,1\}^\N\to\mathbb T$. The $p$-normality of $x$ is equivalent to $p$-normality of $t_0$. This, in turn is equivalent to the fact that $t_0$ is generic for~$\mu_p$. The mapping $t\mapsto nt$ is a \tl\ factor map of the system $(\mathbb T,R)$ onto itself, hence it sends the measure $\mu_p$ to some $R$-\im\ $\mu$. Since $t_0$ is generic for $\mu_p$, $nt_0$ is generic for $\mu$ (see Remark~\ref{mxmy}). If $nt_0$ was $p$-normal, we would have $\mu_p=\mu$ implying that $\mu_p$ is invariant under the maps $R:t\mapsto 2t$ and $S:t\mapsto nt$. Clearly, $2$ and $n$ are relatively prime, $\mu_p$ is ergodic with respect to $R$ (and thus also with respect to the action of the semigroup generated by $R$ and $S$) and $\mu_p$ has positive entropy with respect to $R$. By Theorem~\ref{Rud} $\mu_p$ has to be the Lebesgue measure. This, however, is not true for $p\neq\frac12$, because in this case $h(\mu_p)=H(p)<\log 2=h(\lambda)$ (see Remark~\ref{ren}).
Thus $nt_0$ (equivalently $nx$) is not $p$-normal. 

Now, if $\frac xn$ was $p$-normal, then, by the above argument, $x=n\frac xn$ would not be $p$-normal, contradicting the assumption of the theorem.
\end{proof}

We believe that the answer to the following question is positive:

\begin{ques}Is it true that if $x\in\R$ is $p$-normal with $p\neq\frac12$ then $qx$ is not $p$-normal for any positive rational $q$ which is not a power (positive or negative)~of~$2$?
\end{ques}

\section{Behavior of normal and deterministic numbers under multiplication}\label{S7}

It was proved in Section~\ref{Sn} (see Corollary~\ref{affin}) that the lower and upper entropies of a real number $x$ are preserved under the transformation $L_{q,y}(x)= qx+y$, where $q$ is a nonzero rational number and $y$ is a deterministic number. In particular, $L_{q,y}$ preserves normality and determinism.
It is natural to ask whether a transformation of a more general kind, $L_{y_1,y_2}$, where $y_1\neq 0$ and $y_2$ are deterministic numbers, has the same properties.

As we will see in this section, the answer to this question is a sound ``no''. We will prove the following theorem which demonstrates that multiplication by a nonzero deterministic number can reduce the entropy of a real number from $\log 2$ to $0$:
\begin{thm}\label{xyd} 
There exist real numbers $x,y$ with $x\in\mathcal N(2)$, $y\in\mathcal D(2),\ y\neq 0$, such that $xy\in\mathcal D(2)$.
\end{thm}

In addition, we will show that conversely, multiplication by a deterministic number can bring up the entropy of a real number from $0$ to $\log 2$: 

\begin{thm}\label{yyn}
There exist numbers $y_1, y_2\in\mathcal D(2)$ such that $y_1y_2\in\mathcal N(2)$.
\end{thm}

The structure of this section is as follows: in Subsection~\ref{7.1} we introduce some preliminary notions and results including a special ordering of the family $\{0,1\}^n$ of all blocks of length $n$, called \emph{Gray code}. The numbers $x$ and $y$ appearing in Theorem~\ref{xyd} are constructed in Subsections~\ref{7.2} and~\ref{2det}, correspondingly. In fact, in Subsection~\ref{2det} we construct two deterministic numbers that can play the role of $y$ in Theorem~\ref{xyd}. The first construction provides a \emph{trivially deterministic} number $y$, in the sense that the digit 1 in the binary expansion of $y$ occurs with frequency zero.
Because  trivially deterministic \sq s are in some sense exceptional\footnote{
In the papers of B.\ Weiss \cite{W2} and T.\ Kamae \cite{K1} it is proved that an increasing \sq\ $S=\{n_1,n_2,\dots\}$ of natural numbers of positive lower density \emph{preserves normality} in the sense that whenever $x=(x_n)_{n\ge1}\in\Lambda^\N$ is normal then $x|_S=(x_{n_k})_{k\ge1}$ is also normal, if and only if the indicator function $\mathbbm 1_S\in\{0,1\}^\N$ is deterministic.
Note that this theorem does not apply if $\mathbbm1_S$ is trivially deterministic. In fact, it is easy to see that whenever $\mathbbm1_S$ is trivially deterministic then $S$ does not preserve normality in the sense of Kamae--Weiss.}, we also provide a second construction (which is achieved by modifying the first one), in which $y$ is replaced by a deterministic number $z$, which has positive frequency of occurrences of the block $01$ in its binary expansion. Then the fractional part $\{z\}$ is not generic (under $R$) for $\delta_0$, so it does not fall in the exceptional class of deterministic numbers which we needed to eliminate in Proposition~\ref{pnor+q} (see Remark~\ref{66}). 
Subsection~\ref{sxyz} contains the proof of Theorem~\ref{xyd}. 
Finally, Subsection~\ref{xyn} contains the proof of Theorem~\ref{yyn}.

\subsection{Gray code}\label{7.1}
\begin{itemize}
\item Given an $n\in\N$, consider the family ${\B}_n=\{0,1\}^n$ of all binary blocks of length~$n$. We will say that $B_1,B_2,B_3,\dots,B_{2^n}$ is an \emph{ordering} of ${\B}_n$ if for each $B\in{\B}_n$ we have $B=B_l$ for exactly one $l\in\{1,2,3,\dots,2^n\}$.
\item Given $n\ge2$, a block $B=(b_1b_2\dots b_n)\in{\B}_n$ and an integer $N\in[1,n-1]$, the $N$th prefix of $B$ is the block $B|_{[1,N]}=(b_1b_2\dots b_N)$ and its $N$th suffix is the block $B|_{[N+1,n]}=(b_{N+1}b_{N+2}\dots b_n)$. The notion of the $N$th prefix applies naturally also to infinite unilateral \sq s.
\item For $B\in{\B}_n$ by $\tilde B$ we will denote the ``mirror'' of $B$, that is, $\tilde B$ has $1$'s and $0$'s exactly where $B$ has $0$'s and $1$'s, respectively. 
\end{itemize}

\begin{lem}\label{zig}For any $n\ge1$ and $B\in{\B}_n$ there exists an ordering of ${\B}_n$,\break $B_1,\ B_2,\ B_3,\ \dots,\ B_{2^n-1},\ B_{2^n}$, such that 
\begin{enumerate}
        \item $B_1=B$,
	\item for each $l=1,2,\dots,2^n-1$ the blocks $B_l$ and $B_{l+1}$ differ at only one place,
	\item for each $i=1,2,\dots,n-1$ and $j=0,1,2,\dots 2^{n-i}-1$, the $(n_k-i)$th suffixes (i.e., suffixes of length $i$) of the blocks 
	$$
	B_{j2^i+1},\ B_{j2^i+2},\ B_{j2^i+3},\ B_{j2^i+4},\ \dots,\ B_{j2^i+2^i-1},\ B_{j2^i+2^i}
	$$
	form an ordering of ${\B}_i$, while their $(n_k-i)$th prefixes are all the same. 
	\end{enumerate}
\end{lem}

\begin{rem}
When $B=000\dots0$ is the block of $n$ zeros, the ordering described in Lemma~\ref{zig} is known under the name of \emph{Gray code}. 
\end{rem}

\begin{rem}\label{gran}
In (3), since the $(n_k-i)$th prefixes are the same, the ordering of $\B_i$ formed by the $(n_k-i)$th suffixes has the property that two neighboring blocks differ at only one place.
\end{rem}

\begin{proof}[Proof of Lemma~\ref{zig}]
It suffices to prove this for the block $B=000\dots0$ of $n$ zeros. If $B$ is different, the appropriate ordering is obtained by adding (coordinatewise and modulo $2$) $B$ to each $B_l$, $1\le l\le2^n$, constructed for the block of zeros.

We will proceed inductively. For $n=1$ we have only two blocks and we order them as follows: $B_1=0,\ B_2=1$. Suppose that for some $n\ge1$ we have the ordering $B_1,B_2,\dots,B_{2^n}$ of ${\B}_n$ starting with $B_1=000\dots0$ ($n$ zeros) and satisfying~(2) and~(3). Then, define an ordering of ${\B}_{n+1}$ by:
$$
0B_1,\,0B_2,\,\dots,\,0B_{2^n-1},\,0B_{2^n},\,1B_{2^n},\,1B_{2^n-1},\,\dots,\,1B_2,\,1B_1.
$$
This ordering clearly satisfies (1),~(2) and~(3) for $n+1$ in place of $n$.
\end{proof}

\begin{lem}\label{zag}Let $n\in\N$ be even. Fix $B_1\in{\B}_n$ and let $B_1,B_2,\dots,B_{2^n}$ be an ordering of $\B_n$ such that any two neighboring blocks differ at only one place. Then the \sq\ of blocks
\begin{equation}\label{norder}
B_1,\tilde B_2,B_3,\tilde B_4,\dots B_{2^n-1},\tilde B_{2^n}
\end{equation}
is an ordering of ${\B}_n$.
\end{lem}
\begin{proof}
Notice that for each $l=1,2,\dots,2^n$ the blocks $B_l$ and $\tilde B_l$ differ at all $n$ places, which is an even number, hence the distance between $B_l$ and $\tilde B_l$ in the ordering $B_1,B_2,\dots,B_{2^n}$ is even. This implies that in the new \sq\ \eqref{norder} either both of them have a tilde or none. In the first case, they just switch places in the ordering (note that double tilde is no tilde). In the second case they do not change their positions. In conclusion, all blocks from ${\B}_n$ appear in the \sq\ \eqref{norder} exactly once, and hence this \sq\ is an ordering of~${\B}_n$.
\end{proof}

\subsection{Construction of a ``Champernowne-like'' binary \sq.}\label{7.2}
In this subsection we will construct a normal binary \sq\ $\kappa$ which has a special intricate structure and which will be instrumental in proving Theorems~\ref{xyd} and~\ref{yyn} in Subsections~\ref{sxyz} and~\ref{xyn}, respectively. 

We start by defining the block $B^1_1=01$ and denoting its length by $n_1$ (i.e., $n_1=2$). Inductively, once $B^k_1$ is defined and has length $n_k$ which is a power of $2$, we define $B^{k+1}_1$ as the concatenation
$$
B^{k+1}_1 = B^k_1\tilde B^k_2B^k_3\tilde B^k_4\dots B^k_{2^{n_k}-1}\tilde B^k_{2^{n_k}},
$$
where the blocks are ordered according to \eqref{norder} applied to ${\B}_{n_k}$, starting from $B^k_1$. The length of $B^{k+1}_1$ equals $n_k2^{n_k}$ (which is a power of $2$) and we denote it by $n_{k+1}$. Since, for each $k$,  $B^k_1$ is a prefix of $B^{k+1}_1$, the \sq\ of blocks $(B^k_1)_{k\ge1}$ converges (coordinatewise) to an infinite \sq\ in $\{0,1\}^\N$. 
\begin{defn}\label{kapa}
The binary \sq\ $\kappa$ is defined  as the coordinatewise limit of the blocks $B^k_1$. 
\end{defn}

Figure~\ref{f1} shows the initial part of $\kappa$ with complete blocks $B^1_1,B^2_1$ and a small part of $B^3_1$. 
\begin{figure}[h]
$$
\underset{B^3_1}{\underbrace{\underset{B^2_1}{\underbrace{\overset{B^1_1}{\overbrace{01}}\overset{\tilde B^1_2}{\overbrace{11}}\,\overset{B^1_3}{\overbrace{10}}\,\overset{\tilde B^1_4}{\overbrace{00}}}} \underset{\tilde B^2_2}{\underbrace{10000110}}\, \underset{B^2_3}{\underbrace{01111011}}\, \underset{\tilde B^2_4}{\underbrace{10000101}}\, \underset{B^2_5}{\underbrace{01111110}}\, \underset{\tilde B^2_6}{\underbrace{10000000}}\, \underset{B^2_7}{\underbrace{01111101}}\,\dots}}
$$
\caption{The \sq\ $\kappa$.}
\label{f1}
\end{figure}

\begin{thm}\label{kappanor}The \sq\ $\kappa\in\{0,1\}^\N$ is normal.
\end{thm}

\begin{proof} Given $m\ge1$ and $\eps>0$, a binary block $B$ will be called \emph{$(\eps,m)$-normal} if the densities of all blocks of length $m$ in $B$ (see \eqref{empirical}) are $\eps$-close to the ``correct'' value~$2^{-m}$. A~binary \sq\ is normal if and only if, for any $m\ge1$ and $\eps>0$, all its sufficiently long prefixes are $(\eps,m)$-normal. From now on we fix an integer $m$ and we abbreviate the term $(\eps,m)$-normal as just \emph{$\eps$-good}. For $\eps>0$, the following easy facts hold:
\begin{enumerate}
	\item A concatenation of sufficiently long $\eps$-good blocks is $2\eps$-good.
	\item For $n$ large enough, a concatenation of any ordering of ${\B}_n$ is $\eps$-good.
	\item If $n$ is large enough, $B_1,\,B_2,\,B_3,\,\dots,\,B_{2^n}$ is an ordering of ${\B}_n$ and\break$C_1,\,C_2,\,C_3,\,\dots,\,C_{2^n}$ are $\eps$-good blocks (no matter how long) then the ``alternating concatenation'' 
	$$
	C_1B_1C_2B_2\dots C_{2^n}B_{2^n}
	$$ 
	is $2\eps$-good.
	\item For small enough $\delta>0$, large enough $n$ and $B\in\B_n$ we have:
 \begin{enumerate} 
 \item if a block $B'$ which is obtained by removing from $B$ at most $n\delta$ symbols is $\eps$-good then $B$ is $2\eps$-good (when removing symbols from a block we ``close the gaps'', i.e., we shift the remaining parts of the block together, so that $(1-\delta)n\le |B'|\le n$),
 \item if a block $B'$ which is obtained by inserting between the symbols of $B$ at most $n\delta$ additional symbols (so that $n\le |B'|\le (1+\delta)n$) is $\eps$-good then $B$ is $2\eps$-good, 
 \item if a block $B'$ of length $n$ obtained by changing at most $n\delta$ symbols in $B$ is $\eps$-good then $B$ is $2\eps$-good.
\end{enumerate}
\end{enumerate}

We will need the following lemma concerning the blocks $B^{k+1}_1$ described in the construction of $\kappa$.  

\begin{lem}\label{claim} Given $\eps>0$, for small enough $\delta$ and large enough $k$, for each $1\le N\le n_{k+1}-1$ the $N$th prefix of $B^{k+1}_1$, $A=B^{k+1}_1|_{[1,N]}$, is either $\eps$-good or $N< n_{k+1}\delta$ (in the latter case we will say that the prefix is \emph{ignorable}).
\end{lem}

\begin{proof} Assume that $k$ is so large that $2^{-n_k}<\frac{\delta^2}2$ and that $B_1^{k+1}$, which is a concatenation of an ordering of $\B_{n_k}$, is $\frac\eps4$-good, by virtue of~(2). Assume that $N\ge n_{k+1}\delta$ (i.e., that the prefix is non-ignorable). The last two inequalities, together with the formula $n_{k+1}=n_k2^{n_k}$, imply that $2n_k<N\delta$. Thus, we can extend the prefix $A$ to the right by at most $N\delta$ terms, and create a slightly larger prefix $A'=B^{k+1}_1|_{[1,l_0n_k]}$ which is a complete concatenation of an even number of the blocks $B^k_l$ and their mirrors, that is 
$$
A'=B^k_1\tilde B^k_2B^k_3\tilde B^k_4\dots B^k_{l_0-1}\tilde B^k_{l_0}.
$$
Since $N\ge n_{k+1}\delta= n_k2^{n_k}\delta$, we have $l_0\ge2^{n_k}\delta$. Now, by (4b), it suffices to show that $A'$, is $\frac\eps2$-good. 

If $k$ is large enough then there exists $i>n_k(1-\delta)$ such that  $2^i\le 2^{n_k}\delta^2<l_0\delta$. Now we let $A''=B_1^{k+1}|_{[1,j_0n_k2^i]}$, where $j_02^i$ largest multiple of $2^i$ smaller than $l_0$. Note that $|A'|-|A''|<n_k2^i\le n_kl_0\delta=|A'|\delta$.
Thus, by (4a), it will be enough to show that $A''$, is $\frac\eps4$-good. The prefix $A''$ can be naturally divided into $j_0$ subblocks, each having length $n_k2^i$. We denote these subblocks by $C_j$ with $0\le j \le j_0-1$. Each $C_j$ is a concatenation of the form
$$
C_j=B^k_{j2^i+1}\tilde B^k_{j2^i+2}B^k_{j2^i+3}\tilde B^k_{j2^i+4}\dots B^k_{j2^i+2^i-1}\tilde B^k_{j2^i+2^i}=P_1S_1P_2S_2\dots P_{2^i}S_{2^i},
$$
where $P_l$ and $S_l$ are the $(n_k-i)$th prefix and $(n_k-i)$th suffix of $B^k_{j2^i+l}$ (for $l$ odd) or of $\tilde B^k_{j2^i+l}$ (for $l$ even), respectively. 
By Lemma~\ref{zig}(3), Remark~\ref{gran} and Lemma~\ref{zag}, the blocks $S_l$ form an ordering of ${\B}_i$. Note that since $i>n_k(1-\delta)$, by choosing $k$ even larger we can assure that, by (2), the concatenation of the blocks $S_l$ is $\frac\eps{16}$-good. Now, $C_j$ is obtained from this concatenation by inserting the missing $(n_k-i)$th prefixes $S_l$. Since $(n_k-i)<n_k\delta$, these prefixes have jointly less than $|C_j|\delta$ symbols. Thus, by (4b), every ``piece'' $C_j$ is $\frac\eps8$-good, and hence, by (1), $A''$ is $\frac\eps4$-good as desired.
\end{proof}

We continue with the proof of Theorem~\ref{kappanor}. We fix $\delta>0$ so small that (4) holds for large enough $n$ even when $\delta$ is replaced by $2\delta$. We also require that Lemma~\ref{claim} holds for $\delta$, with large enough $k$. 

For large $k$ the block $B^k_1$ is $\eps$-good, because it is a concatenation of an ordering of ${\B}_{n_{k-1}}$ (we can assume that $n_{k-1}$ is large enough as required in (2)). The block $B^k_2$ (and hence also $\tilde B^k_2$) is $2\eps$-good because it differs from $B^k_1$ only at the last place. We can argue in this manner, using the property (4c), up to $B^k_l$ (and $\tilde B^k_l$) as long as $B^k_l$ differs from $B^k_1$ at less than $n_k\delta$ terminal places. It follows from Lemma~\ref{zig} that, for each $i\in[1,n_k]$, the symbol at the position $n_k+1-i$ changes (i.e., differs from the $(n_k+1-i)$th symbol in $B^k_1$) for the first time in $B^k_{2^{i-1}+1}$. This means that $B^k_l$ differs from $B^k_1$ at at most $\log_2l$ terminal positions. So, the largest $l$ such that $B^k_l$ is guaranteed to be $2\eps$-good satisfies $\log_2l<n_k\delta$. In particular, we have shown that  
\begin{enumerate}
	\item[(5)] for $l<2^{n_k\delta}$ the block $B^k_l$ (and hence also $\tilde B^k_l$) is $2\eps$-good.
\end{enumerate}
\smallskip

In order to prove the theorem it suffices to show that the $N$th prefix of $\kappa$, $A=\kappa|_{[1,N]}$, is $8\eps$-good, for all $N$ large enough. So, we fix a large $N$ and we let $k$ be such that $n_k<N\le n_{k+1}$ ($k$ is the largest number such that the coordinate $N$ falls outside $B^k_1$). Since $N$ is large, so is $k$. We can thus assume that $k$ is so large that (in addition to validity of Lemma~\ref{claim}) the following two conditions hold:  
\begin{enumerate} 
\item[($\alpha$)] $n_k>\frac{2-\log\delta}\delta$,
\item[($\beta$)] the number $n=\lceil n_k\delta+\log_2\delta\rceil-1$ is large enough for the validity of (2) and (3).
\end{enumerate}
We need to consider three cases.

\smallskip\noindent{\bf Case 1.} $N\ge n_{k+1}\delta$. In this case $A$ is a non-ignorable prefix of $B^{k+1}_1$, which is $2\eps$-good by Lemma~\ref{claim} (see Figure~\ref{fig10}). 

\smallskip\noindent{\bf Case 2.} $N\le n_k2^{n_k\delta}$.
The coordinate $N$ falls within a block $B^k_l$ or $\tilde B^k_l$ (depending on the parity of $l$), with an $l$ satisfying $1<l< 2^{n_k\delta}$. We assume that $l$ is odd (the even case is similar). Then $A$ is the concatenation $B^k_1\tilde B^k_2\dots B^k_{l-2}\tilde B^k_{l-1}$ with a suffix $P$, which is a prefix of $B^k_l$ (or the entire block $B^k_l$), appended at the right end. By (5), the concatenation comprises just $2\eps$-good blocks, and hence, by (1), it is $4\eps$-good. It remains to consider the suffix $P$.
\begin{enumerate}
\item[(a)] If $P$ is an ignorable prefix of $B^k_l$ (i.e., shorter than $n_k\delta$) then $P$ is an ignorable suffix of $A$ as well, hence $A$ is $8\eps$-good by (4a) (see Figure~\ref{fig10}). 
\end{enumerate}
If $P$ is a non-ignorable prefix of $B^k_l$ then there are two further cases: 
\begin{enumerate}
\item[(b)] $P$ does not reach the coordinates where $B^k_l$ differs from $B^k_1$, or 
\item[(c)]it reaches there. 
\end{enumerate}
\noindent In the case (b), $P$ is identical as a non-ignorable prefix of $B^k_1$, and hence it is $2\eps$-good by Lemma~\ref{claim} (see Figure~\ref{fig10}). In the case (c), recall that $B^k_l$ differs from $B^k_1$ only at at most $n_k\delta$ terminal positions. Since $P$ reaches there, its lenght is at least $n_k(1-\delta)$. Because, by (5), $B^k_l$ is $2\eps$-good, $P$ is $4\eps$-good by (4b) (see Figure~\ref{fig10}). In either case, $P$ is $4\eps$-good and, by~(1), $A$ is $8\eps$-good.

\begin{figure}[H]
    \centering
    \includegraphics[width=\textwidth]{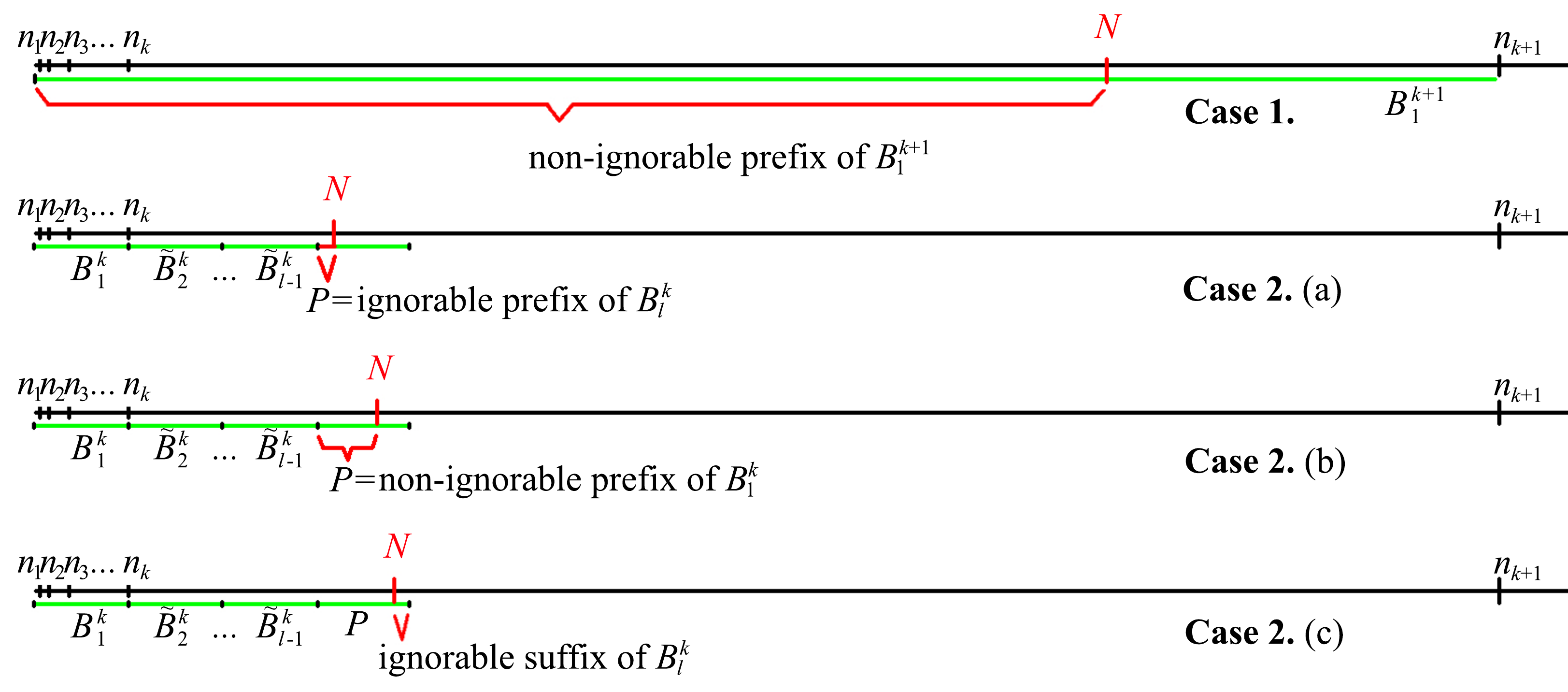}
    \caption{}
    \label{fig10}
\end{figure}

\smallskip\noindent{\bf Case 3.} $n_k2^{n_k\delta}<N<n_{k+1}\delta=n_k2^{n_k}\delta$. By ($\alpha$), we have $2^{-n_k\delta}<\delta$, and hence $n_k<N\delta$. Choose the largest $i\ge0$ such that $n_k2^i<N\delta\le2n_k2^i$. Notice that the above assumptions imply that $n_k2^{n_k\delta}\delta<N\delta<n_k2^{i+1}$, and hence 
$$
i>n_k\delta+\log_2\delta-1.
$$ 
Note that by ($\beta$), (2) and (3) hold for any ordering of $\B_i$ (a fact that will be useful later). 

Let now $N'=j_0n_k2^i$ be the largest multiple of $n_k2^i$ smaller than $N$ (note that $\lfloor\frac1\delta\rfloor<j_0\le\lceil\frac2\delta\rceil$). Then $N-N'< n_k2^i< N\delta$, so, by (4a), in order to show that $A$ is $8\eps$-good, it will be enough to show that the new prefix $A'=\kappa|_{[1,N']}$ is $4\eps$-good. 

The prefix $A'$ equals the concatenation $B^k_1\tilde B^k_2\dots\tilde B^k_{j_02^i}$. Let $s=\lceil\log_2(j_0)\rceil$. The prefix $\kappa|_{[1,N']}$ is contained in the (possibly longer) concatenation
$$
B^k_1\tilde B^k_2\dots \tilde B^k_{2^{s+i}}.
$$
By Lemma~\ref{zig}(3), the $(n_k-s-i)$th prefixes of all the blocks $B^k_l$, $l=1,2,\dots,2^{s+i}$ are the same, hence they are the same as the $(n_k-s-i)$th prefix of $B^k_1$.

Since $j_0\le\lceil\frac2\delta\rceil$, we have, by ($\alpha$),
$$
s=\lceil\log(j_0)\rceil<-\log\delta+2<n_k\delta.
$$
We remove from each block $B^k_l$ ($1\le l\le j_02^i$) the inner subblock of length $s$, $B^k_l|_{[n_k-s-i+1, n_k-i]}$, and denote the block obtained in this manner by $B'_l$. If we show that $B'_l$ is $2\eps$-good, this will imply, by (4a), that $B^k_l$ is $4\eps$-good (and so is $\tilde B^k_l$). Now, $B'_l$ consists of the $(n_k-s-i)$th prefix and $(n_k-i)$th suffix of $B^k_l$. By Lemma~\ref{zig}(3), for each $j=0,1,\dots j_0-1$, within the cluster of blocks $B_{j2^i+1},\tilde B_{j2^i+2},\dots,\tilde B_{{(j+1)}2^i}$, the $(n_k-i)$th suffixes form an ordering of $\B_i$.
In $A'$, these suffixes are mixed with the $(n_k-s-i)$th prefixes of the $B^k_l$'s and $\tilde B^k_l$'s. There are now two possibilities (see Figure~\ref{fig11}): 

(a) $n_k-s-i< n_k\delta$, or 

(b) $n_k-s-i\ge n_k\delta$.

\noindent In case (a), the $(n_k-s-i)$th prefixes of the blocks $B^k_l$ and $\tilde B^k_l$ are ignorable, so we can remove them from the blocks $B^k_l$ together with the inner subblocks $B^k_l|_{[n_k-s-i+1, n_k-i]}$. In this manner, by removing at most $2\delta|A'|$ symbols, 
$A'$ is reduced to a block $A''$ which is a concatenation of orderings of $\B_i$. Since (2) holds for $\B_i$, every such concatenation is $\eps$-good, and we conclude (by (1)) that $A''$ is $2\eps$-good. By (4a) (which is also valid with $2\delta$), $A'$ is $4\eps$-good, as required.

In case~(b), the $(n_k-s-i)$th prefix of each $B^k_l$ is $2\eps$-good, because it is equal to a non-ignorable prefix of $B^k_1$ (which is $2\eps$-good by Lemma~\ref{claim}). The mirrors of such prefixes are also $2\eps$-good, and thus we can use (3) to deduce that $A'$ is $4\eps$-good.
This ends the proof.
\end{proof}

\begin{figure}[H]
    \centering
    \includegraphics[width=\textwidth]{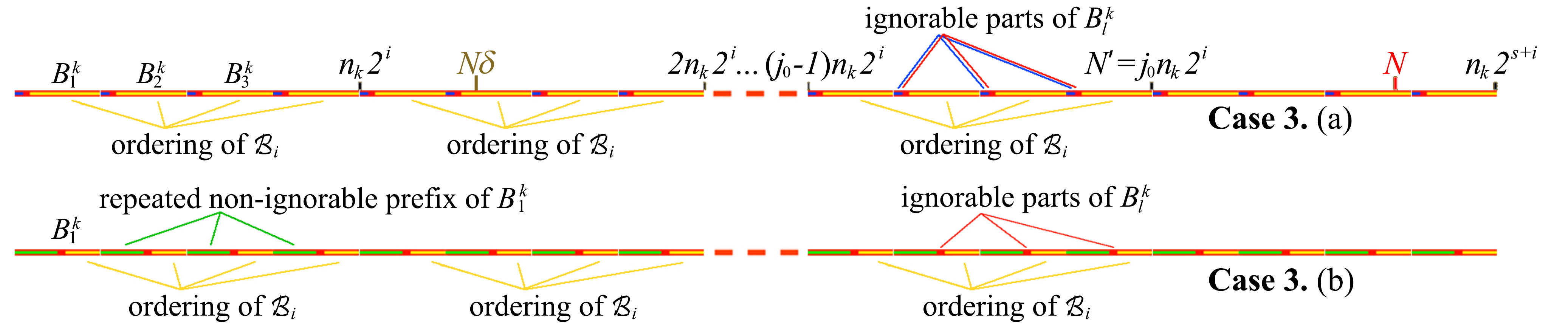}
    \caption{}
    \label{fig11}
\end{figure}

\subsection{Two special deterministic numbers.}\label{2det}
Given an increasing \sq\ of natural numbers $(n_k)_{k\ge1}$, let $\mathsf{FS}((n_k)_{k\ge1})$ denote the set of finite sums of $(n_k)_{k\ge1}$, that is,
$$
\mathsf{FS}((n_k)_{k\ge1})=\{n_{k_1}+n_{k_2}+\cdots+n_{k_i}, k_1<k_2<\cdots<k_i,i\in\N\}.
$$
Assume now that $(n_k)_{k\ge1}$ is the \sq\ defined in the preceding subsection (i.e., $n_1=2$, $n_{k+1}=n_k2^{n_k}$). Let $S=\{0\}\cup\mathsf{FS}((n_k)_{k\ge1})$ and let us write the elements of $S$ in the increasing order. Explicitly, we have
\begin{multline*}
S=\{s_0,s_1,s_2,\dots\}=\\
\{0,2,8,10,2048, 2050,2056,2058,2048\cdot2^{2048},2048\cdot2^{2048}+2,2048\cdot2^{2048}+8,\dots\}.
\end{multline*}
Observe that the density of $S$ is zero. 
Indeed, it is not hard to see that
$$
\bar d(S)=\limsup_{N\to\infty}\frac{|S\cap[0,N]|}{N+1}=\lim_{k\to\infty}\frac{|S\cap[0,N_k]|}{N_k+1},
$$
where $N_k=n_1+n_2+\cdots+n_k$. Note that $\frac{|S\cap[0,N_k]|}{N_k+1}=2^k/(1+n_1+n_2+\cdots+n_k)$, which obviously tends to zero. Thus, $\bar d(S)=0$.

Let $y$ be the number whose binary expansion matches the indicator function of $S$ (with the coordinate zero representing the integer part of $y$), i.e., 
$$
y=s_0.s_1s_2s_3\dots=1.010000010100000\dots.
$$
Since $S$ has density zero, $y$ is trivially deterministic. 

Let us remark here that generally, for real numbers $x$ and $y$, $\{xy\}$ need not equal $\{x\}\{y\}$. Since $y>1$, we cannot replace $y$ by its fractional part~$\{y\}$. For this reason, in what follows we must keep track of the binary dot and the integer part represented by the digit at the coordinate $0$ in the expansion of $y$ and numbers of the form $xy$. 

We also define $z=\frac43y$. By Corollary~\ref{affin}, $z$ is deterministic as well. 
 
\begin{lem}\label{1/2}
The block $01$ appears in the binary expansion of $z$ with frequency $\frac12$. 
\end{lem}

\begin{proof}\phantom{.}

\noindent{\bf Observation.} Let us call a finite (of length at least 2) or infinite \sq\ of alternating $0$'s and $1$'s (starting from either $0$ or $1$) a \emph{regular pattern}. Finite regular patterns are allowed to have even or odd length. By convention, any unknown, potentially non-regular finite pattern (block) will be appearing in our figures within a frame. Let $A$ be a block of length $l\ge1$ and consider the \sq\ $\eta\in\{0,1\}^{\N\cup\{0\}}$ starting at the coordinate $0$ with $A$ followed by an infinite regular pattern, e.g., $\eta=\boxed{\ \ A\ \ }10101010101\dots$. Let $n\ge l+4$ be even and let $\zeta$ be the \sq\ $\eta$ shifted to the right so that it starts at the coordinate $n$. The binary summation $\eta\pl\zeta$ (with the carry) is shown on Figure~\ref{f2}.
\begin{figure}[h]
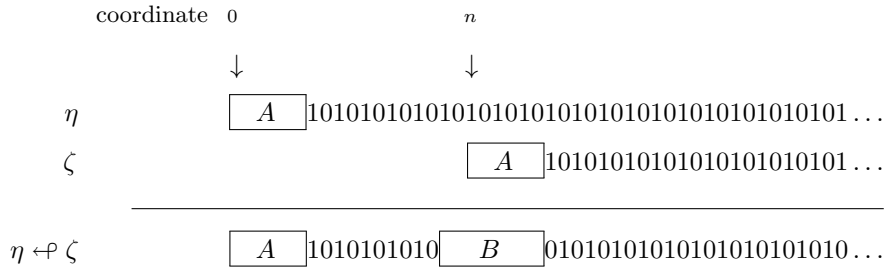

\begin{align*}
^\text{\small{coordinate}}\ \ ^0\phantom{00000000000000000}^n\phantom{^{+1}.0000000000000000000000000000}&\\
\downarrow\!\phantom{0000000000000000}\downarrow\phantom{..00000000000000000000000000000}&\\
\eta\hspace{2cm}\boxed{\ \ A\ \ \,}10101010101010101010101010101010101010101\dots&\\
\zeta\hspace{5.19cm}\boxed{\ \ A\ \ \,}10101010101010101010101\dots&\\
\rule{10cm}{0.4pt}&\\
\eta\pl\zeta\hspace{2cm}\boxed{\ \ A\ \ \,}1010101010\boxed{\, \ \ \ B\ \ \ \ }01010101010101010101010\dots&
\end{align*}
\caption{Summation with potentially non-regular blocks.}
\label{f2}
\end{figure}

\noindent In the sum $\eta\pl\zeta$ we have two potentially non-regular blocks: $A$ of length $l$ starting at the coordinate $0$ and ending at $l-1$, and $B$ of length $l+2$ starting at the coordinate $n-2$ and ending at $n+l-1$. The regular pattern between $A$ and $B$ has length $n-l-2\ge2$. To the right of $B$ there occurs an infinite regular pattern (mirrored with respect to those in $\eta$ and $\zeta$). 
\smallskip

We continue with the proof. The binary expansion of $\frac43$ is $1.010101\dots$, hence the \sq\ obtained by ignoring the binary dot is the infinite regular pattern $1010101\dots$ starting at the coordinate $0$, with $1$'s at the even positions. Since $y=\sum_{i\ge0}2^{-s_i}$, we have $z=\sum_{i\ge0}\tfrac432^{-s_i}$, that is, the \sq\ representing the binary expansion of $z$ (with the binary dot ignored) can be obtained by summing (with the ``carry'') countably many copies of $10101010\dots$ shifted by $s_0, s_1, s_2,$ etc.\ positions to the right. 
\begin{figure}[H]
\begin{align*}
^{0}\ \ \, ^{n_1}\phantom{010..}^{n_2}\phantom{.....0101010101010101010101010101010101000}^{n_3}\phantom{.10110101010000000\dots}&\\
\downarrow\ \ \,\downarrow\phantom{0000}\downarrow\ \phantom{\downarrow1010101010101010101010101010101010100.}\downarrow\ \phantom{\downarrow.0101\downarrow\ \downarrow}\phantom{0101010.\dots}&\\
1.01010101010101010101010101010101010101010101010101010101010101010101\dots&\\
1010101010101010101010101010101010101010101010101010101010101010101\dots&\\
1010101010101010101010101010101010101010101010101010101010101\dots&\\
10101010101010101010101010101010101010101010101010101010101\dots&\\
1010101010101010101\dots&\\
10101010101010101\dots&\\
10101010101\dots&\\
101010101\dots&
\end{align*}
\caption{The summation representing $\frac43y$.}
\label{f3}
\end{figure}
\noindent Figure~\ref{f3} shows the \sq s to be summed up in order to obtain the binary expansion of $z$. The coordinate $n_3$ is intentionally shown much smaller than it is in reality just to make it fit on the page. The sum of the first two rows is $1.101010101\dots$ with the first symbol $1$ being an irregular block of lenght $1$, so we will write $\boxed{\!\!1\!\!}.101010101\dots$. By adding the rows pairwise, the summation on Figure~\ref{f3} reduces to:
\begin{figure}[H]
\begin{align*}
^{0}\ \,\ \phantom{^{n_1}}\phantom{010..}^{n_2}\phantom{.....0101010101010101010101010101010101000}^{n_3}\phantom{.101101010100000a\dots}&\\
\downarrow\ \,\ \phantom{\downarrow}\phantom{0101}\downarrow\ \phantom{\downarrow1010101010101010101010101010101010100.}\downarrow\ \phantom{\downarrow.0101\downarrow\ \downarrow}\phantom{01010a.\dots}&\\
\boxed{\!\!1\!\!}.1010101010101010101010101010101010101010101010101010101010101010101\dots&\\
\boxed{\!\!1\!\!}10101010101010101010101010101010101010101010101010101010101\dots&\\
\boxed{\!\!1\!\!}10101010101010101\dots&\\
\boxed{\!\!1\!\!}101010101\dots&\\
\end{align*}
\caption{}
\label{f01}
\end{figure}
\noindent Now, in the summation of the first two rows we can refer to our Observation with the parameters $l=1$ and $n=n_2$. According to our Observation, we can predict that the sum of these rows should have two potentially non-regular blocks of lengths $l=1$ and $l+2=3$ (which we can write as $n_1+1$). The regular pattern between these blocks should have length $n-l-2=n_2-1-2=5$ (which we can write as $n_2-n_1-1$). The last potentially non-regular block should end at the position \mbox{$n+l-1=n_2=8$}. Indeed, the sum of these rows equals $\boxed{\!\!1\!\!}.10101\boxed{\!\!100\!\!}010101010101010\dots$, which complies with the predictions based on the Observation. 
Note that the regular pattern between the non-regular blocks does not change when the remaining rows are added. The summation on Figure~\ref{f3} now reduces to:
\begin{figure}[H]
\begin{align*}
^{0}\ \ \phantom{^{n_1}}\phantom{010..0.i}\phantom{.....0101010101010101010101010101010101000}^{n_3}\phantom{.101101010100000\dots}&\\
\downarrow\ \ \phantom{\downarrow}\phantom{0101\ i}\phantom{\downarrow}\ \phantom{\downarrow1010101010101010101010101010101010100.}\downarrow\ \phantom{\downarrow.0101\downarrow\ \downarrow}\phantom{01010.\dots}&\\
\boxed{\!\!1\!\!}.10101\boxed{\!\!100\!\!}0101010101010101010101010101010101010101010101010101010101\dots&\\
\boxed{\!\!1\!\!}10101\boxed{\!\!100\!\!}01010101\dots&
\end{align*}
\caption{}
\label{f02}
\end{figure}
\noindent We will treat the non-regular blocks in the second row as one block of length $n_2+1=9$. According to the Observation with the new parameters $l=n_2+1$ and $n=n_3$, we can predict that the sum of these rows should have a third non-regular block of length $n_2+1+2=n_2+n_1+1=11$ preceded by a regular pattern of length $n_3-n_2-1-2=n_3-n_2-n_1-1$.  The last non-regular block should end at the coordinate $n_3+n_2+1-1=n_3+n_2$. Indeed, the sum equals
\begin{figure}[H]
$$
\boxed{\!\!1\!\!}.10101\boxed{\!\!100\!\!}010101010101010101010101010101010101010\boxed{\!\!11100000001\!\!}10101010\dots,
$$
\caption{}
\label{f03}
\end{figure}
\noindent which complies with the predictions based on the Observation. Again, the regular patterns between the non-regular blocks do not change when the remaining rows are added. 
\smallskip

Using inductively the Observation, we can see that the infinite sum (representing the binary expansion of $z=\frac43y$) has non-regular blocks of lengths $n_k+n_{k-1}+\cdots+n_2+n_1+1$ preceded by regular patterns of lengths $n_{k+1}-n_k-n_{k-1}-\cdots-n_2-n_1-1$. Since the numbers $n_k$ are defined by $n_1=2$, $n_{k+1}=n_k2^{n_k}$, we have 
$$
\lim_{k\to\infty}\frac1{n_{k+1}}{(n_k+n_{k-1}+\dots+n_2+n_1+1)}=0,
$$ 
and hence the non-regular blocks occupy a set of density~$0$. The regular patterns in the binary expansion of $z$ have increasing lengths and occupy a subset of density $1$ and thus the block $01$ appears in the expansion of $z$ with frequency $\frac12$, as claimed.
\end{proof}

\subsection{Normal number times a deterministic number can be deterministic}\label{sxyz}
Let $x\in[0,1)$ be the number whose binary expansion is the \sq\ $\kappa$ (see Definition~\ref{kapa}) enumerated from $1$ to $\infty$ (i.e., with the binary dot falling to the left of the first digit 0).

The next theorem shows that the normal number $x$ and the deterministic numbers $y$ and $z$ constructed in Subsection~\ref{2det} satisfy the assertion of Theorem~\ref{xyd}.  
 
\begin{thm}\label{cex}Then numbers $xy$ and $xz$ are deterministic.  
\end{thm}

\begin{proof}
It suffices to show that $xy$ is deterministic (then $xz = \frac43xy$ is also deterministic by Corollary~\ref{affin}).
\smallskip

Before embarking on the proof, we will make a few additional observations concerning the binary addition $\pl$ with the carry. By a \emph{switch} we will mean a block of the form $01$ or $10$.
\smallskip

\noindent{\bf Observation 1.} Let $\eta,\zeta\in\{0,1\}^\N$ be binary \sq s. Suppose that for some interval $[a,b]\subset\N$ the blocks
$B=\eta|_{[a,b]}$ and $C=\zeta|_{[a,b]}$ are ``almost mirrors'' of each other, i.e., $C$ differs from the mirror $\tilde B$ of $B$ at a single coordinate $a\le l\le b$. Then the block $D=(\eta\pl\zeta)|_{[a,b]}$ has at most two switches. We skip an elementary verification. This is illustrated by Figure~\ref{f4}: 
\begin{figure}[h]
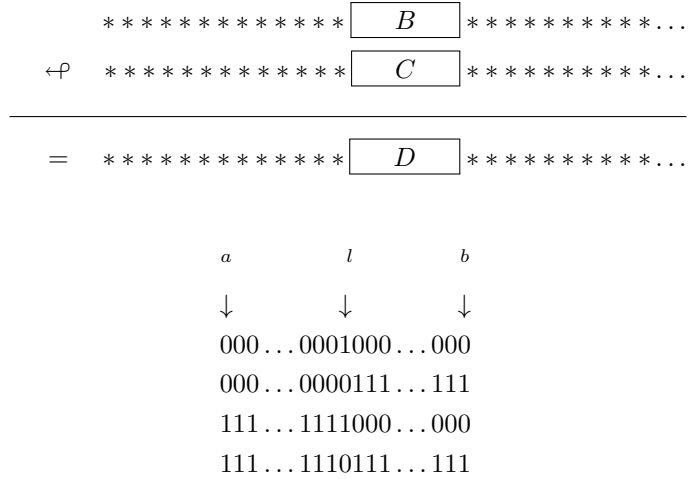

\begin{align*}
\phantom{+\ \ \ }*************\boxed{\ \ \ \ B \ \ \ \ }**********\dots&\\
\pl\ \ \ *************\,\boxed{\ \ \ \ C \ \ \ \ }**********\dots&\\
\rule{9cm}{0.4pt}&\\
=\ \ \ *************\,\boxed{\ \ \ \ D \ \ \ \ }**********\dots&
\end{align*}

\bigskip
\begin{align*}
\phantom{^\text{coordinate relative to $D$}} \ \ \ \ ^a\phantom{.00\dots000}^l\phantom{000\dots00}^b&\phantom{00000000000000000000}\\
\downarrow\phantom{\,.0\dots00}\downarrow\phantom{\,.00\dots0}\downarrow&\phantom{00000000000000000000}
\\
000\dots0001000\dots000&\phantom{00000000000000000000}\\
000\dots0000111\dots111&\phantom{00000000000000000000}\\
111\dots1111000\dots000&\phantom{00000000000000000000}\\
111\dots1110111\dots111&\phantom{00000000000000000000}
\end{align*}
\caption{Top diagram: Addition $\pl$ (with the carry) of ``almost mirrored'' blocks. The stars represent unspecified symbols. Bottom diagram: The block $D=(\eta\pl\zeta)|_{[a,b]}$ has one of the four presented forms, each with at most two switches, for example the first block has two switches, one at the coordinates $(l-1,l)$ and another, at the coordinates $(l,l+1)$.}
\label{f4}
\end{figure}

\smallskip\noindent{\bf Observation 2.} Let $\eta,\zeta\in\{0,1\}^\N$ be binary \sq s and let $[a,b]\subset\N$.
Suppose that each of the blocks $\eta|_{[a,b]}$ and $\zeta|_{[a,b]}$ admits at most $m\ge1$ switches. Then in $(\eta\pl\zeta)|_{[a,b]}$ there may occur at most $4m+1$ switches: every switch in $\eta|_{[a,b]}$ or $\zeta|_{[a,b]}$ may produce at most two switches in $(\eta\pl\zeta)|_{[a,b]}$, and an additional switch may occur in $(\eta\pl\zeta)|_{[a,b]}$ at the terminal coordinates $(b-1,b)$ due to the (unknown) symbols appearing in $\eta$ and $\zeta$ to the right of $b$. An example of this phenomenon is demonstrated by Figure~\ref{f5}.

\begin{figure}[h]
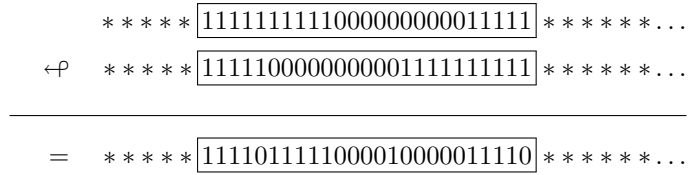

\begin{align*}
\phantom{+\ \ \ }*****\boxed{\!1111111111000000000011111\!}******\dots&\\
\pl\ \ \ *****\,\boxed{\!1111100000000001111111111\!}******\dots&\\
\rule{9cm}{0.4pt}&\\
=\ \ \ *****\,\boxed{\!1111011111000010000011110\!}******\dots&
\end{align*}
\caption{Each of the blocks in top two rows has $m=2$ switches, the bottom block has $7\le4m+1$ switches.}
\label{f5}
\end{figure}

\smallskip\noindent{\bf Observation 3.} Suppose that we perform the addition $\pl$ of $2^k$ binary \sq s and we know that within some interval $[a,b]$ each of these \sq s has at most two switches. Then the number of switches within $[a,b]$ in the sum is at most $3\cdot 4^k$. 
This is best seen by applying Observation 2 inductively on $k$. The iterations of the function $n\mapsto 4n+1$ starting with $n_0=2$ grow slower than $3\cdot4^k$ (where $k$ is the number of iterates).

\begin{rem}The bound $3\cdot4^k$ is largely overestimated. It does not take into account that eventually many of the switches will overlap and cancel out. In fact, the number of switches grows linearly with $k$. But proving a tighter estimate requires tedious work while the crude estimate $3\cdot4^k$ is perfectly sufficient for us.
\end{rem}

\smallskip
We continue with the proof. We have $xy = \sum_{i\ge0}x2^{-s_i}$, hence the \sq\  representing the binary expansion of $xy$ is obtained by summing (using $\pl$) countably many copies of the \sq\ $\kappa$ shifted by $s_0, s_1, s_2,$ etc.\ positions to the right. This is illustrated by Figure~\ref{f6}. 

\begin{figure}[h]
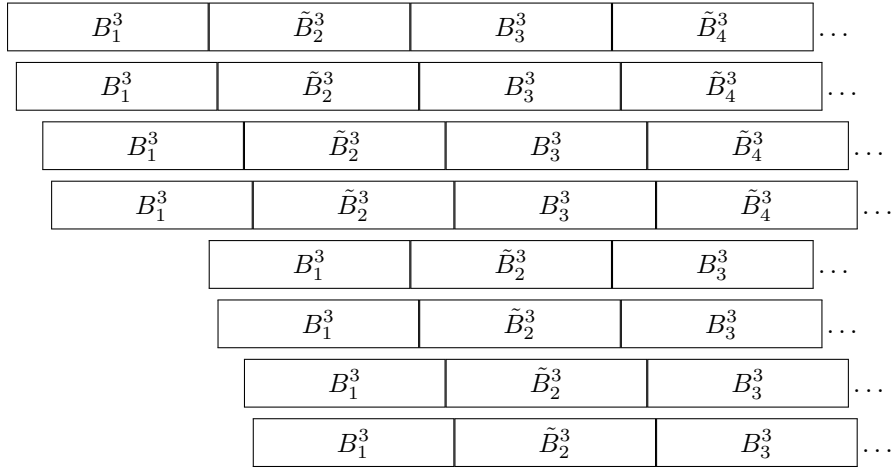

\begin{align*}
\boxed{\hspace{1cm}\vphantom{\tilde B}B^3_1\hspace{1cm}}\boxed{\hspace{1cm}\tilde B_2^3\hspace{1cm}}\boxed{\hspace{1cm}\vphantom{\tilde B}B_3^3\hspace{1cm}}\boxed{\hspace{1cm}\tilde B_4^3\hspace{1cm}}\dots\ \ \ \ \ &\\
\boxed{\hspace{1cm}\vphantom{\tilde B}B^3_1\hspace{1cm}}\boxed{\hspace{1cm}\tilde B^3_2\hspace{1cm}}\boxed{\hspace{1cm}\vphantom{\tilde B}B^3_3\hspace{1cm}}\boxed{\hspace{1cm}\tilde B^3_4\hspace{1cm}}\dots\ \ \ \ &\\
\boxed{\hspace{1cm}\vphantom{\tilde B}B^3_1\hspace{1cm}}\boxed{\hspace{1cm}\tilde B^3_2\hspace{1cm}}\boxed{\hspace{1cm}\vphantom{\tilde B}B^3_3\hspace{1cm}}\boxed{\hspace{1cm}\tilde B^3_4\hspace{1cm}}\dots\ &\\
\boxed{\hspace{1cm}\vphantom{\tilde B}B^3_1\hspace{1cm}}\boxed{\hspace{1cm}\tilde B^3_2\hspace{1cm}}\boxed{\hspace{1cm}\vphantom{\tilde B}B^3_3\hspace{1cm}}\boxed{\hspace{1cm}\tilde B^3_4\hspace{1cm}}\dots&\\
\boxed{\hspace{1cm}\vphantom{\tilde B}B^3_1\hspace{1cm}}\boxed{\hspace{1cm}\tilde B^3_2\hspace{1cm}}\boxed{\hspace{1cm}\vphantom{\tilde B}B^3_3\hspace{1cm}}\dots\ \ \ \ \ &\\
\boxed{\hspace{1cm}\vphantom{\tilde B}B^3_1\hspace{1cm}}\boxed{\hspace{1cm}\tilde B^3_2\hspace{1cm}}\boxed{\hspace{1cm}\vphantom{\tilde B}B^3_3\hspace{1cm}}\dots\ \ \ \ &\\
\boxed{\hspace{1cm}\vphantom{\tilde B}B^3_1\hspace{1cm}}\boxed{\hspace{1cm}\tilde B^3_2\hspace{1cm}}\boxed{\hspace{1cm}\vphantom{\tilde B}B^3_3\hspace{1cm}}\dots\ &\\
\boxed{\hspace{1cm}\vphantom{\tilde B}B^3_1\hspace{1cm}}\boxed{\hspace{1cm}\tilde B^3_2\hspace{1cm}}\boxed{\hspace{1cm}\vphantom{\tilde B}B^3_3\hspace{1cm}}\dots&
\end{align*}
\caption{The summation producing $xy$ (in the binary expansion). The figure is similar to Figure~\ref{f3}, except that instead of shifting the regular pattern representing $\frac43$ we are shifting the \sq\ $\kappa$  shown in Figure~\ref{f1}. Also, we draw the figure in a much smaller horizontal scale. }
\label{f6}
\end{figure}

This time we are not using induction on $k$; the argument works independently for each $k$. We will present it (and draw our figures) for $k=3$. The blocks labeled on Figure~\ref{f6} by $B^3_1,\ B^3_2,\ B^3_3$
have length $n_3=8\cdot 2^8=2048$ (as before, the proportions on the figure are not to scale). The figure is truncated after $\tilde B^3_4$ but the pattern runs till $\tilde B^3_{2^{2048}}$. 
Note that the row~$5$ (counting from the top) is the result of shifting the row~$1$ by exactly~$n_3$ positions to the right. The same applies to the rows~$6$ and~$2$, then~$3$ and~$7$, etc. So, let us rearrange the rows in the following order $1, 5, 2, 6, 3, 7, 4, 8$, as shown in Figure~\ref{fg7}:

\begin{figure}[h]
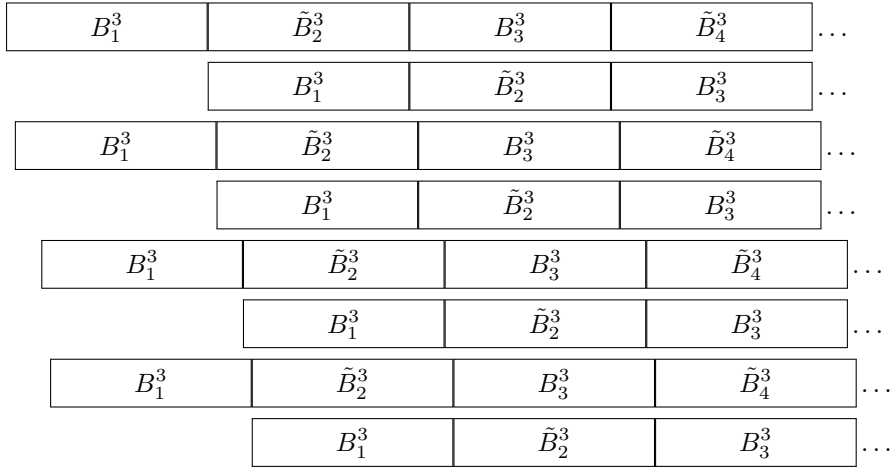

\begin{align*}
\boxed{\hspace{1cm}\vphantom{\tilde B}B^3_1\hspace{1cm}}\boxed{\hspace{1cm}\tilde B^3_2\hspace{1cm}}\boxed{\hspace{1cm}\vphantom{\tilde B}B^3_3\hspace{1cm}}\boxed{\hspace{1cm}\tilde B^3_4\hspace{1cm}}\dots\ \ \ \ \ &\\
\boxed{\hspace{1cm}\vphantom{\tilde B}B^3_1\hspace{1cm}}\boxed{\hspace{1cm}\tilde B^3_2\hspace{1cm}}\boxed{\hspace{1cm}\vphantom{\tilde B}B^3_3\hspace{1cm}}\dots\ \ \ \ \ &\\
\boxed{\hspace{1cm}\vphantom{\tilde B}B^3_1\hspace{1cm}}\boxed{\hspace{1cm}\tilde B^3_2\hspace{1cm}}\boxed{\hspace{1cm}\vphantom{\tilde B}B^3_3\hspace{1cm}}\boxed{\hspace{1cm}\tilde B^3_4\hspace{1cm}}\dots\ \ \ \ &\\
\boxed{\hspace{1cm}\vphantom{\tilde B}B^3_1\hspace{1cm}}\boxed{\hspace{1cm}\tilde B^3_2\hspace{1cm}}\boxed{\hspace{1cm}\vphantom{\tilde B}B^3_3\hspace{1cm}}\dots\ \ \ \ &\\
\boxed{\hspace{1cm}\vphantom{\tilde B}B^3_1\hspace{1cm}}\boxed{\hspace{1cm}\tilde B^3_2\hspace{1cm}}\boxed{\hspace{1cm}\vphantom{\tilde B}B^3_3\hspace{1cm}}\boxed{\hspace{1cm}\tilde B^3_4\hspace{1cm}}\dots\ &\\
\boxed{\hspace{1cm}\vphantom{\tilde B}B^3_1\hspace{1cm}}\boxed{\hspace{1cm}\tilde B^3_2\hspace{1cm}}\boxed{\hspace{1cm}\vphantom{\tilde B}B^3_3\hspace{1cm}}\dots\ &\\
\boxed{\hspace{1cm}\vphantom{\tilde B}B^3_1\hspace{1cm}}\boxed{\hspace{1cm}\tilde B^3_2\hspace{1cm}}\boxed{\hspace{1cm}\vphantom{\tilde B}B^3_3\hspace{1cm}}\boxed{\hspace{1cm}\tilde B^3_4\hspace{1cm}}\dots&\\
\boxed{\hspace{1cm}\vphantom{\tilde B}B^3_1\hspace{1cm}}\boxed{\hspace{1cm}\tilde B^3_2\hspace{1cm}}\boxed{\hspace{1cm}\vphantom{\tilde B}B^3_3\hspace{1cm}}\dots&
\end{align*}
\caption{The summation producing $xy$ (in the binary expansion) after rearranging the order.}
\label{fg7}
\end{figure}

Now, let us add the rows pairwise. The result is shown in Figure~\ref{fg8}.
\begin{figure}[h]
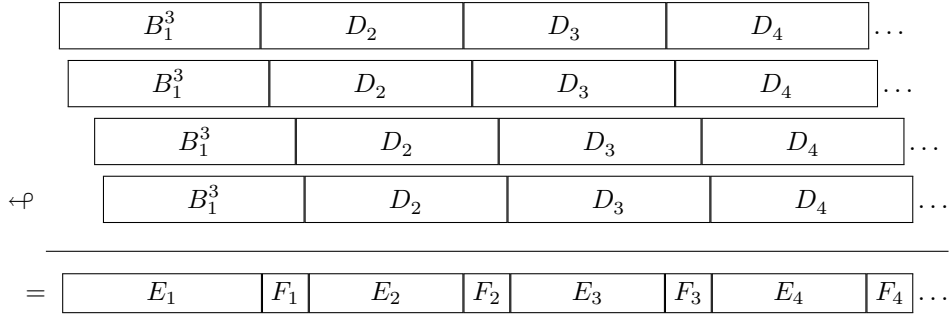

\begin{align*}
\boxed{\hspace{1cm}B^3_1\hspace{1cm}}\boxed{\hspace{1cm}D^{\phantom 3}_2\hspace{1cm}}\boxed{\hspace{1cm}D^{\phantom 3}_3\hspace{1cm}}\boxed{\hspace{1cm}D^{\phantom 3}_4\hspace{1cm}}\dots\ \ \ \ \ &\\
\boxed{\hspace{1cm}B^3_1\hspace{1cm}}\boxed{\hspace{1cm}D^{\phantom 3}_2\hspace{1cm}}\boxed{\hspace{1cm}D^{\phantom 3}_3\hspace{1cm}}\boxed{\hspace{1cm}D^{\phantom 3}_4\hspace{1cm}}\dots\ \ \ \ &\\
\boxed{\hspace{1cm}B^3_1\hspace{1cm}}\boxed{\hspace{1cm}D^{\phantom 3}_2\hspace{1cm}}\boxed{\hspace{1cm}D^{\phantom 3}_3\hspace{1cm}}\boxed{\hspace{1cm}D^{\phantom 3}_4\hspace{1cm}}\dots\ &\\
\pl\ \ \ \ \ \ \ \boxed{\hspace{1cm}B^3_1\hspace{1cm}}\boxed{\hspace{1cm}D^{\phantom 3}_2\hspace{1cm}}\boxed{\hspace{1cm}D^{\phantom 3}_3\hspace{1cm}}\boxed{\hspace{1cm}D^{\phantom 3}_4\hspace{1cm}}\dots&\\
\rule{12cm}{0.4pt}&\\
=\ \boxed{\hspace{1cm}E_1\hspace{1cm}}\boxed{F_1}\boxed{\hspace{0.7cm}E_2\hspace{0.7cm}}\boxed{F_2}\boxed{\hspace{0.7cm}E_3\hspace{0.7cm}}\boxed{F_3}\boxed{\hspace{0.7cm}E_4\hspace{0.7cm}}\boxed{F_4}\dots&
\end{align*}
\caption{The summation producing $xy$ (in the binary expansion) after rearranging the order and summing pairwise.}
\label{fg8}
\end{figure}

Since for each $i=1,2,\dots,2^{2048}$ the blocks $B^3_i$ and $B^3_{i+1}$ differ at one place, by Observation~2, each of the blocks $D_i$ has at most two switches. The blocks $F_1,F_2,\dots,F_{2^{2048}}$ have lenghs $n_1+n_2=10$ while the blocks $E_2,E_3,\dots,E_{2^{2048}}$ have length $n_3-n_2-n_1=2038$ (we recommend consulting also Figure~\ref{f3}). By Observation~3, each of the latter blocks admits at most $3\cdot4^2=48$ switches. Jointly, in each concatenation of the form $F_iE_{i+1}$ we have at most $n_1+n_2+3\cdot4^2$ switches. In general, if we divide the initial block of length $n_{k+1}=n_k2^{n_k}$ of the \sq\ associated with $xy$ into
blocks of length $n_k$ then in all but the first one of them there will be at most $n_1+n_2+\cdots+n_{k-1}+3\cdot4^{k-1}$ switches. Since the ratio $\frac1{n_k}(n_1+n_2+\cdots+n_{k-1}+3\cdot4^{k-1})$ tends to zero, we conclude that the frequency of switches (i.e., of the blocks $01$ and $10$) in the binary expansion of $xy$ is zero.

Now observe that the frequency zero of switches in the binary expansion $\omega$ of $xy$ implies that $\omega$ is a deterministic \sq. Indeed, consider the endomorphism $\pi:\{0,1\}^\N\to \{0,1\}^\N$ given by $(a_n)_{n\in\N}\mapsto (a_n+a_{n+1}\!\!\mod2)_{n\in\N}$ (here we apply the coordinatewise addition mod 2, without the carry). Note that the image $\pi(\omega)$ has the symbol $1$ at a  coordinate $n$ if and only if $\omega$ has a switch at the coordinates $(n,n+1)$. This implies that $\pi(\omega)$ has density zero of symbols $1$, and thus it is trivially deterministic. The map $\phi$ is $2$-$1$, so it preserves entropy (see, e.g.,~\cite[Theorem~4.1.15]{Do}), and hence it sends non-deterministic \sq s to non-deterministic \sq s. Thus, the \sq\ $\omega$ (and hence the number $xy$) is deterministic.
\end{proof}

\subsection{Products of deterministic numbers need not be deterministic} \label{xyn}
In this subsection we show that the product of two deterministic numbers need not be deterministic (it can even be normal). We also show that the square of a deterministic number need not be deterministic. These facts are consequences of the following claim, whose proof will be given after the derivation of the immediate corollaries:

\begin{prop}\label{1/y}
Let $y$ be the deterministic number constructed in Subsection~\ref{2det}. Then $\frac1y$ is deterministic.
\end{prop}

\begin{cor}\label{xynn} The product of two deterministic numbers need not be deterministic (it can be normal).
\end{cor}

\begin{proof} 
Let $a=xy$ and $b=\frac1y$, where $x$ and $y$ are as in Theorem~\ref{cex} and Proposition~\ref{1/y}. Both $a$ and $b$ are deterministic while $ab=x$, which is normal.
\end{proof}

\begin{cor}\label{x^2}
The square of a deterministic number need not be deterministic. 
\end{cor}
\begin{proof}
Let $a$ and $b$ be the deterministic numbers as in Corollary~\ref{xynn}. Let $s = \frac{a+b}2$ and $t=\frac{a-b}2$. By Theorems~\ref{x+y} and~\ref{xq}, both $s$ and $t$ are deterministic.
Then 
$$
s^2-t^2=(s+t)(s-t)=ab=x,
$$
which is normal. Thus, by Theorem~\ref{x+y} again, at least one of the squares $s^2$, $t^2$ is not deterministic.
\end{proof}

\begin{proof}[Proof of Proposition~\ref{1/y}] In what follows $(n_k)_{k\ge1}$ is the \sq\ introduced in the process of constructing the number $y$ (recall that the binary expansion of $y$ matches the indicator function of the set $S=\{0\}\cup\mathsf{FS}((n_k)_{k\ge1})$\,). 
We define inductively binary blocks $B_k$ as follows: 
\begin{align*}
&B_1=11 \text{ (note that the length of $B_1$ is $2=n_1$), and then}\\
&B_{k+1}=(B_k 0^{n_k})^{\frac{n_{k+1}}{2n_k}}, \ k\ge1,
\end{align*}
where each exponent should be interpreted as the number of repetitions. In words, $B_{k+1}$ consists of $\frac{n_{k+1}}{2n_k}$ (recall that $n_{k+1}=n_k2^{n_k}$, hence $2n_k$ divides $n_{k+1}$) repetitions of $B_k000\dots0$, where $B_k$ (of length $n_k$) is followed by $n_k$ zeros. The length of $B_{k+1}$ is $n_{k+1}$. For example,
\begin{align*}
B_2 &= 11001100,\\
B_3 &= 11001100000000001100110000000000110011\dots1100110000000000,
\end{align*}
where in $B_3$ the block $B_2$ is followed by eight zeros and the block $B_200000000$ is repeated 128 times. Then the length of $B_3$ is exactly equal to $2048=n_3$. The coordinates in the blocks $B_k$ are counted from $1$ to $n_k$. We let $\omega$ be the infinite one-sided \sq\ (starting at coordinate 1), obtained as the limit of the blocks $B_k$, and we define $v$ as the number whose binary expansion matches $\omega$ with the binary dot on the left of coordinate 1 (so that $v<1$). Observe that the digit $1$ occurs in $\omega$ with frequency zero. This follows from the fact that the fraction of $1$'s in $B_{k+1}$ is half the fraction in $B_k$. So $v$ is trivially deterministic.

In order to show that $v=\frac1y$, let us compute the product $vy$. This is done in an already familiar manner, by adding (using $\pl$) copies of $\omega$ shifted by the elements of $S=\{0\}\cup\mathsf{FS}((n_k)_{k\ge1})=\{0,n_1,n_2,n_2+n_1,n_3,\dots\}$. We obtain the following diagram (cf.\ Figure~\ref{f3}):
\begin{figure}[H]
\begin{align*}
^{0}\,\ ^{n_1}\phantom{010..}^{n_2}\phantom{.....100}^{2n_2}\phantom{.10001100100101010101010101010000}^{n_3}\phantom{.1011010101000000\dots}\phantom{0}&\\
\downarrow\,\ \downarrow\phantom{0101}\downarrow\ \phantom{\downarrow100.}\downarrow\phantom{00101010101010101010101010101000j}\downarrow\ \phantom{\downarrow.0101\downarrow\ \downarrow}\phantom{010101.\dots}\phantom{0}&\\
.11001100000000001100110000000000\dots1100110000000000000000000000000000\dots&\\
11001100000000001100110000000000\dots11001100000000000000000000000000\dots&\\
11001100000000001100110000000000\dots11001100000000000000000000\dots&\\
11001100000000001100110000000000\dots110011000000000000000000\dots&\\
110011000000000011\dots&\\
1100110000000000\dots&\\
1100110000\dots&\\
\pl\phantom{0000000000000000000000000000000000000000000000000000000000}11001100\dots&\\
\rule{12.5cm}{0.4pt}&\\
.11111111111111111111111111111111\dots1111111111111111111111111111111111\dots&
\end{align*}
\caption{The summation producing $vy$ (in the binary expansion).}
\label{fh4}
\end{figure}
In each column of the diagram there appears exactly one digit~$1$. Indeed,
this fact can be visually checked in the initial $n_2$ columns. In columns $n_2+1,\dots,2n_2$, the top two rows become zeros, while rows 3 and 4 duplicate the pattern of the top two rows in columns $1,\dots,n_2$. Hence the ``one digit 1'' rule applies to the initial $2n_2$ columns. Next, the pattern in the top four rows in columns $1,\dots,2n_2$ is repeated periodically until coordinate $n_3$, hence the ``one digit 1'' rule extends to the initial $n_3$ columns. Inductively, once the ``one digit 1'' rule is verified for the initial $n_k$ columns, in columns $n_k+1,\dots,2n_k$ the top $2^{k-1}$ rows become zeros, while the next $2^{k-1}$ rows duplicate the pattern of the top $2^{k-1}$ rows in columns $1,\dots,n_k$, so the rule applies to the initial $2n_k$ columns. Then, repetitions in the top $2^k$ rows extend the rule to the initial $n_{k+1}$ columns. Eventually, the binary expansion of $vy$ is the \sq\ of just $1$'s (in this particular case the carry never occurs, so $\pl$ is the same as $+$), i.e., $vy=1$, as needed (this is the unique case in this paper when we use the alternative binary expansion of a rational number, ending with $1$'s).
\end{proof}

\subsection{Some natural open problems} The goal of this section is to present some natural open problems motivated by the results of the previous subsections and by the following simple observation, which extends Proposition~\ref{any} (we work with a fixed base $r\ge2$ but for brevity in what follows we skip mentioning the base):
\begin{prop}\label{seven}
Any real number $z\neq 0$ can be represented as the sum, difference, product, ratio, and product of reciprocals of two normal numbers, as well as the sum or difference of a normal number and the reciprocal of a normal number.
\end{prop}
\begin{proof}
The map $x\mapsto x^{-1}$ is invertible on $\R\setminus\{0\}$ and non-singular (preserves the class of sets of Lebesgue measure zero). Since the set ${\mathcal N}$ of normal numbers has full Lebesgue measure, the set of reciprocals of normal numbers (henceforth denoted by ${\mathcal N}^{-1}$) also has full Lebesgue measure. In addition, each of the sets:
$z-{\mathcal N}$, $z+{\mathcal N}$, $z\cdot{\mathcal N}^{-1}$, $z\cdot{\mathcal N}$, $\frac1z\cdot{\mathcal N}^{-1}$, and $\frac1z\cdot{\mathcal N}$ has full Lebesgue measure. The same applies to the sets $z-{\mathcal N}^{-1}$, $z+{\mathcal N}^{-1}$.  Let
\begin{multline*}
{\mathcal N}_z = \\
{\mathcal N}\cap (z-{\mathcal N})\cap (z+{\mathcal N})\cap(z\cdot{\mathcal N}^{-1})\cap(z\cdot{\mathcal N})\cap(\tfrac1z\cdot{\mathcal N}^{-1})\cap(\tfrac1z\cdot{\mathcal N})\cap(z-{\mathcal N}^{-1})\cap(z+{\mathcal N}^{-1}).
\end{multline*}
Clearly, the set $\mathcal N_z$ has full measure. Let $x\in{\mathcal N}_z$. Then $x$ is normal and there are normal numbers $x_1,x_2,\dots,x_8$ such that 
$$
x=z-x_1=z+x_2=\tfrac z{x_3}=zx_4=\tfrac1{zx_5}=\tfrac{x_6}z=z-\tfrac1{x_7}=z+\tfrac1{x_8},
$$
implying that
$$
z=x+x_1=x-x_2=x{x_3}=\tfrac x{x_4}=\tfrac1{xx_5}=\tfrac{x_6}x=x+\tfrac1{x_7}=x-\tfrac1{x_8}.
$$
\end{proof}

\begin{rem}
Similarly, it can be shown that any nonzero real number $z$ can be represented as the sum, difference, product, ratio, and product of the reciprocals of two non-normal numbers, as well as the sum or difference of a non-normal number and the reciprocal of a non-normal number. The proof uses the fact that the set of non-normal numbers is residual (i.e., the set of normal numbers is of first Baire category, see for example \cite[footnote 13]{OU} and \cite[Proposition 4.7]{BDM}) and that the map $x\mapsto x^{-1}$ preserves the class of residual sets. 
\end{rem}

Here is finally a list of some open questions.
\begin{enumerate}
    \item Is the reciprocal of a normal number always normal? 
    \item Is the reciprocal of a nonzero deterministic number always deterministic? 
    \item Does there exist a normal number whose reciprocal is deterministic?
    \item Can any nonzero real number be represented as (i) the 
    product, (ii) the ratio, or (iii) the product of reciprocals, of two deterministic numbers? 
    \item Can any nonzero real number be represented as (i) the product, (ii) the ratio, or (iii) the product of reciprocals, of a normal and a deterministic number?
    \item Are there irrational numbers $a$ with the property that $ax$ is normal for all normal $x$?
    \item Are there any irrational numbers $b$ with the property that $by$ is deterministic for every deterministic $y$?
\end{enumerate}

\section{Appendix}
In this appendix, we will sketch the proof of Theorem~\ref{detdet}. For the reader's convenience, we repeat here the formulation of this theorem. 
\smallskip

\noindent{\bf Theorem 3.9.}
\emph{A \sq\ $\omega\in\{0,1,\dots,r-1\}^\N$ is deterministic if and only if it has subexponential epsilon-complexity.}
\medskip

The proof utilizes the notion of combinatorial entropy of a block. Recall (see~\eqref{empirical}, Section~\ref{S2}) that any block $B$ of length $m$, over a finite alphabet $\Lambda$, and determines a density function $\mu_B$ on blocks $C$ of length $n\le m$ by the formula:
\begin{equation}\label{czep}
\mu_B(C)=\frac1{m-n+1}|\{i\in[1,m-n+1]:B|_{[i,i+n-1]}\approx C\}|.
\end{equation}
Note that $\{\mu_B(C):C\in\Lambda^n\}$ is a probability vector. 

\begin{defn}{\rm(cf. \cite[Section 2.8]{Do})}
Fix some $n\in\N$ and let $B$ be a block of length $m\ge n$, over a finite alphabet $\Lambda$. The \emph{$n$th combinatorial entropy of $B$} is defined as 
\begin{equation}\label{czep1}
H_n(B) = -\frac1n\sum_{C\in\Lambda^n}\mu_B(C)\log(\mu_B(C)).
\end{equation}
\end{defn}

In the proof of Theorem 3.7, we will need the following fact, see \cite[Lemma~1]{BGH} or \cite[Lemma~2.8.2]{Do} (we use the notation from \cite{Do}):
\begin{thm}\label{cl}
For $c>0$ let $\mathbf C[n,m,c]$ denote the number of blocks of length $m$, over $\Lambda$, such that  $H_n(B)\le c$.
Then 
$$
\limsup_{m\to\infty}\frac{\log_2(\mathbf C[n,m,c])}m\le c.
$$
\end{thm}

\begin{proof}[Proof of Theorem~\ref{detdet}] Let $\Lambda=\{0,1,\dots,r-1\}$ and suppose that $\omega\in\Lambda^\N$ is deterministic. Let $\M_\omega$ denote the set of measures quasi-generated (via the shift $\sigma$) by~$\omega$.
By Definition~\ref{deta}, the fact that $\omega$ is deterministic means that $h(\mu)=0$ for all $\mu\in\M_\omega$. For any \im\ $\mu$ on $\Lambda^\N$, by the Kolmogorov--Sinai Theorem (\cite{S}) we have $h(\mu) = h_\mu(\P,T)=\lim_{n\to\infty}\frac1n H_\mu(\P^n)$, where $\P$ is the (generating) partition into cylinders corresponding to blocks of length 1, $\P=\{[a]:a\in\Lambda\}$, where 
$[a]=\{(a_n)_{n\in\N}\in\Lambda^\N: a_1=a\}$.
Because the cylinders $[a]$ are clopen subsets of $\Lambda^\N$, the functions $\mu\mapsto\frac1n H_\mu(\P^n)$, $n\in\N$, are continuous on Borel probability measures and for any \im\ $\mu$ the \sq\ $\frac1n H_\mu(\P^n)$ is nonincreasing (see, e.g.,~\cite[Fact~2.3.1]{Do}), and converges to $0$ on $\M_\omega$. Since $\M_\omega$ is compact (in the topology of the weak* convergence), the convergence is uniform. This implies that for any $\eps>0$, for a large enough $n$ we have $\frac1n H_\mu(\P^n)<\frac{\eps^2}2$ for all $\mu\in\M_\omega$. Recall that $\M_\omega$ coincides with the set of accumulation points of the \sq\ of measures $A_m(\omega)=\frac1m\sum_{i=1}^m\delta_{\sigma^i\omega}$ (see~\eqref{wsl}). 
Observe that since the atoms of the partition $\P^n$ are clopen, the function $\nu\mapsto H_\nu(\P^n)$ (see Section~\ref{S3}, formula (1)) is continuous on $\M(\Lambda^\N)$. As a consequence, we get that
$$
\tfrac1n H_{A_m(\omega)}(\P^n)<\eps^2,
$$
for all sufficiently large $m$.
On the other hand, it is elementary to see that 
$$
\tfrac1n H_{A_m(\omega)}(\P^n)=H_n(\omega|_{[1,m+n-1]}).
$$
We conclude that  
\begin{equation}\label{hh1}
 H_n(\omega|_{[1,M]})<\eps^2   
\end{equation}
for all large enough $M$. Observe that if $m$ is large and $M>m$ then for any $C\in\Lambda^n$ we have 
$$
\Bigl|\mu_{\omega|_{[1,M]}}(C) \ - \ \frac1{M-m+1}\sum_{i\in[1,M-m+1]}\mu_{\omega|_{[i,i+m-1]}}(C)\Bigr|<\delta_0,
$$
where $\delta_0>0$ does not depend on $C$ and, by choosing $M$ large enough, can be made arbitrarily small. Since the entropy function 
$$
P=\{p_1,p_2,\dots,p_k\}\mapsto H(P)=-\sum_{i=1}^k p_i\log(p_i)
$$ 
on the compact convex set of probability vectors of a fixed dimension $k\ge2$ is continuous and concave (see, e.g.,~\cite[Fact 1.1.3]{Do}), for large enough $M$, we have
$$
H_n(\mu_{\omega|_{[1,M]}})\ge \frac1{M-m+1}\sum_{i\in[1,M-m+1]}H_n(\mu_{\omega|_{[i,i+m-1]}}) -\delta_1,
$$
where $\delta_1>0$ is again arbitrarily small. Choosing $\delta_1<\eps^2-H_n(\omega|_{[1,M]})$, by~\eqref{hh1} we get 
$$
\frac1{M-m+1}\sum_{i\in[1,M-m+1]}H_n(\mu_{\omega|_{[i,i+m-1]}})<\eps^2.
$$
This implies that the number of $i\in[1,M-m+1]$ such that $H_n(\omega|_{[i,i+m-1]})\ge\eps$ does not exceed $\eps(M-m+1)$. Letting $M$ tend to infinity, we obtain that the set 
$$
\mathbb S=\{i\in\N:H_n(\omega|_{[i,i+m-1]})\ge\eps\}
$$
has upper density less than $\eps$. Let $F=\{B\in\Lambda^m:H_n(B)<\eps\}$.
Then for any $i\in\N$ we have either $i\in\mathbb S$ or $\omega|_{[i,i+m-1]}\in F$. By Theorem~\ref{cl}, if $m$ is large enough then $|F|<2^{2m\eps}$ and by Remark~\ref{inco}, we have $C_\omega(\eps,m)\le 2^{2m\eps}$. Thus, according to  Definition~\ref{detb}, $\omega$ has subexponential epsilon-complexity.
\smallskip

Now suppose that $\omega$ has subexponential epsilon-complexity. Choose $\mu\in\M_\omega$. There exists a \sq\ $\mathcal J=(n_k)_{k\ge1}$ along which $\omega$ generates $\mu$. Recall that then the \sq\ of blocks $\omega|_{[1,n_k]}$ generates $\mu$ (see Definition~\ref{wss} and the discussion that follows it). By continuity of the entropy function $P\mapsto H(P)$ on the probability vectors, we find that for each $m\ge1$, we have
$$
\lim _{k\to\infty}H_m(\omega|_{[1,n_k]})=\tfrac1m H_\mu(\P^m).
$$

Fix an $\eps>0$. By Definition~\ref{complexity}, there exists $m$ such that all blocks of length $m$ appearing in $\omega$ can be divided into two classes: class~1 of cardinality less than $2^{\eps m}$ and class~2 such that the blocks from class~2 appear in $\omega$ with joint frequency (see Definition~\ref{frequency}~(b)) less than $\eps$. Then, 
for $k$ large enough, the joint frequency of the blocks from class 2
in the block $C=\omega|_{[1,n_k]}$ equals some $\zeta<\eps$. Thus, we can write
\begin{multline*}
H_m(C) = -\frac1m\Bigl(\sum_{B\in\text{class 1}}\mu_C(B)\log\mu_B(C)+\sum_{B\in\text{class 2}}\mu_C(B)\log\mu_B(C)\Bigr)=\\
-\frac1m\Bigl((1-\zeta)\sum_{B\in\text{class 1}}\frac{\mu_C(B)}{1-\zeta}\bigl(\log\frac{\mu_C(B)}{1-\zeta}+\log(1-\zeta)\bigr)+\\
\zeta\sum_{B\in\text{class 2}}\frac{\mu_C(B)}{\zeta}\bigl(\log\frac{\mu_C(B)}{\zeta} +
\log\zeta\bigr)\Bigr)=\\
-\frac1m\Bigl((1-\zeta)\log(1-\zeta)+(1-\zeta)\sum_{B\in\text{class 1}}\frac{\mu_C(B)}{1-\zeta}\log\frac{\mu_C(B)}{1-\zeta}+\\
\zeta\log\zeta+\zeta\sum_{B\in\text{class 2}}\frac{\mu_C(B)}{\zeta}\log\frac{\mu_C(B)}{\zeta}\Bigr)=\\
\frac1m\Bigl(H(\zeta,1-\zeta) + (1-\zeta)H(\text{class 1})+ \zeta H(\text{class 2})\Bigr),
\end{multline*}
where $H(\zeta,1-\zeta)=-(1-\zeta)\log(1-\zeta)-\zeta\log\zeta$, $H(\text{class 1})$ is the entropy of the probability vector $\{\frac{\mu_C(B)}{1-\zeta}:B\in\text{class 1}\}$, and $H(\text{class 2})$ is defined analogously. Clearly, $H(\text{class 1})\le\log|\text{class 1}|\le m\eps$ and $H(\text{class 2})\le\log|\text{class 2}|\le|m\log|\Lambda|$, which implies that
$$
H_m(C)\le \tfrac1m H(\zeta,1-\zeta)+\eps+\eps\log|\Lambda|.
$$
Further, we have
$$
\tfrac1m h_\mu(\P^m)=\lim_kH_m(\omega|_{[1,n_k]})\le\tfrac1m H(\zeta,1-\zeta)+\eps+\eps\log|\Lambda|,
$$
and so, by letting $m$ grow, we obtain
$$
h(\mu)=h_\mu(\P)=\lim_{m\to\infty}\tfrac1m h_\mu(\P^m)\le\eps+\eps\log|\Lambda|.
$$
Since $\eps$ is arbitrarily small, we have shown that $h(\mu)=0$ and hence $\omega$ is deterministic.
\end{proof}

\end{document}